\documentclass[11pt,leqno]{article}
\usepackage{amsmath, amscd, amsthm, amssymb, graphics, xypic, mathrsfs,setspace,fancyhdr}
\usepackage[pagebackref=true]{hyperref}
\hypersetup{backref}

\newcommand{\setof}[1]{\{ #1 \}}
\newcommand{\tensor}{\otimes}

\newcommand{\Spec}{\operatorname{Spec}}
\newcommand{\isomto}{{\stackrel{\sim}{\;\longrightarrow\;}}}
\newcommand{\isomt}{{\stackrel{{\scriptscriptstyle{\sim}}}{\;\rightarrow\;}}}
\newcommand{\smallsim}{{\scriptscriptstyle{\sim}}}

\renewcommand{\O}{{\mathcal O}}
\newcommand{\real}{{\mathbb R}}
\newcommand{\cplx}{{\mathbb C}}
\newcommand{\Q}{{\mathbb Q}}
\newcommand{\Z}{{\mathbb Z}}
\newcommand{\aone}{{\mathbb A}^1}
\newcommand{\pone}{{\mathbb P}^1}
\newcommand{\Stilde}{{\tilde{{\mathcal S}}}}

\newcommand{\ga}{{{\mathbb G}_{a}}}
\newcommand{\gm}{{{\mathbb G}_{m}}}
\newcommand{\et}{\text{\'et}}
\newcommand{\ho}[1]{{\mathcal H}({#1})}
\newcommand{\hop}[1]{{\mathcal H}_{\bullet}({#1})}

\newcommand{\Shv}{{\mathcal Shv}}
\newcommand{\Sm}{{\mathcal Sm}}
\newcommand{\Spc}{{\mathcal Spc}}
\newcommand{\K}{{\underline{\mathbf K}}}

\newcommand{\hsnis}{{\mathcal H}_s^{Nis}(k)}
\newcommand{\hspnis}{{\mathcal H}_{s,\bullet}^{Nis}(k)}
\newcommand{\hspet}{{\mathcal H}_{s,\bullet}^{\et}(k)}
\newcommand{\hsptau}{{\mathcal H}_{s,\bullet}^{\tau}(k)}
\newcommand{\hset}{{\mathcal H}_s^{\et}(k)}
\newcommand{\hstau}{{\mathcal H}_s^{\tau}(k)}
\newcommand{\het}[1]{{\mathcal H}^{\et}(#1)}
\newcommand{\hpet}[1]{{\mathcal H}^{\et}_{\bullet}(#1)}

\newcommand{\F}{{\mathcal F}}
\newcommand{\fkset}{{\mathcal F}_k-{\mathcal Set}}
\newcommand{\fkrset}{{\mathcal F}_k^r-{\mathcal Set}}

\newcommand{\simpnis}{{\Delta}^{\circ}\Shv_{Nis}({\mathcal Sm}_k)}
\newcommand{\simpet}{{\Delta}^{\circ}\Shv_{\text{\'et}}({\mathcal Sm}_k)}

\newcounter{intro}
\theoremstyle{plain}
\newtheorem{thm}{Theorem}[subsection]

\newtheorem{lem}[thm]{Lemma}
\newtheorem{cor}[thm]{Corollary}
\newtheorem{prop}[thm]{Proposition}

\newtheorem{problem}[thm]{Problem}

\newtheorem{conj}[thm]{Conjecture}
\newtheorem*{thm*}{Theorem}
\newtheorem*{problem*}{Problem}

\newtheorem{thmintro}{Theorem}
\newtheorem{problemintro}[thmintro]{Problem}
\newtheorem{defnintro}[thmintro]{Definition}
\newtheorem{scholiumintro}[thmintro]{Scholium}

\theoremstyle{definition}
\newtheorem{defn}[thm]{Definition}

\newtheorem{notation}[thm]{Notation}

\theoremstyle{remark}
\newtheorem{rem}[thm]{Remark}
\newtheorem{remintro}[thmintro]{Remark}
\newtheorem{extension}[thm]{Extension}
\newtheorem{ex}[thm]{Example}

\numberwithin{equation}{subsection}

\begin{document}
\pagestyle{fancy}
\renewcommand{\sectionmark}[1]{\markright{\thesection\ #1}}
\fancyhead{}
\fancyhead[LO,R]{\bfseries\footnotesize\thepage}
\fancyhead[LE]{\bfseries\footnotesize\rightmark}
\fancyhead[RO]{\bfseries\footnotesize\rightmark}
\chead[]{}
\cfoot[]{}
\setlength{\headheight}{1cm}

\author{\begin{small}Aravind Asok\thanks{Aravind Asok was partially supported by National Science Foundation Award DMS-0900813.}\end{small} \\ \begin{footnotesize}Department of Mathematics\end{footnotesize} \\ \begin{footnotesize}University of Southern California\end{footnotesize} \\ \begin{footnotesize}Los Angeles, CA 90089-2532 \end{footnotesize} \\ \begin{footnotesize}\url{asok@usc.edu}\end{footnotesize}\\
\and \begin{small}Fabien Morel\end{small} \\ \begin{footnotesize}Mathematisches Institut\end{footnotesize} \\ \begin{footnotesize}Ludwig Maximilians Universit\"at\end{footnotesize} \\ \begin{footnotesize}Theresienstrasse 39, D-80 333 M\"unchen\end{footnotesize} \\
\begin{footnotesize}\url{morel@mathematik.uni-muenchen.de}\end{footnotesize}}

\title{\vskip 20pt {\bf {Smooth varieties up to $\aone$-homotopy \\ and algebraic $h$-cobordisms}}\vskip 20pt}
\maketitle

\begin{abstract}
We start to study the problem of classifying smooth proper varieties over a field $k$ from the standpoint of $\aone$-homotopy theory.  Motivated by the topological theory of surgery, we discuss the problem of classifying up to isomorphism all smooth proper varieties having a specified $\aone$-homotopy type.  Arithmetic considerations involving the sheaf of $\aone$-connected components lead us to introduce several different notions of connectedness in $\aone$-homotopy theory.  We provide concrete links between these notions, connectedness of points by chains of affine lines, and various rationality properties of algebraic varieties (e.g., rational connectedness).

We introduce the notion of an $\aone$-$h$-cobordism, an algebro-geometric analog of the topological notion of $h$-cobordism, and use it as a tool to produce non-trivial $\aone$-weak equivalences of smooth proper varieties.  Also, we give explicit computations of refined $\aone$-homotopy invariants, such as the $\aone$-fundamental sheaf of groups, for some $\aone$-connected varieties.  We observe that the $\aone$-fundamental sheaf of groups plays a central yet mysterious role in the structure of $\aone$-$h$-cobordisms.  As a consequence of these observations, we completely solve the classification problem for rational smooth proper surfaces over an algebraically closed field: while there exist arbitrary dimensional moduli of such surfaces, there are only countably many $\aone$-homotopy types, each uniquely determined by the isomorphism class of its $\aone$-fundamental sheaf of groups.
\end{abstract}
\newpage

\begin{footnotesize}
\tableofcontents
\end{footnotesize}

\section{Classification problems in algebraic geometry}
\label{s:introduction}
In this paper, drawing its inspiration from geometric topology, we investigate the problem of classifying smooth proper algebraic varieties over a field using the techniques of $\aone$-homotopy theory.  In geometric topology, one can, without loss of generality, restrict the classification problem by considering {\em connected} manifolds, and in this setting classification can be performed most effectively for highly connected spaces ({\em cf.} \cite{Wallhighlyconnected}).  Similarly, we restrict our algebro-geometric classification problem by imposing appropriate connectedness hypotheses; these restrictions are {\em highly} non-trivial.  Indeed, the problem of even defining an appropriate analog of connectedness is subtle, especially as we require the notion to have a close relationship with $\aone$-homotopy theory.   We consider notions called $\aone$-connectedness, \'etale $\aone$-connectedness, and weak $\aone$-connectedness; each notion is motivated by homotopic and arithmetic considerations.  Contrary to the situation in geometric topology, we will see that one cannot often impose ``higher $\aone$-connectedness" hypotheses for such varieties because strictly positive dimensional smooth proper $\aone$-connected varieties {\em always} have non-trivial $\aone$-fundamental group (see Propositions \ref{prop:nontrivialfundamentalgroup} and \ref{prop:etaleaonefundamentalgroupnontrivial}).

Next, we link $\aone$-connectedness to the birational geometry of algebraic varieties.  Over perfect fields $k$, we show separably rationally connected smooth proper varieties introduced by Koll\'ar-Miyaoka-Mori are weakly $\aone$ connected (see Definition \ref{defn:aoneconnected} and Corollary \ref{cor:separablyrationallyconnected}), and, if $k$ furthermore has characteristic $0$, retract $k$-rational smooth proper varieties are $\aone$-connected (see Theorem \ref{thm:stablyrational}).  Over any field, Theorem \ref{thm:propercharacterization} provides a geometric characterization of $\aone$-connectedness for smooth proper varieties, and its subsequent corollaries provide connections with Manin's notion of $R$-equivalence.  These results provide new {\em homotopic} techniques, and produce a host of new invariants, to study ``nearly rational" varieties (see Propositions \ref{prop:brauergroup} and \ref{prop:strongfactor}).

Finally, we outline a general program for studying the classification problem and provide supporting evidence.  We give a detailed study of the classification in low dimensions (see Theorems \ref{thm:surfacesversion1} and \ref{thm:refinedsurfaces}).  We emphasize the r\^ole played by the $\aone$-fundamental (sheaf of) group(s) of a smooth proper variety and provide a number of detailed computations (see Propositions \ref{prop:hirzebruchfundamental} and \ref{prop:fundamentalgroupblowup}).  We reformulate these low dimensional results as a low-dimensional solution to the $\aone$-surgery problem (\ref{problem:aonesurgery}), which suggests that the isomorphism classes of smooth varieties having a given $\aone$-homotopy type have additional structure.

\subsubsection*{Connectivity restrictions in $\aone$-homotopy: geometry vs. arithmetic}
Throughout this paper, the word {\em manifold} will mean compact without boundary smooth manifold.  Classically, topologists considered two fundamental classification problems: (i) classify $n$-dimensional manifolds up to diffeomorphism, and (ii) classify $n$-dimensional manifolds up to homotopy equivalence.  Problem (i) refines Problem (ii), and the study of both problems breaks down along dimensional lines.  Extremely explicit classifications exist in dimensions $1$ or $2$, and Thurston's geometrization program provides a classification in dimension $3$. While group theoretic decision problems prevent algorithmic solutions to either problem in dimension $\geq 4$, the celebrated Browder-Novikov-Sullivan-Wall theory of surgery ({\em cf.} \cite[\S1]{Ranicki}) shows that, in dimensions $\geq 5$, the problem of identifying diffeomorphism classes of manifolds having a fixed homotopy type can be effectively reduced to computations in homotopy theory.

The problem of classifying smooth algebraic varieties over a field $k$ up to isomorphism is formally analogous to Problem (i).  To state an analog of Problem (ii), one must choose an appropriate notion of ``homotopy equivalence" and consider the corresponding homotopy category.   We use the $\aone$- (or motivic) homotopy theory developed by the second author and V. Voevodsky in \cite{MV}.  The prefix $\aone$- draws attention to the fact that the affine line in $\aone$-homotopy theory plays the same r\^ole as the unit interval in ordinary homotopy theory.  The resulting {\em $\aone$-homotopy category of smooth varieties over $k$} is denoted $\ho{k}$, and isomorphisms in $\ho{k}$ are called {\em $\aone$-weak equivalences}.  One source of examples of $\aone$-weak equivalences is obtained by formally replacing the unit interval by the affine line in the classical definition of a homotopy equivalence.  Another source of examples is provided by the \u Cech simplicial resolution associated with a Nisnevich covering of a smooth scheme (see Example \ref{ex:cechobjects} for more details).  One of the difficulties of the theory stems from the fact that general $\aone$-weak equivalences are obtained from these two classes by a complicated formal procedure (see, e.g., \cite[Proposition 8.1]{Dugger} for development of this point of view).

{\em Compact and without boundary manifolds} are akin to algebraic varieties {\em smooth and proper} over a field $k$.  Collecting these observations, we suggest a natural analog of Problem (ii) in the context of algebraic varieties over a field $k$.

\begin{problemintro}
\label{problem:classification}
Classify smooth proper $k$-varieties up to $\aone$-weak equivalence.
\end{problemintro}

In order that two manifolds be homotopy equivalent, they must at least have the same number of connected components.  One usually studies connected manifolds before investigating the disconnected case.  A manifold $M$ is path (or chain) connected if every pair of points lies in the image of a (chain of) map(s) from the unit interval, or equivalently if the set $\pi_0(M)$ has exactly one element.  In $\aone$-homotopy theory one attaches to a smooth variety $X$ a (Nisnevich) {\em sheaf} $\pi_0^{\aone}(X)$ of $\aone$-connected components (see Definition \ref{defn:sheavesofaoneconnectedcomponents}); \S \ref{s:connectedness} contains an elaborate analysis of this and closely related definitions.  A variety is {\em $\aone$-connected} if it has the same sheaf of connected components as a point, i.e., the spectrum of the base field, and $\aone$-disconnected otherwise.  For varieties that are $\aone$-disconnected, the classification problem can look drastically different than it does for varieties that are (close to) $\aone$-connected (see, e.g., Proposition \ref{prop:curves}).  For arithmetic reasons, we introduce a variant of $\aone$-connectedness that we call \'etale $\aone$-connectedness.

Transposing topological intuition, one might na\"ively imagine that a smooth variety $X$ is $\aone$-connected if any pair of $k$-points is contained in the image of a morphism from the affine line.  We now discuss some precise results underlying this intuition.  We introduce two different notions of ``path" connectivity, $\aone$-chain connectedness and \'etale $\aone$-chain connectedness, depending on whether one requires such path connectedness properties for all separable field extensions of $k$ or just separably closed field extensions of $k$ (see Definition \ref{defn:chainconnected} for more precise statements).  The following results connect the geometric ideas just mentioned to the homotopy theoretic definition alluded to in the previous paragraph.  In particular, we provide a geometric characterization of $\aone$-connectedness for smooth proper schemes over a field.

\begin{thmintro}[see Theorem \ref{thm:propercharacterization} and Corollary \ref{cor:chainconverse}]
If $X$ is a smooth proper scheme over a field $k$, $X$ is $\aone$-connected if and only if it is $\aone$-chain connected (see \textup{Definition \ref{defn:chainconnected}}).
\end{thmintro}

Section \ref{s:unramified} is devoted to developing techniques for proving a more general result that implies this one.  The following result provides a link between important notions of birational geometry and the aforementioned connectivity properties.

\begin{thmintro}[see Theorem \ref{thm:stablyrational} and Corollary \ref{cor:separablyrationallyconnected}]
Suppose $k$ is a perfect field.
\begin{itemize}
\item If $X$ is a separably rationally connected smooth proper variety, then $X$ is weakly $\aone$-connected.
\item If furthermore $k$ has characteristic $0$, and $X$ is a retract $k$-rational variety, then $X$ is $\aone$-connected.
\end{itemize}
A $k$-variety $X$ is retract $k$-rational if, e.g., it is $k$-rational, stably $k$-rational, or factor $k$-rational (see \textup{Definition \ref{defn:rationalitynotions}}).
\end{thmintro}

We feel the suggested parallels with geometric topology and the direct links with birational geometry justify the importance of these connectivity notions and our subsequent focus on them.  We refer the reader to Appendix \ref{s:notationalpostscript} for a convenient summary of the relationships between the various notions just mentioned.

One can study disconnected manifolds by separate analysis of each connected component (however, {\em cf.} \cite[pp. 34-35]{Wall} for discussion in the context of classification).  In contrast, in $\aone$-homotopy theory, for any smooth variety $X$ there is a canonical epimorphism $X \to \pi_0^{\aone}(X)$ that is in general highly non-trivial.  While separably rationally connected smooth proper varieties are weakly $\aone$-connected, they need not be $\aone$-connected in general.  Indeed, the difference between the $\aone$-connectedness and weak $\aone$-connectedness encodes subtle arithmetic information; see Example \ref{ex:conicbundles} for a particularly geometric manifestation of this phenomenon.  Expanding on this, we explore cohomological aspects of smooth $\aone$-connected varieties in \S \ref{s:classifyingspaces}.

\subsubsection*{Classification Part I: low dimensional results}
For manifolds having dimension $1$ or $2$ the homotopy classification and the diffeomorphism classification coincide.  Each connected component of a $1$-dimensional manifold is diffeomorphic to the circle $S^1$.  Each connected component of a $2$-dimensional manifold is homotopy equivalent to either $S^2$, the connected sum of $g$ copies of $S^1 \times S^1$ for some integer $g \geq 1$, or the connected sum of $g'$ copies of ${\mathbb{RP}^2}$ for some integer $g' \geq 1$.

Suppose $k$ is an algebraically closed field.  There is a unique up to $\aone$-weak equivalence smooth proper $\aone$-connected $k$-variety of dimension $1$, namely $\pone$ (see Proposition \ref{prop:curves}).  Thus, just as for manifolds, the isomorphism and $\aone$-homotopy classifications of $\aone$-connected smooth proper varieties coincide in dimension $1$.

In dimension $2$, however, we see that the $\aone$-homotopy and isomorphism classifications of $\aone$-connected varieties do {\em not} coincide.  Indeed, smooth proper $k$-rational surfaces are $\aone$-connected, and we expect the converse to be true ({\em cf.} Conjecture \ref{conj:epimorphism}).  In fact, Corollary \ref{cor:aoneconnectednesssurfaces} shows that $k$-rational smooth proper surfaces are {\em exactly} the smooth proper $\aone$-connected surfaces for fields having characteristic $0$.  The $\aone$-homotopy classification of rational smooth proper surfaces, which addresses Problem \ref{problem:classification} in dimension $\leq 2$ and is strikingly similar to its topological counterpart, can be stated in most elementary terms as follows.

\begin{thmintro}[See Theorem \ref{thm:surfaces}]
\label{thm:surfacesversion1}
Any rational smooth proper surface over an algebraically closed field $k$ is $\aone$-weakly equivalent to either $\pone \times \pone$, or a blow-up of some (possibly empty) fixed, finite collection of distinct $k$-points on ${\mathbb P}^2$.
\end{thmintro}

The isomorphism classification of rational smooth proper surfaces is well known: over an algebraically closed field, any such surface is (non-uniquely) isomorphic to an iterated blow-up of points of either ${\mathbb P}^2$, or a Hirzebruch surface ${\mathbb F}_{a} = {\mathbb P}(\O_{\pone} \oplus \O_{\pone}(a))$.  In particular, there can be moduli of isomorphism classes of such varieties, but the $\aone$-homotopy classes are parameterized by a discrete set.  Theorem \ref{thm:surfacesversion1} leads us to search for an explanation of the discrepancy between the isomorphism and $\aone$-homotopy classifications.

\subsubsection*{High dimensions and internal structure of homotopy types via surgery}
Just as the topological fundamental group obstructs enumeration of homotopy types of manifolds, the $\aone$-fundamental (sheaf of) group(s) likely obstructs enumeration of $\aone$-homotopy types of smooth varieties, even if they are $\aone$-connected.  Thus, Problem \ref{problem:classification} is probably insoluble in dimension $\geq 4$, so we seek a further refinement.  Surgery theory asserts, roughly speaking, that one can measure the extent to which the homotopy and diffeomorphism classifications for manifolds of a given dimension differ.  Indeed, one main application of surgery theory is to the solution (in dimensions $\geq 5$) of Problem (iii): determine the set of diffeomorphism classes of manifolds in a given homotopy type.  As Kervaire and Milnor explain in their celebrated work on classification of exotic spheres (see \cite[Theorem 1.1]{KervaireMilnor} and the subsequent remark), arguably still the best introduction to surgery theory, the analysis of Problem (iii) consists of two independent components.

The first component of the investigation is provided by Smale's theory of $h$-cobordisms that relates diffeomorphism to more homotopy theoretic notions (e.g., cobordism theory) via Morse theory.  An {\em $h$-cobordism} $(W,M,M')$ between smooth manifolds $M$ and $M'$ is a cobordism such that the inclusions $M \hookrightarrow W$ and $M' \hookrightarrow W$ are homotopy equivalences.  The fundamental group of $M$ plays a central r\^ole in the theory of $h$-cobordisms.  Smale's celebrated $h$-cobordism theorem showed that an $h$-cobordism between simply connected manifolds of dimension $\geq 5$ is necessarily {\em trivial}, i.e., diffeomorphic to a product of the form $M \times I$.  An $h$-cobordism between non-simply connected manifolds of dimension $\geq 5$ need not be diffeomorphic to a product, but Barden, Mazur and Stallings formulated an appropriate generalization, now called the {\em $s$-cobordism theorem}, giving necessary and sufficient conditions for triviality of such $h$-cobordisms (see, e.g., \cite[Chapter 8]{Ranicki} for discussion of these results in the context of classification).

Given a CW complex $X$, let ${\mathscr S}(X)$ denote the {\em structure set} of $X$, i.e., the set of $h$-cobordism classes of manifolds homotopy equivalent to $X$.  At least in dimensions $\geq 5$, Problem (iii) reduces via the $h$-cobordism (or $s$-cobordism) theorem to determining if the structure set ${\mathscr S}(X)$ is non-empty, and, when it is, providing a description of its elements.  The second component of the analysis of Problem (iii), contained in the beautiful work of Browder, Novikov, Sullivan and Wall, provides a description of ${\mathscr S}(X)$.  In its modern formulation, these authors construct a two-stage obstruction theory fitting ${\mathscr S}(X)$ into an appropriate exact sequence and effectively reducing determination of ${\mathscr S}(X)$ to computations in homotopy theory; we outline this approach at the end of \S \ref{s:fundamentalgroup}.  The literature on surgery theory is famously dense and the standard references are \cite{Browder} and \cite{Wall}; other references, each having different emphasis, include \cite{Ranicki,MadsenMilgram,Rognes}.

Mimicking this development in algebraic geometry, we see the strengths of the $\aone$-homotopy category.  M. Levine and the second author developed a natural algebro-geometric analog of cobordism theorem, called {\em algebraic cobordism} (see \cite{LM}).  In this spirit, we introduce a notion of $h$-cobordism in algebraic geometry (see Definition \ref{defn:aonehcobordism}), motivated by Morse theory, that we call {\em $\aone$-$h$-cobordism}.  An $\aone$-$h$-cobordism between smooth proper varieties $X$ and $X'$ consists of a pair $(W,f)$, with $W$ a smooth variety, and $f: W \to \aone$ a proper, surjective morphism such that $X \cong f^{-1}(0)$, $X' \cong f^{-1}(1)$, and the inclusions $X \hookrightarrow W$ and $X' \hookrightarrow W$ are $\aone$-weak equivalences.  We think of the pair $(W,f)$ as a cobordism together with a choice of ``Morse function."  In analogy with the situation in geometric topology, we suggest the following definition and problem.

\begin{defnintro}
Given a (fibrant, $\aone$-local) space ${\mathcal X}$ (see the conventions), a {\em scheme structure} on (or {\em homotopy smoothing of}) ${\mathcal X}$ is a pair $(X,s)$, consisting of a smooth proper scheme $X$ and an $\aone$-weak equivalence $s: X \to {\mathcal X}$.  The {\em $\aone$-structure set} of ${\mathcal X}$, denoted ${\mathscr S}_{\aone}({\mathcal X})$, is the set of scheme structures on ${\mathcal X}$ subject to the equivalence relation generated by $(X,s) \sim (X',s')$ if there exists a triple $(W,f,H)$ consisting of an $\aone$-$h$-cobordism $(W,f)$ between $X$ and $X'$, and a morphism $H: W \to {\mathcal X}$ that upon composition with the morphism $X \hookrightarrow W$ (resp. $X' \hookrightarrow W'$) gives $f$ (resp. $f'$).   Two scheme structures on a space ${\mathcal X}$ equivalent in ${\mathscr S}_{\aone}({\mathcal X})$ will be called {\em $\aone$-block equivalent}.
\end{defnintro}

\begin{problemintro}[$\aone$-surgery problem]
\label{problem:aonesurgery}
Given an $\aone$-connected space ${\mathcal X}$, determine if the set ${\mathscr S}_{\aone}({\mathcal X})$ is non-empty, and, if it is, determine the set of $\aone$-block equivalence classes.
\end{problemintro}

To show that Problem \ref{problem:aonesurgery} is extremely rich, we give techniques for constructing $\aone$-$h$-cobordisms of smooth proper schemes (see Propositions \ref{prop:bundles} and \ref{prop:blowups}).  The proof of Theorem \ref{thm:surfacesversion1} is obtained by detailed study of $\aone$-$h$-cobordisms between rational smooth proper surfaces.  Extending this discussion, the results of \S \ref{s:classifyingspaces} show that $\aone$-$h$-cobordisms constructed by these means are abundant and lays some technical foundation for a general investigation.

Reflecting on some basic computations highlights central differences between topology and algebraic geometry.  First, $\aone$-$h$-cobordisms between smooth schemes are ``rarely" trivial (i.e., isomorphic to products of the form $X \times \aone$), even when the associated ``Morse function" has no critical values.  Following topological ideas, we look to the $\aone$-fundamental group, investigated in great detail in \cite{MField}, for an explanation of the deviation between $\aone$-block equivalence and isomorphism as schemes.  Second, there are arguably few interesting invariants available to distinguish ``nearly rational" varieties.  For example, given a unirational smooth proper complex variety $X$, Serre famously showed \cite[Proposition 1]{Serre} that the set $X(\cplx)$, viewed as a complex manifold, is simply connected.  In stark contrast to the topological situation, the $\aone$-fundamental group of any strictly positive dimensional rational smooth proper complex variety is {\em always} a highly non-trivial invariant (see Propositions \ref{prop:nontrivialfundamentalgroup} and \ref{prop:etaleaonefundamentalgroupnontrivial} for precise and more general statements)!  Said differently, using Example \ref{ex:aonerigid}, one can see that the only $\aone$-connected and $\aone$-simply connected smooth proper variety over a field is a point.  We record the following result mainly for the sake of amusement; see Proposition \ref{prop:etaleaonehcobordismtheorem} for a precise and more general statement.

\begin{scholiumintro}[$\aone$-$h$-cobordism theorem]
Any $\aone$-$h$-cobordism between smooth proper $\aone$-connected and $\aone$-simply connected varieties over a field is trivial.
\end{scholiumintro}

The $\aone$-fundamental group of $\pone$, determined in \cite{MField}, plays a distinguished r\^ole in $\aone$-homotopy theory, and we review aspects of this computation here.  The $\aone$-fundamental groups of projective spaces, $SL_n$, and smooth proper toric varieties have also been studied (\cite[\S 6]{MField}, and \cite{ADExcision,Wendt}).  One main computational result of this paper is the determination of the $\aone$-fundamental group of various rational smooth proper varieties (see Propositions \ref{prop:hirzebruchfundamental} and \ref{prop:fundamentalgroupblowup}).  Combining the computations of this paper with Theorem \ref{thm:surfacesversion1} provides a solution to Problem \ref{problem:aonesurgery} for rational smooth proper surfaces.

\begin{thmintro}[See Corollary \ref{thm:surfacesfundamentalgroup}]
\label{thm:refinedsurfaces}
Let $k$ be an algebraically closed field.  Two rational smooth proper surfaces are $\aone$-$h$-cobordant if and only if their $\aone$-fundamental groups are isomorphic.  Thus, for any rational smooth proper surface $X$ the set ${\mathscr S}_{\aone}(X)$ consists of a single element.
\end{thmintro}

\begin{remintro}
Combining Corollary \ref{cor:aoneconnectednesssurfaces} with Theorem \ref{thm:refinedsurfaces} provides a solution to the $\aone$-surgery problem for smooth proper $\aone$-connected varieties of dimension $\leq 2$ over algebraically closed fields having characteristic $0$.
\end{remintro}

\begin{remintro}
Theorem \ref{thm:refinedsurfaces} and the discussion preceding it provides the following lesson:  the extent to which the isomorphism and $\aone$-homotopy classifications differ depends on ``minimality properties" of an $\aone$-homotopy type in a sense we will explore in Problem \ref{problem:aoneminimality}.  Indeed, we will see that blowing-up makes the $\aone$-fundamental group more complicated.  At the end of \S \ref{s:fundamentalgroup}, we discuss possible analogs of the $s$-cobordism theorem in $\aone$-homotopy theory and formulate a general approach to the $\aone$-surgery problem.
\end{remintro}

The $\aone$-homotopy type of a smooth proper variety encodes universal cohomological information about the variety, and, in particular, information about Hodge structures on cohomology, \'etale homotopy type, (higher) Chow groups, algebraic K-theory, or Hermitian K-theory.   The eventual goal of this kind of study of smooth proper schemes is to understand the {\em arithmetic building blocks}, or {\em motivic skeleton}, of smooth proper varieties over a field using a surgery-style obstruction theory.  The introduction to each section contains more detailed discussion of the results contained therein.

\subsubsection*{Conventions and notation}
Throughout this paper, $k$ denotes a field.  Henceforth, we use the word {\em scheme} as a synonym for separated scheme having essentially finite type over $k$, i.e., a filtering limit of $k$-schemes having finite type over $k$ with smooth affine bonding morphisms.  The word {\em variety} means integral scheme having finite type over $k$. Using this terminology, let $\Sm_k$ denote the category of smooth schemes having finite type over $k$.  The words map and morphism are used synonymously through this paper, and we denote them by solid arrows.  Rational maps, where they occur, are denoted by dashed arrows.

We let $\Spc_k$ ($\Spc_{k,\bullet}$) stand for the category of {\em (pointed) spaces over $k$}, i.e., the category of (pointed) simplicial {\em Nisnevich} sheaves of sets on $\Sm_k$.  Similarly, $\Spc_k^{\et}$ ($\Spec_{k,\bullet}^{\et}$) will stand for the category of {\em (pointed) \'etale spaces over $k$}, i.e., the category of (pointed) simplicial \'etale sheaves of sets on $\Sm_k$.  The word {\em sheaf} will uniformly mean sheaf in the Nisnevich topology ({\em cf.} \cite[\S 3.1]{MV}), unless otherwise indicated.  We designate schemes by upper case Roman letters (e.g., $X,Y$), spaces by upper case calligraphic letters (e.g., ${\mathcal X},{\mathcal Y}$), and pointed (simplicial) spaces by explicit specification of the base-point.  We write ${\mathcal X}_n$ for the sheaf of $n$-simplices of a space ${\mathcal X}$.  The Yoneda embedding induces a fully-faithful functor ${\mathcal Sm}_k \to \Spc_k$ (resp. $\Spc_k^{\et}$); we systematically abuse notation and write $X$ for the space associated with $X \in {\mathcal Sm}_k$ via this functor.  If $k$ is clear from context, we write $\ast$ interchangeably for the space $\Spec k$ or a base-($k$-)point.  For an extension field $L/k$, we write ${\mathcal X}(L)$ for ${\mathcal X}(\Spec L)$.

Our conventions regarding $\aone$-homotopy theory, together with a quick review of the relevant material, and statements of most of the properties we use are summarized in Appendix \ref{s:homotopyoverview}; see \cite{MV,MIntro} for detailed development of the subject.  We write $\hsnis$ ($\hspnis$) for the {\em (pointed) simplicial homotopy category} and $\hset$ ($\hspet$) for the {\em (pointed) \'etale simplicial homotopy category}.  We write $\ho{k}$ ($\hop{k}$) for the {\em (pointed) $\aone$-homotopy} category, and $\het{k}$ ($\hpet{k}$) for the {\em (pointed) \'etale $\aone$-homotopy category}.

Given two pointed spaces $({\mathcal X},x)$ and $({\mathcal Y},y)$, we write $[{\mathcal X},{\mathcal Y}]_s$ (resp. $[({\mathcal X},x),({\mathcal Y},y)]_s$) for the set of (pointed) simplicial homotopy classes of maps from ${\mathcal X}$ to ${\mathcal Y}$.  Likewise, we write $[{\mathcal X},{\mathcal Y}]_{\aone}$ (resp. $[({\mathcal X},x),({\mathcal Y},y)]_{\aone}$) for the set of (pointed) $\aone$-homotopy classes of maps from ${\mathcal X}$ to ${\mathcal Y}$.  The corresponding objects in the \'etale topology will be decorated with an extra sub or super-script ``$\et$."

Recall also the {\em $\aone$-homotopy groups} of a pointed space $({\mathcal X},x)$, denoted $\pi_i^{\aone}({\mathcal X},x)$, are the sheaves associated with the presheaves $U \mapsto [S^i_s \wedge U_+,({\mathcal X},x)]_{\aone}$; the spheres $S^i_s$ are studied in \cite[\S 3.2.2]{MV}.  Finally, one word of caution is in order: we (almost) never use stable homotopy theoretic considerations in this paper, so a lowercase super- or subscript $s$ is always short for {\em simplicial}.

\subsubsection*{Acknowledgments}
The first author would like to thank Brent Doran for interesting conversations about this and related projects, and Jeff Giansiracusa and Johannes Ebert for many interesting discussions about surgery theory and classification problems.  The second author is very much indebted to his advisor Jean Lannes for his enthusiastic and inspiring explanations of the beautiful theory of surgery in the classical case.  We would like to thank Jean-Louis Colliot-Th\'el\`ene for pointing out a counterexample to a conjecture in an earlier version of this paper, and providing some extremely useful comments and references regarding various notions of ``near rationality" and their intricate interrelations.   We also thank Kei Hagihara for pointing out an error in a previous version of this paper.

\section{Connectedness in $\aone$-homotopy theory}
\label{s:connectedness}
In this section, we discuss several notions of connectedness in $\aone$-homotopy theory, point out some fundamental differences between these notions and the usual notion of (path) connectedness for a topological space, and relate our notions of connectedness to birational geometry of algebraic varieties.  While classification of disconnected manifolds reduces to classification of connected components, the corresponding fact in $\aone$-homotopy theory is more subtle (see Remark \ref{rem:connectedcomponents}, Lemma \ref{lem:aonerigid} and Proposition \ref{prop:curves}).  Definition \ref{defn:chainconnected} and Proposition \ref{prop:aoneconnected} provide a geometric condition that guarantees $\aone$-connectedness of a space.  Theorem \ref{thm:propercharacterization} and Corollary \ref{cor:chainconverse} complement these results by providing a (partial) converse to Proposition \ref{prop:aoneconnected} and give a geometric characterization of (weak) $\aone$-connectedness for smooth proper schemes over a field.

Definition \ref{defn:combinatorial} and Lemma \ref{lem:combinatorialconnected} give large classes of smooth $\aone$-connected varieties.  Theorem \ref{thm:stablyrational} demonstrates, in particular, that retract $k$-rational varieties over a field $k$ having characteristic $0$ are necessarily $\aone$-connected.  While separably rationally connected smooth proper varieties over a perfect field are not $\aone$-connected in general, Theorem \ref{cor:separablyrationallyconnected} shows such varieties are weakly $\aone$-connected.  In particular, separably rationally connected smooth proper $k$-varieties are precisely the weakly $\aone$-connected smooth proper $k$-varieties.  Finally, we note here that Appendix \ref{s:notationalpostscript} provides a summary of the various notions of connectivity and rationality introduced in this section and used in the rest of the paper.

\subsection{$\aone$-connectedness: Definitions}
\begin{defn}
\label{defn:sheavesofsimplicialconnectedcomponents}
Suppose ${\mathcal X} \in \Spc_k$ (resp. $\Spc_k^{\et}$).  The {\em sheaf of (\'etale) simplicial connected components of ${\mathcal X}$}, denoted $\pi_0^{s}({\mathcal X})$ (resp. $\pi_0^{s,\et}({\mathcal X})$), is the (\'etale) sheaf associated with the presheaf $U \mapsto [U,{\mathcal X}]_s$ (resp. $U \mapsto [U,{\mathcal X}]_{s,\et}$) for $U \in \Sm_k$.
\end{defn}

\begin{defn}
\label{defn:sheavesofaoneconnectedcomponents}
Suppose ${\mathcal X} \in \Spc_k$.  The {\em sheaf of $\aone$-connected components of ${\mathcal X}$}, denoted $\pi_0^{\aone}({\mathcal X})$, is the sheaf associated with the presheaf
\[
U \longmapsto [U,{\mathcal X}]_{\aone},
\]
for $U \in {\mathcal Sm}_k$.  Similarly, for ${\mathcal X} \in \Spc^{\et}_k$, the {\em sheaf of \'etale $\aone$-connected components}, denoted $\pi_0^{\aone,\et}({\mathcal X})$, is the \'etale sheaf associated with the presheaf
\[
U \longmapsto [U,{\mathcal X}]_{\aone,\et}
\]
for $U \in {\mathcal Sm}_k$.
\end{defn}

\begin{rem}
Suppose ${\mathcal X} \in \Spc_k$ (resp. $\Spc_k^{\et}$).  If $L_{\aone}({\mathcal X})$ denotes the $\aone$-localization functor (see Definition \ref{defn:aonehomotopycategories}), then one has by definition $\pi_0^{s}(L_{\aone}({\mathcal X})) = \pi_0^{\aone}({\mathcal X})$ (resp. $\pi_0^{s,\et}(L_{\aone}({\mathcal X}) = \pi_0^{\aone,\et}({\mathcal X}))$.
\end{rem}

The final object $\Spec k$ in the category $\Spc_k$ (resp. $\Spc^{\et}_k$) is simplicially fibrant and $\aone$-local (see Example \ref{ex:aonelocalsmoothschemes} for an explanation).  From this fact, one deduces that $\pi_0^{\aone}(\Spec k) = \Spec k$ and that $\pi_0^{\aone,\et}(\Spec k) = \Spec k$.  These two observations allow us to define $\aone$-homotopic notions of connectedness.

\begin{defn}
\label{defn:aoneconnected}
We say that ${\mathcal X} \in \Spc_k$ (resp. $\Spc_k^{\et}$) is (\'etale) {\em $\aone$-connected} if the canonical morphism ${\mathcal X} \to \Spec k$ induces an isomorphism of sheaves $\pi_0^{\aone}({\mathcal X}) \isomt \Spec k$ (resp. isomorphism of \'etale sheaves $\pi_0^{\aone,\et}({\mathcal X}) \isomt \Spec k$).  We say that ${\mathcal X} \in \Spc_k$ is {\em weakly $\aone$-connected} if the map $\pi_0^{\aone}({\mathcal X}) \to \Spec k$ is an isomorphism on sections over separably closed extensions $L/k$.  Similarly, we say that ${\mathcal X}$ in $\Spc_k$ (resp. $\Spc_k^{\et}$) is {\em (\'etale) $\aone$-disconnected} if it is not (\'etale) $\aone$-connected.
\end{defn}

We will see later that, even for varieties over $\cplx$, $\aone$-connectedness and weak $\aone$-connectedness are distinct notions.  From the existence of a fibrant resolution functor (see the discussion just prior to Definition \ref{defn:aonehomotopycategories}), one can deduce the following result.

\begin{cor}[Unstable $\aone$-$0$-connectivity theorem {\cite[\S 2 Corollary 3.22]{MV}}]
\label{cor:unstableaoneconnectivitytheorem}
Suppose ${\mathcal X} \in \Spc_k$ (resp. $\Spc_k^{\et}$).  The canonical map ${\mathcal X} \to Ex_{\aone}({\mathcal X})$ (resp. ${\mathcal X} \to Ex_{\et,\aone}({\mathcal X})$) induces an epimorphism $\pi_0^{s}({\mathcal X}) \to \pi_0^{\aone}({\mathcal X})$ (resp. $\pi_0^{s,\et}({\mathcal X}) \to \pi_0^{\aone,\et}({\mathcal X})$).
\end{cor}

\begin{ex}
 Suppose $X$ is a smooth $\aone$-connected $k$-scheme. Since $\Spec k$ is Henselian local, the map $X(\Spec k) \to \pi_0^{\aone}(X)(\Spec k)$ is surjective, and we conclude that $X$ necessarily has a $k$-rational point.  The corresponding statement for smooth \'etale $\aone$-connected schemes is false, i.e., smooth \'etale $\aone$-connected $k$-schemes need not have a $k$-rational point if $k$ is not separably closed.
\end{ex}

\begin{rem}
\label{rem:connectedcomponents}
In topology, if $M$ is a manifold (recall the conventions of \S \ref{s:introduction}), then we can study $M$ by analyzing each connected component separately, since each such component will again be a manifold.  In $\aone$-homotopy theory, given a morphism $* \to \pi_0^{\aone}(X)$ corresponding to a $k$-rational point $x \in X(k)$, the ``($\aone$-)connected component of $X$ containing the point $x$" is a space that need not be a smooth scheme.
\end{rem}

\subsubsection*{$\aone$-rigidity}
\begin{defn}
\label{defn:aonerigid}
A scheme $X \in {\mathcal Sm}_k$ is called {\em $\aone$-rigid} (see \cite[\S 3 Example 2.4]{MV}) if for every $U \in {\mathcal Sm}_k$, the map
\[
X(U) \longrightarrow X(U \times \aone)
\]
induced by pullback along the projection $U \times \aone \to U$ is a bijection.
\end{defn}

\begin{lem}
\label{lem:aonerigid}
If $X \in {\mathcal Sm}_k$ is $\aone$-rigid, then for any $U \in \Sm_k$ the canonical maps
\[
\begin{split}
X(U) &\longrightarrow [U,X]_{\aone}, \text{ and } \\
X(U) &\longrightarrow [U,X]_{\aone,\et}
\end{split}
\]
are bijections.  Consequently, the canonical map $X \rightarrow \pi_0^{\aone}(X)$ (resp. $X \rightarrow \pi_0^{\aone,\et}(X)$) is an isomorphism of (\'etale) sheaves.
\end{lem}

\begin{proof}
This follows immediately from Example \ref{ex:aonelocalsmoothschemes} because $\aone$-rigid smooth schemes are precisely the $\aone$-local smooth schemes.
\end{proof}

\begin{ex}
\label{ex:aonerigid}
Any $0$-dimensional scheme over a field $k$ is $\aone$-rigid.  Abelian $k$-varieties are $\aone$-rigid, and smooth complex varieties that can be realized as quotients of bounded Hermitian symmetric domains by actions of discrete groups are also $\aone$-rigid.  From the above collection of $\aone$-rigid varieties, one can produce new examples by taking (smooth) subvarieties or taking products.  The unstable $\aone$-homotopy types of $\aone$-rigid smooth schemes are, in a sense, uninteresting: all higher $\aone$-homotopic invariants of such varieties are trivial.  Note however the stable $\aone$-homotopy types of $\aone$-rigid varieties can be extremely subtle; see for example \cite{Roendigs} for a study of the stable $\aone$-homotopy types of curves.
\end{ex}

In the spirit of results to appear later in this section, we mention the following result.

\begin{lem}
A smooth $k$-scheme $X$ is $\aone$-rigid if and only if for every finitely generated separable extension $L/k$ the map
\[
X(L) \longrightarrow X(\aone_L)
\]
induced by the projection $\aone_L \to \Spec L$ is a bijection.
\end{lem}

\subsubsection*{$\aone$-homotopy classification of curves}
Lemma \ref{lem:aonerigid} implies that the classification of smooth $\aone$-rigid schemes up to (\'etale) $\aone$-weak equivalence coincides with the isomorphism classification.  Using the classification of curves over a field $k$, one can show that any (open subscheme of a) curve of genus $g \geq 1$ is $\aone$-rigid, and similarly that (any open subscheme of) $\gm$ is $\aone$-rigid.  Combining these facts with the observation that smooth $\aone$-connected $k$-schemes have a $k$-rational point, we deduce the following result ({\em cf.} \cite[\S 3 Remark 2.5]{MV}).

\begin{prop}[$\aone$-homotopy classification of curves]
\label{prop:curves}
Two smooth proper curves of genus $g \geq 1$ are $\aone$-weakly equivalent if and only if they are isomorphic.  A smooth proper curve is $\aone$-connected if and only if it is isomorphic to $\pone$.
\end{prop}

\begin{rem}
Generalizing Proposition \ref{prop:curves}, one can show that for arbitrary fields the $\aone$-homotopy classification and isomorphism classification of curves coincide.  One method to do this has been worked out by M. Severitt (see \cite{Severitt}).  Indeed the discussion above reduces the problem to determining when two smooth conics without a rational point are $\aone$-weakly equivalent.  In this case, Severitt uses work of Karpenko \cite{Karpenko} to show that two conics have isomorphic motives if and only if they are isomorphic.  This result can then be ``lifted" to the $\aone$-homotopy category.  It is, however, possible to give a relatively elementary ``geometric" proof of the fact that two smooth proper conics over $k$ are $\aone$-weakly equivalent if and only if they are isomorphic.  We defer this result for considerations involving its length.
\end{rem}

\subsection{$\aone$-chain connectedness}
We now recall some algebro-geometric analogs of path connectedness.  Given $X \in \Sm_k$, $L$ a finitely generated separable extension of $k$, and points $x_0,x_1 \in X(L)$ an {\em elementary $\aone$-equivalence between $x_0$ and $x_1$} is a morphism $f: \aone \to X$ such that $f(0) = x_0$ and $f(1) = x_1$.   We will say that two points $x,x' \in X(L)$ are {\em $\aone$-equivalent} if they are equivalent with respect to the equivalence relation generated by elementary $\aone$-equivalence.

\begin{notation}
\label{notation:aoneequivalenceclassesofpoints}
We write $X(L)/\smallsim$ for the quotient of the set of $L$-rational points for the above equivalence relation and refer to this quotient as the set of {\em $\aone$-equivalence classes of $L$-points}.
\end{notation}

\begin{defn}
\label{defn:chainconnected}
We say that $X \in {\mathcal Sm}_k$ is {\em (weakly) $\aone$-chain connected} if for every finitely generated separable field extension $L/k$ (resp. separably closed field extension) the set of $\aone$-equivalences classes of $L$-points $X(L)/\smallsim$ consists of exactly $1$ element.
\end{defn}

\begin{rem}
The definition above is closely related to Manin's notion $R$-equivalence.  Suppose $k$ is a field and $X \in \Sm_k$.  Recall that two $k$-points in $X$ are called {\em directly $R$-equivalent} if there exists a morphism from an open subscheme of $\pone$ to $X$ whose image contains the given points.  We write $X(k)/R$ for the quotient of $X(k)$ by the equivalence relation generated by $R$-equivalence.  We say that $X$ is {\em separably $R$-trivial} if for every finitely generated separable extension field $L$ of $k$, $X(L)/R = \ast$.  If $X$ is a smooth proper $k$-variety, then $\aone$-chain connectedness of $X$ is equivalent to the notion of {\em separable $R$-triviality} of $X$.
\end{rem}

The {\em algebraic $n$-simplex} is the smooth affine $k$-scheme
\[
\Delta^{n}_{\aone} := \Spec k[x_0,\ldots,x_n]/(\sum_{i=0}^n x_i -1).
\]
Note that $\Delta^n_{\aone}$ is non-canonically isomorphic to ${\mathbb A}^n_k$.  Given $X \in {\mathcal Sm}_k$, let $Sing_*^{\aone}(X)$ (resp. $Sing_*^{\aone,\et}(X)$) denote the Suslin-Voevodsky singular construction of $X$, i.e., the (\'etale) simplicial sheaf defined by
\[
U \mapsto Hom_{\Sm_k}(\Delta^{\bullet}_{\aone} \times U,X);
\]
(see \cite[p. 88 and p. 107]{MV}).  By construction, there is a canonical morphism $X \to Sing_*^{\aone}(X)$ (resp. $X \to Sing_*^{\aone,\et}(X)$) that is an $\aone$-weak equivalence (in the \'etale topology); see \cite[\S 2 Corollary 3.8]{MV} for more details.

\begin{defn}
\label{defn:svsingularconnected}
For $X \in {\mathcal Sm}_k$, set
\[
\pi_0^{ch}(X) := \pi_0^s(Sing_*^{\aone}(X)),
\]
and
\[
\pi_0^{ch,\et}(X) := \pi_0^{s,\et}(Sing_*^{\aone,\et}(X)).
\]
We refer to the sheaf $\pi_0^{ch}(X)$ (resp. the \'etale sheaf $\pi_0^{ch,\et}(X)$) as the {\em sheaf of (\'etale) $\aone$-chain connected components of $X$.}
\end{defn}

\begin{lem}
\label{lem:chainepimorphism}
Suppose $X \in \Sm_k$.  The maps
\[
\begin{split}
\pi_0^{ch}(X) &\longrightarrow \pi_0^{\aone}(X) \\
\pi_0^{ch,\et}(X) &\longrightarrow \pi_0^{\aone,\et}(X)
\end{split}
\]
are epimorphisms.
\end{lem}

\begin{proof}
Since the canonical map $X \to Sing_*^{\aone}(X)$ (resp. $X \to Sing_*^{\aone,\et}(X)$) is an $\aone$-weak equivalence (in the \'etale topology), the result follows immediately from Corollary \ref{cor:unstableaoneconnectivitytheorem} applied to $Sing_*^{\aone}(X)$ or $Sing_*^{\aone,\et}(X)$.
\end{proof}

\begin{cor}
\label{cor:sectionsoverfields}
Suppose $X$ is a smooth variety over a field $k$.  If $L/k$ is a finitely generated separable extension (or separably closed extension) and $X(L)/\smallsim = *$, then $\pi_0^{\aone}(X)(L) = *$.  Thus, if $X$ is weakly $\aone$-chain connected, it is weakly $\aone$-connected.
\end{cor}

\subsubsection*{Connectedness and chain connectedness}
An analog of the last statement of Corollary \ref{cor:sectionsoverfields} involving $\aone$-chain connectedness has been studied in several places.

\begin{prop}[{\em cf.} \textup{\cite[Lemma 3.3.6]{MIntro}} and \textup{\cite[Lemma 6.1.3]{MStable}}]
\label{prop:aoneconnected}
If $X \in \Sm_k$ is $\aone$-chain connected, then $X$ is $\aone$-connected.
\end{prop}

\begin{conj}
\label{conj:epimorphism}
The epimorphism $\pi_0^{ch}(X) \to \pi_0^{\aone}(X)$ of \textup{Lemma \ref{lem:chainepimorphism}} is always an isomorphism.  In particular, an object $X \in \Sm_k$ is $\aone$-chain connected if and only if it is $\aone$-connected.
\end{conj}

\begin{rem}
In support of the first statement of this conjecture, we establish Theorem \ref{thm:propercharacterization}, which shows that if $X$ is a proper scheme, then the epimorphism of Lemma \ref{lem:chainepimorphism} is a bijection on sections over finitely generated separable extensions $L$ of $k$.  Corollary \ref{cor:chainconverse} establishes that a smooth proper variety is $\aone$-connected if and only if it is $\aone$-chain connected.  Note that the conjecture above would follow immediately if one could prove that $Sing_*^{\aone}(X)$ is $\aone$-local (we might call an $X$ for which $Sing_*^{\aone}(X)$ is $\aone$-local {\em chain $\aone$-local}); this condition has been verified in a number of cases (see, e.g., \cite[Theorem 7.2]{MField}).
\end{rem}

\begin{defn}
\label{defn:combinatorial}
We will say that an $n$-dimensional smooth $k$-variety $X$ is {\em covered by affine spaces} if $X$ admits an open affine cover by finitely many copies of ${\mathbb A}^n_k$ such that the intersection of any two copies of ${\mathbb A}^n_k$ has a $k$-point (this last condition is superfluous if $k$ is infinite).
\end{defn}

\begin{lem}
\label{lem:combinatorialconnected}
If $X$ is a smooth $k$-variety that is covered by affine spaces, then $X$ is $\aone$-chain connected.
\end{lem}

\begin{ex}
\label{ex:coveredbyaffinespaces}
For simplicity assume that $k$ is an algebraically closed field.  Smooth $k$-varieties covered by affine spaces are all rational as algebraic varieties.  However, the collection of such varieties includes all rational smooth proper varieties of dimension $\leq 2$, smooth proper toric varieties (\cite{Fulton}), and generalized flag varieties for connected reductive groups over $k$.  Generalizing both of these examples, recall that a normal variety on which a connected reductive group $G$ acts is said to be {\em spherical} if a Borel subgroup $B \subset G$ acts with a dense orbit.  Using the local structure theory of Brion-Luna-Vust, one can check that any smooth proper spherical variety over an algebraically closed field having characteristic $0$ is covered by affine spaces (see \cite[1.5 Corollaire]{BLV}).  On the other hand, we will see that even over $\cplx$, there are smooth proper varieties that are $\aone$-connected yet not covered by affine spaces (see Example \ref{ex:BCTSSD}).
\end{ex}

\subsection{Near rationality}
Recall that two $k$-varieties $X$ and $Y$ are {\em $k$-birational} or {\em $k$-birationally equivalent} if the function fields $k(X)$ and $k(Y)$ are isomorphic as $k$-algebras.  We now review some ``near rationality" properties for algebraic varieties that appear in the sequel.  The treatment below is necessarily quite abridged, and we refer the reader to \cite{CTSRationalityFields} and \cite{Kollar} for more details.

\begin{defn}
\label{defn:rationalitynotions}
A $k$-variety $X$ is called
\begin{itemize}
\item[i)] {\em $k$-rational} if it is $k$-birational to ${\mathbb P}^n$,
\item[ii)] {\em stably $k$-rational} if there exists an integer $n \geq 0$ such that $X \times {\mathbb P}^n$ is $k$-rational,
\item[iii)] a {\em direct factor of a $k$-rational variety}, or simply {\em factor $k$-rational}, if there exists a $k$-variety $Y$ such that $X \times Y$ is $k$-rational,
\item[iv)] {\em retract $k$-rational} if there exists an open subscheme $U$ of $X$ such that the identity map $U \to U$ factors through an open subscheme $V$ of an affine space (over $k$).
\item[v)] {\em (separably) $k$-unirational} if $k(X)$ is a subfield of a purely transcendental extension of $k$ (separable over $k(X)$), i.e., there exists a (separable) dominant rational map from a projective space to $X$, and finally
\item[vi)] {\em separably rationally connected} if there is a $k$-variety $Y$ and a morphism $u: U = Y \times \pone \to X$ such that the map $u^{(2)}: U \times_Y U \to X \times X$ is dominant and smooth at the generic point.
\end{itemize}
\end{defn}

We will say that a $k$-variety is {\em rational} if it is $\bar{k}$-rational for an algebraic closure $\bar{k}$ of $k$.  Similar conventions could be made for the other definitions, but we will not use these notions in this paper.

\begin{lem}[{\em cf.} {\cite[Proposition 1.4]{CTSRationalityFields}}]
\label{lem:rationalitynotions}
If $X$ is a smooth $k$-variety, each of the first four conditions of \textup{Definition \ref{defn:rationalitynotions}} implies the subsequent one.  If $X$ is separably $k$-unirational then $X$ is separably rationally connected.
\end{lem}

\begin{proof}
The first two implications of the statement are clear from the definitions.  For the third implication assume $Y$ is a $k$-variety such that $X \times Y$ is $k$-rational.  Let $U \subset X \times Y$ be a non-empty open subscheme which is isomorphic to an open subscheme of an affine space.  Let $(x_0,y_0) \in U(k)$.  Let $X_1$ be the non-empty open subscheme of $X$ defined by $X_1 \times \setof{y_0} := U \cap (X \times \setof{y_0})$.  The open set $U_1 = U \cap (X_1 \times Y)$ is still isomorphic to an open set of affine space.  The composite map $X_1 \to U_1 \to X_1$, with the first map induced by $x \mapsto (x,y_0)$ and the second map induced by projection onto $X$ provides the necessary retraction.  The fourth implication is clear from the definitions.  For the last statement, see \cite[Example 3.2.6.2]{Kollar}.
\end{proof}

\begin{rem}
The birational geometry of nearly $k$-rational varieties is an incredibly rich subject (see, e.g., \cite{ColliotThelene} or \cite{Kollar}); the comments we now make are intended to give a flavor of results.  Chevalley and Manin introduced and studied a class of varieties they called {\em special} (see \cite[\S 14]{Manin}) that form a slightly more general class than our smooth schemes covered by affine spaces.

The Zariski cancellation problem, sometimes called the birational cancellation problem, asked whether stably $k$-rational varieties are necessarily $k$-rational.  A negative solution to this problem (even over $\cplx$) was provided in the celebrated work \cite{BCTSSD}; see Example \ref{ex:BCTSSD} for more details.

It is known that if $k$ is not algebraically closed, there exist varieties that are factor $k$-rational yet not stably $k$-rational \cite{CTS}; see Example \ref{ex:CTS} for more details.  The notion of retract $k$-rationality was introduced and studied by Saltman ({\em cf}. \cite[Definition 3.1]{Saltman}) in relation to Noether's problem regarding rationality of fields of invariants.

For fields having characteristic $0$, Campana, Koll{\'a}r, Miyaoka, and Mori introduced rational connectedness (see \cite[IV.3]{Kollar}).  Separable rational connectedness is equivalent to rational connectedness for fields having characteristic $0$ but is a more well-behaved notion in positive characteristic (see \cite[IV.3.2]{Kollar}).  At the moment it is not known whether there exist separably rationally connected varieties that are not unirational, though there are many expected counterexamples.
\end{rem}

\begin{ex}
\label{ex:BCTSSD}
If $k$ is a non-algebraically closed perfect field, there exist stably $k$-rational, non-$k$-rational smooth proper surfaces by \cite{BCTSSD}, though see also \cite{ShepherdBarron}.  Such varieties are $\aone$-connected by Corollary \ref{cor:stablyrational}.  More explicitly, over any field $k$ having characteristic unequal to $2$, let $P \in k[x]$ be an irreducible separable polynomial of degree $3$ and discriminant $a$.  Any smooth proper model of the surface $X_a$ given by the affine equation $y^2 - az^2 = P(x)$ has the property that $X_a \times {\mathbb P}^3$ is $k$-birationally equivalent to ${\mathbb P}^3$, though if $a$ is not a square in $k$, then $X_a$ is not $k$-rational (see \cite[Th\'eor\`eme 1 p. 293]{BCTSSD}).  If $k$ is algebraically closed, one can consider the above result for $k(t)$ to obtain threefolds that are stably rational yet non-rational (see \cite[Th\'eor\`eme 1' p. 299]{BCTSSD} for a precise statement).
\end{ex}

\begin{ex}
\label{ex:CTS}
If $k$ is not algebraically closed, there exist examples of smooth proper varieties that are factor $k$-rational yet not stably $k$-rational: see \cite[Proposition 20 C p. 223]{CTS}.  Indeed, one can construct a pair of tori $T$ and $T'$ over $\Q$ such that $T \times T'$ is $k$-rational while neither $T$ nor $T'$ is $k$-rational.  Taking smooth proper models of these tori provides the required example.
\end{ex}

\subsubsection*{Near rationality and $\aone$-connectedness}
We now proceed to link (weak) $\aone$-connectedness with rationality properties of algebraic varieties.  For algebraically closed fields having characteristic $0$, it was initially hoped that $\aone$-connectedness in the sense studied above would be equivalent to separable rational connectedness, however we will see that there are both geometric (Example \ref{ex:conicbundles}) and cohomological reasons this cannot be true (see the beginning of \S \ref{s:classifyingspaces}).

We will say that {\em weak factorization holds over $k$ in dimension $n$} if given any two $k$-birationally equivalent smooth proper varieties $X$ and $X'$ of dimension $n$, there exist a sequence of smooth proper varieties $Z_1,\ldots,Z_n,X_1,\ldots,X_n$ of dimension $n$, and a diagram of the form
\[
X \longleftarrow Z_1 \longrightarrow X_1 \longleftarrow Z_2 \longrightarrow \cdots \longleftarrow Z_{n-1} \longrightarrow X_n \longleftarrow Z_n \longrightarrow X',
\]
where each morphism with source $Z_i$ is a blow-up at a smooth center.

\begin{thm}
\label{thm:stablyrational}
Suppose $k$ is a perfect field, and assume weak factorization holds over $k$ in dimension $n$.
\begin{itemize}
\item[i)] If $X$ and $X'$ are $k$-birationally equivalent smooth proper varieties of dimension $n$, then $X$ is (weakly) $\aone$-chain connected if and only if $X'$ is (weakly) $\aone$-chain connected.
\end{itemize}
Suppose further that $k$ has characteristic $0$.
\begin{itemize}
\item[ii)] If $X$ is a retract $k$-rational smooth proper variety, then $X$ is $\aone$-chain connected and thus $\aone$-connected.
\end{itemize}
\end{thm}

\begin{proof}
For (i) using the assumption that weak factorization holds in dimension $n$, it suffices to check that if $X' \to X$ is a blow-up of a smooth proper variety at a smooth center, then $X'$ is (weakly) $\aone$-chain connected if and only if $X$ is (weakly) $\aone$-chain connected; this is exactly the content of Proposition \ref{prop:blowupconnected} below.

For (ii), we know that there exists an open subscheme $U \subset X$ and an open subscheme $V$ of ${\mathbb A}^m$ such that $id: U \to U$ factors through $V$.  Thus, there are a rational smooth proper variety $Z$, and rational maps $X \dashrightarrow Z \dashrightarrow X$ factoring the identity map.  By resolution of indeterminacy, we can assume that there exists a rational smooth proper variety $Y$ dominating $Z$ and a proper birational morphism $Y \to X$.  Again using resolution of indeterminacy, we can assume there is a smooth proper variety $X'$ and proper birational morphisms $X' \to X$ and $X' \to Y$ such that the morphism $X' \to X$ restricts to the identity on $U$.

Note that \cite[Theorem 0.1.1]{AKMW} establishes weak factorization in the sense above for any field $k$ having characteristic $0$ and any integer $n \geq 0$.  By the result of (i), $X$ is $\aone$-chain connected if and only if $X'$ is $\aone$-chain connected.  For any finitely generated separable extension $L/k$, composition induces maps of $\aone$-equivalence classes of $L$-points
\[
X'(L)/\smallsim \longrightarrow Y(L)/\smallsim \longrightarrow X(L)/\smallsim.
\]
Since $Y$ is $k$-rational and ${\mathbb P}^n$ is $\aone$-chain connected, again using (i), we deduce that $Y(L)/\smallsim$ consists of exactly $1$ element.  Since the composite map is a bijection, it follows that $X(L)/\smallsim$ must also consist of a single element.  Applying Proposition \ref{prop:aoneconnected} finishes the proof.
\end{proof}

Weak factorization in the above sense for surfaces over perfect fields $k$ having arbitrary characteristic is well known (see, e.g., \cite[Theorem II.11 and Appendix A]{Beauville}).  Thus, we have deduced the following result.

\begin{cor}
\label{cor:stablyrational}
If $k$ is a perfect field, any $k$-rational smooth proper surface is ${\mathbb A}^1$-connected.  If $k$ is a field having characteristic $0$, then any stably $k$-rational, or factor $k$-rational smooth proper variety is $\aone$-connected.
\end{cor}

\begin{prop}[{\em cf.} {\cite[Proposition 10]{CTS}}]
\label{prop:blowupconnected}
Suppose $f: X \to Y$ is a blow-up of a smooth proper $k$-scheme at a smooth closed subscheme $Z$ of codimension $r+1$.  For any finitely generated separable field extension $L/k$, $f$ induces a map of $\aone$-equivalence classes of $L$-points $X(L)/\smallsim \to Y(L)/\smallsim$ that is a bijection.  Moreover, $Y$ is $\aone$-chain connected if and only if $X$ is $\aone$-chain connected.
\end{prop}

\begin{proof}
The schematic fibers of $f$ are either projective spaces of dimension $r$ or of dimension $0$ so $X(L) \to Y(L)$ is surjective for any finitely generated extension $L/k$.  It follows that $f$ induces a surjective function $X(L)/\smallsim \to Y(L)/\smallsim$.

Since $X$ and $Y$ are both proper, to show this function $X(L)/\smallsim \to Y(L)/\smallsim$ is injective, it suffices to prove the following fact.  Given $y, y' \in Y(L)$, a morphism $h: \pone_L \to Y$ joining $y$ and $y'$, and lifts $x,x' \in X(L)$ such that $f(x) = y$ and $f(x') = y'$, the points $x$ and $x'$ are $\aone$-equivalent $L$-points.  By the surjectivity statement of the previous paragraph, the $L(t)$-point of $Y$ determined by $h$ lifts to an $L(t)$ point of $X$.  Since $X$ is proper, this $L(t)$ point determines a morphism $\pone_L \to X$ lifting $h$.
\end{proof}

If $k$ is a field, and $k^s$ is a separable closure of $k$, recall that a $k$-variety $X$ is {\em strongly rationally connected} if any $k^s$-point can be joined to a generic $k^s$-point by a {\em proper} rational curve ({\em cf.} \cite[Definition 14]{HassettTschinkel}).  The next result follows from the definitions together with Corollary \ref{cor:sectionsoverfields}, and the fact that strong rational connectedness for a smooth proper $k$-variety is invariant under separably closed extensions (see \cite{Kollar} Theorem IV.3.9).

\begin{thm}
\label{thm:stronglyrationallyconnected}
Suppose $k$ is a perfect field.  If $X$ is a strongly rationally connected smooth proper $k$-variety, then $X$ is weakly $\aone$-connected.
\end{thm}

\begin{cor}
\label{cor:separablyrationallyconnected}
If $k$ is a perfect field and $X$ is a separably rationally connected smooth proper $k$-variety, then $X$ is weakly $\aone$-connected.
\end{cor}

\begin{rem}
As was pointed out to us by Colliot-Th\'el\`ene determining whether $\aone$-chain connectedness in the sense of Definition \ref{defn:chainconnected} is equivalent to retract $k$-rationality is an open problem.
\end{rem}

\subsection{Comparison results}
\subsubsection*{Comparison of $\aone$- and \'etale $\aone$-connectedness}
Given ${\mathcal X} \in \Spc^{\et}_k$ we now provide a comparison of $\aone$-connectedness and \'etale $\aone$-connectedness.  We use the comparison of topologies functoriality developed in Appendix \ref{s:homotopyoverview}.  Following the notation there, we let $\alpha: ({\mathcal Sm}_k)_{\et} \to ({\mathcal Sm}_k)_{Nis}$ denote the comparison of sites map.

By Lemma \ref{lem:aonefunctoriality}, the derived functor ${\mathbf R}\alpha_*$ sends \'etale $\aone$-local objects to $\aone$-local objects.  Adjointness of pullback and pushforward provides, for any $U \in {\mathcal Sm}_k$, a canonical bijection:
\[
Hom_{\het{k}}(U,{\mathcal X}) \isomto [U,{\bf R}\alpha_* {\mathcal X}]_{\aone}.
\]
Thus, for any $U \in \Sm_k$, the unit of adjunction provides morphism:
\[
[U,{\mathcal X}]_{\aone} \longrightarrow Hom_{\het{k}}(U,{\mathcal X}).
\]
Write $a_{\et}\pi_0^{\aone}({\mathcal X})$ for the \'etale sheafification of the presheaf on the left hand side.  Sheafifying both sides for the \'etale topology, we obtain a morphism
\begin{equation}
\label{eqn:etalecomparison}
a_{\et}\pi_0^{\aone}({\mathcal X}) \longrightarrow \pi_0^{\aone,\et}({\mathcal X}).
\end{equation}

\begin{lem}
\label{lem:etalecomparison}
The morphism $a_{\et}\pi_0^{\aone}({\mathcal X}) \to \pi_0^{\aone,\et}({\mathcal X})$ of \textup{Equation \ref{eqn:etalecomparison}} is an epimorphism of \'etale sheaves.  Thus, if the space underlying an object ${\mathcal X} \in \Spc_k^{\et}$ is $\aone$-connected it is also \'etale $\aone$-connected.
\end{lem}

\begin{proof}
We will check that the morphism in question is an epimorphism on stalks.  The morphism ${\mathcal X} \to {\mathbf R}\alpha_*{\mathcal X}$ factors through a morphism ${\mathcal X} \to L_{\aone}({\mathcal X}) \to L_{\aone}({\mathbf R}\alpha_* {\mathcal X})$.  By the unstable $\aone$-$0$-connectivity theorem (\ref{cor:unstableaoneconnectivitytheorem}) in the \'etale topology, the composite map induces an epimorphism of \'etale sheaves of simplicial connected components.
\end{proof}

Using this observation, one can show that any variety that becomes $\aone$-connected over a separable closure is in fact \'etale $\aone$-connected.

\begin{ex}
\label{ex:conicbundles}
Suppose $a_1,\ldots,a_{2m}$ are distinct elements of ${\mathbb R}$, and let $S$ be any smooth compactification of the smooth affine hypersurface in ${\mathbb A}^3$ defined by the equation
\[
x^2 + y^2 = - \prod_{i=1}^{2m} (z - a_i).
\]
Projection onto $z$ determines a morphism from this hypersurface to $\pone$ with conic fibers.  The compactified surface is birationally ruled over $\pone$ and therefore rational over $\cplx$.  In particular $S$ is \'etale $\aone$-connected.  One can show that the manifold $S(\real)$ has $m$ connected components, and the set of $\aone$-equivalences classes of $\real$-points $S(\real)/\smallsim$ coincides with the set $\pi_0(S(\real))$ ({\em cf.} \cite[Corollary 3.4 and Theorem 4.6]{Kollarrealsurface}).  Using the topological realization functor (\cite[\S 3 Lemma 3.6]{MV}) one can show that such $S$ are $\aone$-disconnected.  Thus, even if $X$ becomes $\aone$-connected over a finite extension of a field $k$, it need {\em not} be $\aone$-connected over $k$ itself.  Furthermore, $\pi_0^{\aone}(S)$ is not necessarily a subsheaf of a point even if $\pi_0^{\aone,\et}(S)$ is a point.
\end{ex}

\subsubsection*{Comparing $\aone$-connectedness and $\aone$-chain connectedness}
Suppose $X$ is an arbitrary scheme having finite type over a field $k$.  The Nisnevich topology is subcanonical, i.e., the functor of points of $X$ is a Nisnevich sheaf.  We will abuse notation and write $X$ for both a scheme (possibly singular) and the Nisnevich sheaf on $\Sm_k$ determined by its functor of points.  To provide evidence for Conjecture \ref{conj:epimorphism}, we give the following result whose proof we defer to Section \ref{s:unramified}.

\begin{thm}
\label{thm:propercharacterization}
Suppose $X$ is a proper scheme having finite type over a field $k$.  The canonical epimorphism of \textup{Lemma \ref{lem:chainepimorphism}} induces for every finitely generated separable extension $L/k$ a bijection:
\[
\pi_0^{ch}(X)(L) \longrightarrow \pi_0^{\aone}(X)(L).
\]
\end{thm}

\begin{cor}
\label{cor:chainconverse}
If $k$ is a field, and $X \in \Sm_k$ is proper over $k$, then
\begin{itemize}
\item $X$ is $\aone$-connected if and only if $X$ is $\aone$-chain connected, and
\item $X$ is separably rationally connected if and only if $X$ is weakly $\aone$-connected.
\end{itemize}
\end{cor}

\begin{rem}
This result provides a positive solution to the problem about the structure of the set $[\Spec k,X]_{\aone}$ posed in \cite[p. 386]{MIntro}.
\end{rem}

Let us record some extremely useful consequences of this result.  The first corollary follows immediately by combining Theorem \ref{thm:propercharacterization} with Theorem \ref{thm:stablyrational}.  The second corollary follows from the fact that $\aone$-chain connected varieties are separably rationally connected for fields having characteristic $0$, and separably rationally connected smooth proper surfaces over an algebraically closed field are rational.

\begin{cor}
\label{cor:birationalinvariance}
Suppose $k$ is a field having characteristic $0$.  If $X$ and $X'$ are two $k$-birationally equivalent smooth proper varieties, then $X$ is $\aone$-connected if and only if $X'$ is $\aone$-connected.
\end{cor}

\begin{cor}[{\em cf.} {\cite[Exercise IV.3.3.5]{Kollar}}]
\label{cor:aoneconnectednesssurfaces}
Suppose $k$ is an algebraically closed field having characteristic $0$.  A smooth proper $k$-variety of dimension $\leq 2$ is $\aone$-connected if and only if it is rational.
\end{cor}

\begin{rem}
As we noted above, if $X$ is a smooth proper $k$-variety, then the equivalence relations given by $R$-equivalence and $\aone$-equivalence of points coincide.  Corollary \ref{cor:chainconverse} then implies, e.g., that $X$ is (weakly) $\aone$-connected if and only if it is separably $R$-trivial (separably rationally connected).
\end{rem}

\begin{cor}
Assume $k$ is a field having characteristic $0$.  Suppose $X \in \Sm_k$ and suppose $j: X \hookrightarrow \bar{X}$ is an open immersion into a smooth proper variety.  For any finitely generated separable extension $L$ of $k$, the image of the map $\pi_0^{\aone}(X)(L) \to \pi_0^{\aone}(\bar{X})(L) = \bar{X}(L)/\smallsim$ coincides with $X(L)/R$.  In particular for any $X \in \Sm_k$, the map $X(L) \to X(L)/R$ factors through the surjective map $X(L) \to \pi_0^{\aone}(X)(L)$.
\end{cor}

\section{$\aone$-$h$-cobordisms and rational smooth proper surfaces}
\label{s:cobordisms}
In this section we study the notion of $\aone$-$h$-cobordism of smooth schemes mentioned in \S \ref{s:introduction}.  Using this notion, Theorem \ref{thm:surfaces} provides the $\aone$-homotopy classification for rational smooth proper surfaces; the proof is essentially elementary.  Along the way, we prove general results about the $\aone$-homotopy types of iterated blow-ups of points on smooth proper (\'etale) $\aone$-connected varieties (see Lemma \ref{lem:movingpoints}) and classify the total spaces of ${\mathbb P}^n$-bundles over $\pone$ up to $\aone$-weak equivalence (see Proposition \ref{prop:pnbundlesoverpone}).

Suppose $X$ is a smooth $k$-scheme.  Specifying a regular function $f \in \Gamma(X,\O_X)$ is equivalent to specifying a morphism $f: X \to \aone$.  Note that $\aone(k)$ has two canonical elements, which we denote by $0$ and $1$.  In the remainder of this section, we will write $f^{-1}(0)$ and $f^{-1}(1)$, or just $X_0$ and $X_1$ assuming $f$ is understood, for the scheme-theoretic fibers over the points $0$ and $1$.  We will say that a closed point $x \in \aone$ is a {\em regular value} if the scheme theoretic fiber $f^{-1}(x)$ is a smooth scheme, otherwise $x$ will be called a {\em critical value} of $f$.  We begin by defining $\aone$-$h$-cobordisms and studying their general properties.

\subsection{Basic definitions and general properties}
\begin{defn}
\label{defn:aonehcobordism}
Suppose $X \in \Sm_k$, and $f: X \to \aone$ is a proper, surjective morphism.  We will say that $f$ (or the pair $(X,f)$) is an {\em $\aone$-$h$-cobordism} if $0$ and $1$ are regular values of $f$, and the inclusion maps $X_0 \hookrightarrow X$ and $X_1 \hookrightarrow X$ are $\aone$-weak equivalences.
\end{defn}

If $Y$ is a smooth proper $k$-scheme, the projection morphism $p_Y: Y \times \aone \to Y$ is always an $\aone$-$h$-cobordism that we will call the {\em trivial} $\aone$-$h$-cobordism.  Given an $\aone$-$h$-cobordism $f: X \to \aone$, we will say that $X_0$ and $X_1$ are {\em directly $\aone$-$h$-cobordant}.  More generally, we will say that two varieties $X$ and $Y$ are {\em $\aone$-$h$-cobordant} if they are in the same equivalence class for the equivalence relation generated by direct $\aone$-$h$-cobordance.

\begin{rem}
Given an $\aone$-$h$-cobordism $(W,f)$, one can show that $f$ is necessarily a {\em flat} morphism.  Since by assumption $f$ is smooth at $0$ and smoothness is open, it follows that $f$ has at most finitely many critical values.  Thus, we can think of $f$ as providing an extremely special deformation of the fiber over $0$.
\end{rem}

\begin{rem}
Two varieties $X$ and $Y$ that are $\aone$-$h$-cobordant are algebraically cobordant in the sense that they give rise to the same class in the algebraic cobordism ring $\Omega^*(k)$ \cite[Remark 2.4.8, and Definition 2.4.10]{LM}; this observation justifies our choice of terminology.
\end{rem}

\begin{rem}
Suppose $W$ is an $h$-cobordism between smooth manifolds $M$ and $M'$.  In classical topology, one studies $W$ by means of handle decompositions.  By choosing a Morse function $f: W \to \real$, one can decompose $f$ into elementary pieces corresponding to the critical points of $f$.  The handle decomposition theorem shows that an $h$-cobordism admitting a Morse function with no critical values is necessarily trivial.  On the contrary, we will see in Example \ref{ex:rank2bundles} that in algebraic geometry there exist {\em non-trivial} $\aone$-$h$-cobordisms $(W,f)$ where $f$ is a smooth morphism and thus has no critical values!  In fact, all $\aone$-$h$-cobordisms we can construct are of this form.
\end{rem}

\subsubsection*{$\aone$-$h$-cobordant bundles}
\begin{prop}
\label{prop:bundles}
Suppose $X,Y \in {\mathcal Sm}_k$ with $Y$ proper, and $g: Z \to X \times \aone$ is a smooth surjective morphism.  Assume further that the following condition holds.
\begin{itemize}
\item[{\bf (LT)}] There is a Nisnevich cover $u: U \to X$ such that the pullback along $u \times id: U \times \aone \to X \times \aone$ of $g$ is the projection of a product $U \times \aone \times Y \to U \times \aone$.
\end{itemize}
The morphism $f: Z \to \aone$ induced by composing the morphism $g$ with the projection $X \times \aone \to \aone$ is an $\aone$-$h$-cobordism.
\end{prop}

\begin{proof}
Since we can apply an automorphism of $\aone$ that exchanges the fibers over $0$ and $1$, it suffices to check that the inclusion morphism $Z_0 \hookrightarrow Z$ is an $\aone$-weak equivalence.  By assumption, we can choose an open cover $u: U \to X$ such that the pullback of $g$ along $u \times id$ trivializes.  Fix such a trivialization.  Our choice of trivialization determines an isomorphism $Y \times U \times \aone \isomt Z \times_{X \times \aone} (U \times \aone)$.  Also, the pull-back of the morphism $g: Z \to X \times \aone$ by $u \times id$ coincides via this isomorphism with the projection morphism $Y \times U \times \aone \to U \times \aone$.

Consider now the \u Cech simplicial scheme $\breve{C}(u \times id)$ whose $n$-th term is the $(n+1)$-fold fiber product of $U \times \aone$ with itself over $X \times \aone$.  By the discussion of Example \ref{ex:cechobjects}, the augmentation map $\breve{C}(u \times id) \to X \times \aone$ is a simplicial weak equivalence, and thus also an $\aone$-weak equivalence.  Using the chosen trivialization of $g$ along $u \times id$, one constructs an isomorphism from the \u Cech simplicial scheme associated with the Nisnevich covering map $Z \times_{X \times \aone} (U \times \aone) \to Z$ to the product $Y \times \breve{C}(u \times id)$; for the same reason this map is an $\aone$-weak equivalence.  Similarly, one checks that the map $\breve{C}(u) \to X$ is an $\aone$-weak equivalence and, by restriction, one constructs an isomorphism from the \u Cech simplicial scheme associated with the covering morphism $Z_0 \times_X U \to Z_0$ to the product $Y \times \breve{C}(u)$.

The construction above provides a Cartesian square of the form
\[
\xymatrix{
Y \times \breve{C}(u) \ar[r]\ar[d] & Y \times \breve{C}(u \times id) \ar[d]\\
\breve{C}(u) \ar[r]& \breve{C}(u \times id)
}
\]
If the inclusion morphism $\breve{C}(u) \hookrightarrow \breve{C}(u \times id)$ is an $\aone$-weak equivalence, it follows by \cite[\S 2 Lemma 2.15]{MV} that the product map $Y \times \breve{C}(u) \to Y \times \breve{C}(u \times id)$ is also an $\aone$-weak equivalence.  Since the map $X \hookrightarrow X \times \aone$ is an $\aone$-weak equivalence, the results of the previous paragraph allow us to conclude that $\breve{C}(u) \hookrightarrow \breve{C}(u \times id)$ is also an $\aone$-weak equivalence.
\end{proof}

\begin{rem}
Note that the proof of the above result never uses properness of $Y$.  Non-trivial $\aone$-$h$-cobordisms produced by this method will be described in Example \ref{ex:rank2bundles}.
\end{rem}

\subsubsection*{Blowing up a moving point}
\begin{prop}
\label{prop:blowups}
Assume $k$ is an infinite field.  Suppose $X \in {\mathcal Sm}_k$ is proper variety, and assume we have closed embedding $i: \aone \to X$ that factors through an open immersion ${\mathbb A}^n_k \to X$.
Let $\Gamma \subset X \times \aone$ denote the image of $i \times p_{\aone}$.  The projection $X \times \aone \to \aone$ induces a morphism
\[
f: {\sf Bl}_{\Gamma}(X \times \aone) \longrightarrow \aone.
\]
The morphism $f: {\sf Bl}_{\Gamma}(X \times \aone) \to \aone$ is an $\aone$-$h$-cobordism.
\end{prop}

\begin{proof}
The morphism $f$ is proper and surjective by construction.  Let $\Gamma_t$ denote the scheme theoretic fiber of the composite morphism $\Gamma \hookrightarrow X \times \aone \to \aone$ over $t \in \aone(k)$.  We just have to check that ${\sf Bl}_{\Gamma_0} \hookrightarrow {\sf Bl}_{\Gamma}(X \times \aone)$ (resp. ${\sf Bl}_{\Gamma_1} \hookrightarrow {\sf Bl}_{\Gamma}(X \times \aone)$) is an $\aone$-weak equivalence.

Consider the open cover of $X$ by the copy of ${\mathbb A}^n$ through which $i$ factors and $X \setminus i(\aone)$.  Since blowing up is Zariski local, we can reduce to the corresponding statement for affine space along the same lines as the proof of Proposition \ref{prop:bundles}.  Observe that given $i$ as above, the induced map $i \times p_{\aone}: \aone \to {\mathbb A}^n \times \aone$ is isomorphic to $0 \times id$ by means of the automorphism $(v,t) \mapsto (v - i(t),t)$ of ${\mathbb A}^n \times \aone$.  In this case, the blow-up of a line in ${\mathbb A}^n \times \aone$ is isomorphic to the product ${\sf Bl}_0({\mathbb A}^n) \times \aone$ in a manner compatible with the projection to $\aone$.

\end{proof}

\begin{rem}
It seems reasonable to expect that a result like Proposition \ref{prop:blowups} holds more generally, e.g., for any $\aone$-connected $Z$.
\end{rem}

\subsection{The $\aone$-homotopy classification of rational smooth proper surfaces}
The isomorphism classification of rational smooth proper surfaces over an algebraically closed field is well known (for proofs see, e.g, \cite{Beauville}).  Using the strong factorization theorem for surfaces \cite[Theorem II.11]{Beauville}, one can show that any rational smooth proper surface is isomorphic to an iterated blow-up of some finite collection of points on either ${\mathbb P}^2$ or on ${\mathbb F}_a = {\mathbb P}(\O_{\pone} \oplus \O_{\pone}(a))$.

Recall that ${\mathbb F}_a$ contains a unique curve, denoted here $C_a$, corresponding to the inclusion ${\mathbb P}(\O_{\pone}(a)) \hookrightarrow {\mathbb F}_a$ that has self-intersection number $-a$.  As is well-known, the variety ${\mathbb F}_a$ blown up at a point $x \in C_a(k)$ is isomorphic to ${\mathbb F}_{a-1}$ blown up at a point $x' \in {\mathbb F}_{a-1} \setminus C_{a-1}(k)$.  This construction provides the standard example of non-uniqueness of minimal models for ruled surfaces.  Note also that ${\mathbb F}_1$ can be identified with ${\mathbb P}^2$ blown-up at a point.

\begin{thm}
\label{thm:surfaces}
Suppose $k$ is an algebraically closed field.  For each integer $n \geq 1$, fix a finite set $I_n$ of distinct $k$-points on ${\mathbb P}^2$, and let $S_n = {\sf Bl}_{I_n} {\mathbb P}^2$.  Any $k$-rational smooth proper surface is $\aone$-weakly equivalent to either $\pone \times \pone$, ${\mathbb P}^2$ or some $S_i$.
\end{thm}

\begin{proof}
Let us first treat the minimal surfaces.  Any Hirzebruch surface ${\mathbb F}_a$ is $\aone$-weakly equivalent to either ${\mathbb F}_0 = \pone \times \pone$ or ${\mathbb F}_1 = {\sf Bl}_{x}{\mathbb P}^2$ (for some $k$-point $x$ of ${\mathbb P}^2$) by Lemma \ref{lem:mod2}.

For each integer $n > 0$, fix a finite set of $n$ distinct points on ${\mathbb P}^2$.  Lemma \ref{lem:movingpoints} implies that any variety that is an iterated blow-up of ${\mathbb P}^2$ at $n$ points is $\aone$-$h$-cobordant to the blow-up of ${\mathbb P}^2$ at the chosen set of $n$ distinct points.  Thus, assume we have a surface isomorphic to an iterated blow-up of $n$-points on ${\mathbb F}_a$ for some $n > 0$.  By repeated application of Lemma \ref{lem:decreasea} we deduce that such a variety is $\aone$-$h$-cobordant to a variety of the same kind with $a = 1$.   Since ${\mathbb F}_1 = {\sf Bl}_x {\mathbb P}^2$, the variety in question is isomorphic to an iterated blow-up of ${\mathbb P}^2$ at $n+1$-points, and we can apply Lemma \ref{lem:movingpoints} again to deduce that the resulting variety is $\aone$-$h$-cobordant to the blow-up of ${\mathbb P}^2$ at our chosen set of $n+1$ distinct points.
\end{proof}

\begin{rem}
Since separably rationally connected surfaces over an algebraically closed field $k$ are all rational, Theorem \ref{thm:surfaces} conjecturally provides (see Conjecture \ref{conj:epimorphism}) a complete classification of all $\aone$-connected or \'etale $\aone$-connected surfaces over such fields.  By Corollary \ref{cor:aoneconnectednesssurfaces}, the classification of $\aone$-connected smooth proper surfaces is established if $k$ has characteristic $0$.  Our proof will also show that any pair of $\aone$-weakly equivalent smooth proper surfaces are in fact $\aone$-$h$-cobordant by a series of $\aone$-$h$-cobordisms {\em without critical points}.
\end{rem}

\begin{extension}
It seems reasonable to expect a statement similar Theorem \ref{thm:surfaces} to hold more generally.  Indeed, if we modify the statement by replacing the ``finite set of points" by appropriate $0$-dimensional closed subschemes, we hope the corresponding result holds for $k$-rational smooth proper surfaces over an arbitrary field $k$.  Slightly more generally, it seems reasonable to expect a classification result analogous to Theorem \ref{thm:surfacesversion1} for smooth proper $\aone$-connected surfaces over an arbitrary field $k$; see Corollary \ref{thm:surfacesfundamentalgroup}: recall that over non-algebraically closed fields, rationality is not the same as $\aone$-connectedness even for surfaces.
\end{extension}

\begin{lem}
\label{lem:movingpoints}
Let $k$ be an infinite field.  Suppose $X$ is a smooth proper $k$-variety that is covered by affine spaces.  Suppose $Y$ and $Y'$ are each of the form
\[
{\sf Bl}_{y_1}({\sf Bl}_{y_2}(\cdots({\sf Bl}_{y_n}(X))))
\]
for specified collections of $k$-points $y_1,\ldots,y_n$ and $y'_1,\ldots,y'_m$.  Then $Y$ and $Y'$ are $\aone$-weakly equivalent if and only if $n=m$ in which case they are in fact $\aone$-$h$-cobordant.
\end{lem}

\begin{lem}
\label{lem:mod2}
Two Hirzebruch surfaces ${\mathbb F}_a$ and ${\mathbb F}_b$ are $\aone$-weakly equivalent if and only if $a$ and $b$ are congruent mod $2$, in which case they are $\aone$-$h$-cobordant.
\end{lem}

\begin{lem}
\label{lem:decreasea}
Let $k$ be an infinite field.  For any integer $n > 1$, and arbitrary collections of $k$-points $x_1,\ldots,x_n,y_1,\ldots,y_n$, the iterated blow-ups ${\sf Bl}_{x_1}(\cdots({\sf Bl}_{x_n}({\mathbb F}_a))\cdots)$ and ${\sf Bl}_{y_1}(\cdots({\sf Bl}_{y_n}({\mathbb F}_{a-1}))\cdots)$ are $\aone$-$h$-cobordant.
\end{lem}

\begin{proof}
Using Lemma \ref{lem:movingpoints} and the fact that the (rational) smooth proper surfaces in question are covered by affine spaces ({\em cf.} Example \ref{ex:coveredbyaffinespaces}), we can always assume that $x_n$ lies on $C_a$.  Apply the observation above about non-minimality of Hirzebruch surfaces to identify this iterated blow-up with a corresponding one with ${\mathbb F}_a$ replaced by ${\mathbb F}_{a-1}$.  Another application of Lemma \ref{lem:movingpoints} then allows one to construct an $\aone$-$h$-cobordism from this new iterated blow-up to ${\sf Bl}_{y_1}({\sf Bl}_{y_2}(\cdots({\sf Bl}_{y_n}({\mathbb F}_{a-1}))\cdots))$.
\end{proof}

\subsubsection*{Proof of Lemma \ref{lem:movingpoints}}
Lemma \ref{lem:movingpoints} is an immediate consequence of the following more precise result.

\begin{prop}
\label{prop:movingpointsarbitraryfield}
Assume $k$ is an infinite field, and suppose $X$ is a smooth proper $k$-variety that is covered by affine spaces.  Suppose $f_1: X_1 \to X$ and $f_2: X_2 \to X$ are proper birational morphisms that are composites of blow-ups of $k$-points.  The schemes $X_1$ and $X_2$ are $\aone$-weakly equivalent if and only if $\operatorname{rk} Pic(X_1) = \operatorname{rk} Pic(X_2)$, in which case they are $\aone$-$h$-cobordant.
\end{prop}

\begin{proof}
Observe first that if $X$ is covered by affine spaces, then $X_1$ and $X_2$ are both covered by affine spaces as well.  \newline

\noindent{\em Step 1.} Let us first prove that if $x_1$ and $x_2$ are distinct $k$-points on $X$, then ${\sf Bl}_{x_1}X$ and ${\sf Bl}_{x_2}X$ are $\aone$-$h$-cobordant.  Indeed, we can assume that $\dim X \geq 2$ as otherwise the blow-up of a point is trivial.  Since $X$ is covered by affine spaces, any two points can be connected by chains of lines, so it suffices to prove the statement assuming $x_1$ and $x_2$ both lie on a line.  This is exactly the statement of Proposition \ref{prop:blowups}.\newline

\noindent{\em Step 2.} Suppose $Y$ is covered by affine spaces.  Applying the conclusion of {\em Step 1} to $X = {\sf Bl}_x Y$, we deduce that the blow-up of $X$ at any $k$-point is $\aone$-$h$-cobordant to the blow-up of $X$ at a $k$-point not lying on the exceptional divisor of $X \to Y$.\newline

\noindent{\em Step 3.} Assume $Y$ is covered by affine spaces.  By repeated application of the observation in {\em Step 2} we deduce that given an iterated blow-up of the form ${\sf Bl}_{y_1}({\sf Bl}_{y_2}(\cdots({\sf Bl}_{y_n}(Y))\cdots))$ is $\aone$-$h$-cobordant to the blow-up of distinct points on $Y$.  \newline

\noindent{\em Step 4.} By applying the conclusion of {\em Step 3} to the varieties $X_1$ and $X_2$ we conclude they are both $\aone$-$h$-cobordant to blow-ups of distinct points on $X$.  In fact, since we can choose these points arbitrarily, we deduce that $X_1$ and $X_2$ must in fact be $\aone$-$h$-cobordant.  \newline

\noindent{\em Step 5.} To finish, let us note that the number of points being blown up is an $\aone$-homotopy invariant.  Indeed, the Picard group of a smooth $k$-scheme, and thus its rank, is an $\aone$-homotopy invariant by \cite[\S 4 Proposition 3.8]{MV} and it is well known that blowing up a point on a smooth variety increases the rank of the Picard group.  The result then follows by induction on the number of points.
\end{proof}

\begin{rem}
There is another way to approach the proof of Proposition \ref{prop:movingpointsarbitraryfield}.  If $X$ is a smooth proper $k$-variety that is covered by affine spaces, then one can show that the Fulton-MacPherson compactification (see \cite{FultonMacPherson}) of $n$-points on $X$, often denoted $X[n]$, is also a smooth proper $k$-variety that is covered by affine spaces (as it is an iterated blow-up of $X^{\times n}$ at subschemes covered by affine spaces).  One can then use the construction of $X[n]$ in terms of iterated blow-ups to produce appropriate $\aone$-$h$-cobordisms between iterated blow-ups of points on $X$.
\end{rem}

\subsubsection*{Proof of Lemma \ref{lem:mod2}}
We will deduce Lemma \ref{lem:mod2} from a much more general result regarding ${\mathbb P}^n$-bundles over $\pone$.  We begin with a construction of $\aone$-$h$-cobordisms.

\begin{ex}
\label{ex:rank2bundles}
Any ${\mathbb P}^n$-bundle over $\pone$ is the projectivization of a rank $(n+1)$ vector bundle on $\pone$.  Thus, it suffices for us to study rank $n$ vector bundles on $\pone \times \aone$.  Cover $\pone$ by $\aone_{0} \cong \pone \setminus \infty$ and $\aone_{\infty} = \pone \setminus 0$.  Since all vector bundles on an affine space are trivial ({\em cf.} \cite[Theorem 4]{Quillen}), any rank $(n+1)$ vector bundle on $\aone \times \aone$ is isomorphic to a trivial bundle.  Thus, fix a trivialization of such a bundle over ${\aone}_0 \times \aone$ and ${\aone}_{\infty} \times \aone$.  The intersection of these two open sets is isomorphic to $\gm \times \aone$.  Thus, isomorphism classes of rank $(n+1)$ vector bundles on $\pone \times \aone$ are in bijection with elements of $GL_{n+1}(k[t,t^{-1},x])$ up to change of trivialization, i.e., left multiplication by elements of $GL_{n+1}(k[t^{-1},x])$ and right multiplication by elements of $GL_{n+1}(k[t,x])$; the required cocycle condition is automatically satisfied.

Suppose ${\mathbf a} = (a_1,\ldots,a_{n+1})$, and set
${\mathcal E}({\mathbf a}) = \O_{\pone}(a_1) \oplus \cdots \oplus \O_{\pone}(a_{n+1})$.  The pull-back of ${\mathcal E}({{\mathbf a}})$ to $\pone \times \aone$ has transition function defined by the matrix whose diagonal entries are given by $(t^{-a_1},\ldots,t^{-a_{n+1}})$.  For notational simplicity, consider the rank $2$ case, and consider the transition function defined by
\[
\begin{pmatrix}
t^a & xt \\
0 & 1
\end{pmatrix}.
\]
Over $x = 1$, one can show that this transition function defines the bundle $\O(-a+1) \oplus \O(-1)$.  Over $x = 0$, this transition function defines the bundle $\O(-a) \oplus \O$.  This family of bundles provides an explicit $\aone$-$h$-cobordism between ${\mathbb F}_{a-2}$ and ${\mathbb F}_{a}$.

More generally, set ${\mathbb F}_{{\bf a}} = {\mathbb P}({\mathcal E}({\bf a}))$.  By permuting the elements of ${\bf a}$, we can assume its entries are increasing.  Let ${\bf a}'$ be another increasing sequence of $n+1$ integers.  Using explicit cocycles as above, one can construct $\aone$-$h$-cobordisms between ${\mathbb F}_{{\bf a}}$ and ${\mathbb F}_{{\bf a}'}$ whenever $\sum_i a_i \equiv \sum_i a'_i \mod n+1$ ({\em cf.} \cite[\S 9.4(i) and (iii)]{Ramanathan}).   Note that these $\aone$-$h$-cobordisms do not have critical values yet are not trivial.
\end{ex}

\begin{prop}
\label{prop:pnbundlesoverpone}
Let ${\bf a} = (a_1,\ldots,a_{n+1})$ and ${\bf a}' = (a_1', \ldots, a_{n+1}')$ be a sequences of integers with $a_1 \leq \cdots \leq a_{n+1}$ (and similarly for the entries of ${\bf a}'$).  The varieties ${\mathbb F}_{{\bf a}}$ and ${\mathbb F}_{{\bf a}'}$ are $\aone$-weakly equivalent if and only if $\sum_i a_i \equiv \sum_i a_i' \mod n+1$.
\end{prop}

\begin{proof}
Proposition \ref{prop:bundles} or Example \ref{ex:rank2bundles} constructs explicit $\aone$-$h$-cobordisms between ${\mathbb F}_{\bf a}$ and ${\mathbb F}_{{\bf a}'}$ whenever $\sum_i a_i \equiv \sum_i a_i' \mod n+1$.  For the only if part of the statement, we need to write down appropriate $\aone$-homotopy invariants.

Let us observe that the Chow (cohomology) {\em ring} $CH^*({\mathbb F}_{{\bf a}})$ can be explicitly computed as follows.  The Chern polynomial of a rank $n+1$ vector bundle ${\mathcal E}$ over $\pone$ takes the form $\xi^{n+1} + c_1({\mathcal E}) \xi^n$.  If $\sigma$ denotes the hyperplane class on $\pone$, then we can write $c_1({\mathcal E}) = a \sigma$ for some integer $a$.  Let $d = \sum_i a_i$.  These identifications give an isomorphism of graded rings
\[
CH^*({\mathbb F}_{{\bf a}}) \cong \Z[\sigma,\xi]/\langle \sigma^2,\xi^{n+1} + d \xi^n \sigma \rangle,
\]
where $\sigma$ and $\xi$ both have degree $2$.

For any integer $m > 1$, we have the identities $(\xi + \sigma)^m \sigma = \xi^m \sigma$.  The change of variables $\xi' = \xi + \sigma$, shows that the graded rings $\Z[\sigma,\xi]/\langle \sigma^2,\xi^{n+1} + d \xi^n \sigma \rangle$ and $\Z[\sigma,\xi']/\langle \sigma^2,{\xi'}^{n+1} + (d - n - 1) {\xi'}^n \sigma \rangle$ are abstractly isomorphic.  Thus, the Chow ring of ${\mathbb F}_{{\bf a}}$ depends only on the value of $d \mod n+1$.

On the other hand, if $d$ and $d'$ are integers that are not congruent mod $n+1$, we can see by explicit comparison that the resulting graded rings are not abstractly isomorphic.  Any graded ring homomorphism is given by $\xi \mapsto a_{11}\xi + a_{12}\sigma$ and $\sigma \mapsto a_{21}\xi + a_{22}\sigma$.  In order to be invertible, we require that the matrix with coefficients $a_{ij}$ lies in $GL_{2}(\Z)$.  The upper triangular elements of $GL_2(\Z)$ induce the isomorphisms just mentioned, and an explicit computation shows that more general elements do not introduce new isomorphisms.


Next, observe that the motivic cohomology ring is an invariant of the unstable $\aone$-homotopy type (see, for example, \cite[\S 2 Theorem 2.2]{VRed}).  Finally, we use the fact that the motivic cohomology ring $\oplus_i H^{2i,i}(X,\Z)$ coincides with the Chow cohomology ring by \cite[Corollary 2]{VCompare}.  Combining this with the computation of the previous paragraph provides the explicit $\aone$-homotopy invariants we required.
\end{proof}

\begin{proof}[Proof of Lemma \ref{lem:mod2}]
This result is now a special case of Proposition \ref{prop:pnbundlesoverpone}.
\end{proof}

\subsubsection*{Ring structures on cohomology}
Suppose $M$ is an $(n-1)$-connected closed $2n$-dimensional manifold.  The cup product on cohomology equips the $\Z$-module $H^n(M,\Z)$ with the structure of quadratic space.  Since the ring structure on cohomology is a homotopy invariant, it follows that this quadratic form is a homotopy invariant of the manifold.  This quadratic form is a fundamental invariant in the classification of $(n-1)$-connected $2n$-folds ({\em cf.} \cite{Wall}).

Analogously, if $X$ is a smooth variety, one may consider the motivic cohomology ring $H^{i,j}(X,\Z)$.  In the course of the proof of Lemma \ref{prop:pnbundlesoverpone}, and thus Lemma \ref{lem:mod2}, we used the fact that the subring $H^{2n,n}(X,\Z)$ of the motivic cohomology ring, i.e., the Chow ring, was an $\aone$-homotopy invariant.  A similar statement can be made for Lemma \ref{lem:movingpoints}.  In either case, we only used the structure of $H^{2,1}(X,\Z)$ as a quadratic space.  Thus, as a corollary to the proof of Theorem \ref{thm:surfaces}, we can deduce the following result.

\begin{prop}
\label{prop:intersectionpairing}
Two rational smooth proper surfaces are $\aone$-weakly equivalent if and only if the quadratic forms on $H^{2,1}(X,\Z)$ are isomorphic.
\end{prop}

\section{Classifying spaces, cohomology and strong $\aone$-invariance}
\label{s:classifyingspaces}
In this section we study $\aone$-local classifying spaces in the Nisnevich and \'etale topologies.  Using these techniques we prove some vanishing theorems for cohomology of smooth $\aone$-connected schemes over a field (see Propositions \ref{prop:etalefundamentalgrouptrivial} and \ref{prop:brauergroup}), provide foundations for further study of \'etale $\aone$-connectivity (Definition \ref{defn:aoneconnected}), and provide an explanation for the ``source" of $\aone$-$h$-cobordisms (Definition \ref{defn:aonehcobordism}) constructed by means of Proposition \ref{prop:bundles}.  Furthermore, ideas from this section will be used in the course of the computations of the $\aone$-fundamental group undertaken in \S \ref{s:fundamentalgroup}.

\subsection{Cohomological properties of smooth $\aone$-connected schemes}
As pointed out to the second author by B. Bhatt, if $k$ is an algebraically closed field having characteristic exponent $p$, and $X$ is an $\aone$-connected scheme over $k$, the cohomological Brauer group of $X$ is a $p$-group (see Proposition \ref{prop:brauergroup}).  Even over $\cplx$, Artin and Mumford ({\em cf.} \cite[Appendix 4.1]{Manin}) constructed examples of conic bundles over $2$-dimensional rational surfaces that are unirational but have non-trivial cohomological Brauer group.  Thus, even over $\cplx$ there exist weakly $\aone$-connected smooth proper varieties that are not $\aone$-connected.  Here, we discuss the structure of some low degree cohomology groups of smooth $\aone$-connected  schemes; we defer the proofs of these results to later in this section.

The next result shows, in particular, that for algebraically closed fields $k$ having characteristic $0$, the \'etale fundamental group of a smooth $\aone$-connected $k$-scheme is trivial (see Proposition \ref{prop:strongfactor} and the subsequent discussion for a proof of a more general result).  Let us emphasize that neither of the next two results require properness assumptions.

\begin{prop}[{\em cf.} {\cite[Remark 3.9]{MICM}}]
\label{prop:etalefundamentalgrouptrivial}
Suppose $k$ is a separably closed field having characteristic exponent $p$.  If $X \in {\mathcal Sm}_k$ is $\aone$-connected, then $X$ admits no non-trivial finite \'etale Galois covers of order coprime to $p$.
\end{prop}

The following result was communicated to us by Bhargav Bhatt.  We provide a slightly different proof of a more general result than the one he suggested (see Proposition \ref{prop:strictfactor}); our proof is very similar to the proof of Proposition \ref{prop:etalefundamentalgrouptrivial}

\begin{prop}[B. Bhatt (private communication)]
\label{prop:brauergroup}
Let $k$ be a separably closed field having characteristic exponent $p$.  If $X \in {\mathcal Sm}_k$ is $\aone$-connected, and $x \in X(k)$ is a base-point, then the Brauer group $Br(X)$ is $p$-torsion.
\end{prop}

\begin{ex}
K3 surfaces over a field $k$ are $\aone$-disconnected because they have non-trivial cohomological Brauer group.  Suppose $k$ is an algebraically closed field having characteristic exponent $p$, $\ell$ is a prime number not equal to $p$, and $X$ is a smooth proper variety over $k$. One can show (see \cite[Theorem 3.1 p. 80]{GrothendieckDix}) that the $\ell$-torsion subgroup of $H^2_{\et}(X,\gm)$ is isomorphic to $(\Q_{\ell}/Z_{\ell})^{b_2 - \rho} \oplus M$ where $b_2$ is the second $\ell$-adic Betti number of $X$, $\rho$ is the rank of the N\'eron-Severi group of $X$, and $M$ is a finite $\ell$-group.
\end{ex}

\begin{extension}
There are various generalizations of Propositions \ref{prop:etalefundamentalgrouptrivial} and \ref{prop:brauergroup} that we do not consider here.  We will see that the techniques used in the proofs of these results are quite robust and one can show that various ``higher unramified invariants" (see, e.g., \cite[\S 4.1]{CTPurity} and \cite{CTO}) vanish for smooth $\aone$-connected schemes; this point of view is developed further in \cite{ABirational}.
\end{extension}

\subsubsection*{The source of $\aone$-$h$-cobordisms}
Suppose $X \in {\mathcal Sm}_k$, and consider the projection morphism
\[
p_X: X \times \aone \longrightarrow X.
\]
Suppose $Y \in {\mathcal Sm}_k$.  By a {\em Nisnevich $Y$-bundle over $X$}, we will mean a Nisnevich locally trivial morphism $g: Z \to X$ with fibers isomorphic to $Y$.  Let $Aut(Y)$ denote the subsheaf of $\underline{Hom}(Y,Y)$ consisting of automorphisms of $Y$; $Aut(Y)$ is in fact a sheaf of groups.  Every Nisnevich $Y$-bundle over $X$, say given by $g$, defines (via \u Cech cohomology) an element $[g]$ of $H^1_{Nis}(X,Aut(Y))$.  The image of $[g]$ under the natural map
\[
p_X^*: H^1_{Nis}(X,Aut(Y)) \longrightarrow H^1_{Nis}(X \times \aone,Aut(Y)).
\]
corresponds in geometric terms to the pull-back via $p_X$, i.e., a morphism $p_X^*(g): Z \times_{X} (X \times \aone) \to X \times \aone$ that is automatically a Nisnevich $Y$-bundle over $X \times \aone$.

To move forward, recall the definition of a torsor.  Note that in general $Aut(Y)$ is only a sheaf of groups, as opposed to a smooth $k$-group scheme, so for technical reasons we use here the definition of torsor given in \cite[p. 127-128]{MV}.  For geometric intuition, it is useful to keep in mind the usual definition of a torsor over a $k$-scheme $X$ under a $k$-group scheme $G$ locally trivial in the \'etale (resp. Zariski, Nisnevich) topology, i.e., a triple $({\mathscr P},f,G)$ consisting of a scheme ${\mathscr P}$ equipped with a scheme-theoretically free right $G$-action, a faithfully flat morphism $f:  {\mathscr P} \to X$ that is equivariant for the trivial $G$-action on $X$, such that locally in the \'etale (resp. Zariski, Nisnevich) topology $f$ is isomorphic to a product.  For compactness of notation, such objects will just be called {\em \'etale (resp. Zariski, Nisnevich) locally trivial $G$-torsors over $X$}.  When $Aut(Y)$ is a smooth $k$-group scheme, the definition of \cite{MV} alluded to above coincides with the one.   We will show that if $Y$ is proper and $k$ has characteristic $0$, then $Aut(Y)$ is actually a group scheme (see Proposition \ref{prop:automorphismgroup}), though contrary to our conventions this scheme need not have finite type over $k$.

In this language, the cocycle $[g]$ represents an isomorphism class of $Aut(Y)$-torsors over $X$.  Any non-trivial $\aone$-$h$-cobordism constructed by means of Proposition \ref{prop:bundles} determines an element of $H^1_{Nis}(X \times \aone,Aut(Y))$ not lying in the image of $p_X^*$.  We would like to study conditions on $G$ under which $p_X^*$ is always a bijection.

\begin{rem}
In the special case $X = \pone$, and $G$ a reductive group, the failure of $p_X^*$ to be a bijection has been studied in great detail by Ramanathan (see \cite{Ramanathan}) building on a classical theorem of  Grothendieck-Harder (see \cite{Harder} Satz 3.1 and 3.4) describing all $G$-torsors over $\pone$.  The discussion of Example \ref{ex:rank2bundles} is a special case of that discussion.
\end{rem}

\subsection{Strong $\aone$-invariance and $\aone$-local classifying spaces}
We now use the notation of Appendix \ref{s:homotopyoverview}.  Henceforth, the term {\em Nisnevich (resp. \'etale) sheaf of groups} will be synonymous with Nisnevich (resp. \'etale) sheaf of groups on ${\mathcal Sm}_k$.

\begin{defn}
\label{defn:aoneinvariance}
Suppose $G$ is a Nisnevich sheaf of groups.  We will say that $G$ is {\em strongly $\aone$-invariant} if for every $U \in {\mathcal Sm}_k$, the canonical maps
\[
p_U^*: H^i_{Nis}(U,G) \longrightarrow H^i_{Nis}(U \times \aone,G)
\]
induced by pullback along the projection $p_U: U \times \aone \to \aone$ are bijections for $i = 0,1$.  Similarly, if $G$ is an \'etale sheaf of groups, we will say that $G$ is {\em strongly $\aone$-invariant in the \'etale topology} if for every $U \in {\mathcal Sm}_k$, the maps
\[
p_U^*: H^i_{\et}(U,G) \longrightarrow H^i_{\et}(U \times \aone,G),
\]
defined as above, are bijections for $i = 0,1$.
\end{defn}

\begin{rem}
Strong $\aone$-invariance was introduced and extensively studied in \cite{MField}.  Many examples of strongly $\aone$-invariant sheaves of groups that are non-commutative will be provided in \S \ref{s:fundamentalgroup}.
\end{rem}

Suppose $G$ is a Nisnevich (resp. \'etale) sheaf of groups.  Let $EG$ denote the \u Cech object associated with the structure morphism $G \to \Spec k$.  Since $G \to \Spec k$ is an epimorphism, the discussion of Example \ref{ex:cechobjects} shows that $EG \to \Spec k$ is a simplicial weak equivalence, i.e., $EG$ is simplicially contractible.  There is an obvious right $G$-action on $EG$, and we let $BG$ denote the Nisnevich (resp. \'etale) sheaf quotient $EG/G$.  The map $EG \to BG$ is a right $G$-torsor in the sense mentioned above; this $G$-torsor is called the universal $G$-torsor.

In \cite[\S 4.1, especially Proposition 1.16]{MV}, the second author and Voevodsky showed that $BG$ classifies $G$-torsors locally trivial in the Nisnevich (resp. \'etale) topology.  More precisely, pullback of the universal $G$-torsor determines a bijection from the set of simplicial homotopy classes of maps from a(n \'etale) space ${\mathcal X}$ to a fibrant model of $BG$ to the set of (\'etale) Nisnevich locally trivial $G$-torsors over ${\mathcal X}$.  Using this result, \cite[\S 2 Proposition 3.19]{MV}, and the existence of fibrant replacements, we observe that if $G$ is a Nisnevich (resp. \'etale) sheaf of groups, then $BG$ is $\aone$-local (resp. \'etale $\aone$-local) if and only if $G$ is strongly $\aone$-invariant (in the \'etale topology).

\begin{notation}
Suppose $G$ is an \'etale sheaf of groups. We set
\[
B_{\et}G := \alpha_* BG^f
\]
where $BG^f$ is an \'etale simplicially fibrant replacement for $BG$.
\end{notation}

\begin{lem}
\label{lem:comparison1}
If $G$ is an \'etale sheaf of groups, then $G$ is strongly $\aone$-invariant in the \'etale topology if and only if $B_{\et}G$ is $\aone$-local.  Thus, if $G$ is strongly $\aone$-invariant in the \'etale topology, for every $U \in {\mathcal Sm}_k$ the canonical maps
\[
[\Sigma^i_s \wedge U_+,(B_{\et}G,\ast)]_{\aone} \longrightarrow H^{1-i}_{\et}(U,G)
\]
induced by adjunction are bijections for $i = 0,1$.
\end{lem}

\begin{proof}
By Lemma \ref{lem:aonefunctoriality} the functor ${\bf R}\alpha_*$ preserves $\aone$-local objects.  Now $G$ is strongly $\aone$-invariant in the \'etale topology if and only if $BG$ is $\aone$-local in the \'etale topology.  If $BG$ is $\aone$-local in the \'etale topology, then $BG^f$ is $\aone$-local in the \'etale topology and thus $B_{\et}G$ is $\aone$-local.

If $B_{\et}G$ is $\aone$-local we conclude that for any pointed space $({\mathcal X},x)$ the canonical map
\[
[({\mathcal X},x),(B_{\et}G,\ast)]_s \longrightarrow [({\mathcal X},x),(B_{\et}G,\ast)]_{\aone}
\]
is a bijection.  If $({\mathcal X},x)$ is a pointed \'etale simplicial sheaf, adjunction gives a canonical bijection
\[
Hom_{\hspet}(({\mathcal X},x),(BG,\ast)) \isomto [({\mathcal X},x),(B_{\et}G,\ast)]_s.
\]
The final statement follows immediately from this by applying \cite[\S 4 Proposition 1.16]{MV}.
\end{proof}

We quote without proof the following result of the second author, which we will use frequently in the sequel.   Together with Theorem \ref{thm:strongabelianimpliesstrict}, this result provides the backbone for all structural results regarding $\aone$-homotopy groups.

\begin{thm}[{\cite[Theorem 5.1]{MField}}]
\label{thm:strongaoneinvaranceofpi1}
If $({\mathcal X},x)$ is a pointed space, then $\pi_i^{\aone}({\mathcal X},x)$ is a strongly $\aone$-invariant sheaf of groups for any integer $i > 0$.
\end{thm}

\begin{cor}
\label{cor:etaleimpliesNisnevich}
If $G$ is an \'etale sheaf of groups that is strongly $\aone$-invariant in the \'etale topology, then the Nisnevich sheaf underlying $G$ is strongly $\aone$-invariant.
\end{cor}

\begin{proof}
If $B_{\et}G$ is $\aone$-local, the map $\pi_1^s(B_{\et}G,\ast) \to \pi_1^{\aone}(B_{\et}G,\ast)$ is an isomorphism.  Then, apply the previous theorem together with the identification of $\pi_1^s(B_{\et}G)$ with $G$ itself using Lemma \ref{lem:comparison1}.
\end{proof}

\subsubsection*{The category of strongly $\aone$-invariant sheaves of groups}
Let ${\mathcal Gr}_k$ denote the category of Nisnevich sheaves of groups on $\Sm_k$.  We write ${\mathcal Gr}^{\aone}_k$ for the full subcategory of ${\mathcal Gr}_k$ consisting of strongly $\aone$-invariant sheaves of groups.

\begin{lem}[{\em cf.} {\cite[Remark 6.11]{MField}}]
The category ${\mathcal Gr}^{\aone}_k$ admits small colimits.
\end{lem}

\begin{proof}
We claim that the inclusion functor ${\mathcal Gr}^{\aone}_k \hookrightarrow {\mathcal Gr}_k$ admits a left adjoint defined by the functor $G \mapsto \pi_1^{\aone}(BG,\ast)$.  We have maps
\[
Hom_{{\mathcal Gr}_k}(H,G) \longleftarrow [(BH,\ast),(BG,\ast)]_s \longrightarrow
[(BH,\ast),(BG,\ast)]_{\aone} \longrightarrow Hom_{{\mathcal Gr}^{\aone}_k}(\pi_1^{\aone}(BH,\ast),G)
\]
where the left-most map is given by applying the functor $\pi_1^s$, and the right-most map is given by applying the functor $\pi_1^{\aone}$.  The map in the middle is a bijection since $G$ is strongly $\aone$-invariant, and the Postnikov tower ({\em cf.} \cite[3.10.1]{ADExcision}) can be used to show that both the left-most and right-most maps are bijections; this  observation establishes adjointness.

Now, any functor that {\em is} a left adjoint preserves small colimits (see \cite[V.5]{MacLane}).  The category of presheaves of groups on $\Sm_k$ admits small colimits (defined sectionwise).  Since sheafification is a left adjoint, it follows that ${\mathcal Gr}_k$ admits all small colimits.  Finally, using the fact that the functor $H \mapsto \pi_1^{\aone}(BH)$ is a left adjoint, we deduce that ${\mathcal Gr}^{\aone}_k$ admits all small colimits.
\end{proof}

\begin{defn}
\label{defn:amalgamatedsum}
Given a diagram of strongly $\aone$-invariant sheaves of groups of the form
\[
G_1 \longleftarrow H \longrightarrow G_2,
\]
we write $G_1 \star_H^{\aone} G_2$ for the colimit of this diagram computed in ${\mathcal Gr}^{\aone}_k$.  Precisely, $G_1 \star_H^{\aone} G_2$ is the strongly $\aone$-invariant sheaf of groups $\pi_1^{\aone}(B(G_1 \star_H G_2))$, where $G_1 \star_H G_2$ is the pushout computed in the category ${\mathcal Gr}_k$.  We refer to $G_1 \star_H^{\aone} G_2$ as the {\em sum of $G_1$ and $G_2$ amalgamated over $H$}, or, if $H$ is trivial, as the {\em amalgamated sum of $G_1$ and $G_2$} (where in both cases strong $\aone$-invariance is understood).
\end{defn}

\begin{defn}
\label{defn:freestronglyaoneinvariant}
The {\em free strongly $\aone$-invariant sheaf of groups on a (pointed) sheaf of sets $({\mathcal S},s)$}, denoted $F_{\aone}({\mathcal S})$, is the Nisnevich sheaf of groups $\pi_1^{\aone}(\Sigma^1_s {\mathcal S})$.
\end{defn}

One can show (\cite[Lemma 6.23]{MField}) that if ${\mathcal S}$ is a sheaf of pointed sets, then for any strongly $\aone$-invariant sheaf of groups $G$, the canonical map ${\mathcal S} \to \Omega^1_s \Sigma^1_s {\mathcal S}$ induces a bijection
\[
Hom_{{\mathcal Gr}^{\aone}_k}(F_{\aone}({\mathcal S}),G) \isomto Hom_{\Spc_{k,\bullet}}({\mathcal S},G).
\]
Thus, $F_{\aone}$ is left adjoint to the forgetful functor ${\mathcal Gr}^{\aone}_k \to {\mathcal Spc}_{k,\bullet}$, and this observation justifies our naming convention.

\subsubsection*{Proof of Proposition \ref{prop:etalefundamentalgrouptrivial}}
We first prove the following more precise result.

\begin{prop}
\label{prop:strongfactor}
Suppose $X \in {\mathcal Sm}_k$ is $\aone$-connected, and $G$ is an \'etale sheaf of groups strongly $\aone$-invariant in the \'etale topology.  For any two points $x_1,x_2 \in X(k)$ the restriction maps
\[
x_1^*, x_2^*: H^1_{\et}(X,G) \longrightarrow H^1_{\et}(\Spec k,G)
\]
coincide, and we denote $\rho$ the map induced by any choice of point.  The natural map
\[
H^1_{Nis}(X,G) \longrightarrow H^1_{\et}(X,G)
\]
injects into the inverse image under $\rho$ of the base-point of the pointed set $H^1_{\et}(\Spec k,G)$.  In other words, an \'etale locally trivial $G$-torsor over $X$ whose restriction to a rational point is trivial is Nisnevich locally trivial.
\end{prop}

\begin{proof}
From Lemma \ref{lem:comparison1} we know that for any $U \in {\mathcal Sm}_k$ the canonical map
\[
[U,B_{\et}G]_{\aone} \longrightarrow H^1_{\et}(U,G)
\]
is a bijection.  Given a class $\tau \in H^1_{\et}(X,G)$, choose an explicit representative $\tau: X \to B_{\et}G$ (we can do this because $B_{\et}G$ is $\aone$-fibrant).  The composite map
\[
X \stackrel{\tau}{\longrightarrow} B_{\et}G \longrightarrow \pi_0^{\aone}(B_{\et}G)
\]
factors through the canonical map $X \to \pi_0^{\aone}(X) = \ast$.  As the set of sections of $\pi_0^{\aone}(B_{\et}G)$ over $\Spec k$ is exactly $H^1_{\et}(\Spec k,G)$, this proves the independence statement.

Now, again by Lemma \ref{lem:comparison1} we know that $B_{\et}G$ is $\aone$-local and that $\pi_1^{\aone}(B_{et}G) = G$ (we are abusing notation and writing $G$ for the restriction of the \'etale sheaf $G$ to the Nisnevich topology).  Write $BG$ be the classifying space of the Nisnevich sheaf of groups just mentioned; this space is $\aone$-$0$-connected by \cite[\S 2 Corollary 3.22]{MV}.  The (pointed) map $BG \to B_{\et}G$ given by adjunction induces an inclusion $* \to \pi_0^{\aone}(B_{\et}G)$ and isomorphisms $\pi_i^{\aone}(BG) \to \pi_i^{\aone}(B_{\et}G)$: we will say that this map is the inclusion of the $\aone$-connected component of the base-point.  If $X \in {\mathcal Sm}_k$ is $\aone$-connected, it follows that the induced map
\[
H^1_{Nis}(X,G) \isomto [X,BG]_{\aone} \longrightarrow [X,B_{et}G]
\]
is an injection whose image can be identified with the set of morphisms $X \to B_{\et}G$ that map $\pi_0^{\aone}(X)$ to the base-point of $\pi_0^{\aone}(B_{\et}G)$.  By the discussion of the previous paragraphs, this proves our claim.
\end{proof}

\begin{proof}[Proof of \ref{prop:etalefundamentalgrouptrivial}]
Let $G$ be a finite \'etale group scheme of order prime to $p$.  In this situation, $B_{\et}G$ is $\aone$-local by \cite[\S 4 Proposition 3.1]{MV}, so we could just apply Proposition \ref{prop:strongfactor}.  Really, we just have to observe that $G$ is a strongly $\aone$-invariant sheaf of groups in the \'etale topology by \cite[Expos\'e XV Corollaire 2.2]{SGA43}.
\end{proof}

\begin{rem}
We continue with notation as in Proposition \ref{prop:strongfactor}.  The map $\rho: H^1_{\et}(X,G) \to H^1_{\et}(\Spec k,G)$ can be reinterpreted as follows.  Recall the identification $H^1_{\et}(X,G) := [X,B_{\et}G]_{\aone}$.   Since $X$ is $\aone$-connected, ``evaluation on $\pi_0^{\aone}$" gives a map
\[
[X,B_{\et}G]_{\aone} \longrightarrow Hom_{\Spc_k}(\pi_0^{\aone}(X),B_{\et}G) \isomto Hom_{\Spc_k}(\ast,B_{\et}G) \isomto [\Spec k,B_{\et}G]_{\aone}.
\]
that coincides with $\rho$.
\end{rem}

\begin{rem}
Given a $1$-cocycle of $k$ with values in $G$ associated with a class $\tau \in H^1(\Spec k,G)$, one may twist $G$ by $\tau$ to get another sheaf of groups that we denote by $G_\tau$.  Using a similar but more involved argument, one can prove that the sheaf $G_\tau$ is also strongly $\aone$-invariant in the \'etale topology, and the fiber of $\rho$ at $\tau$ is (the image of) $H^1_{Nis}(X,G_\tau)$.
\end{rem}

\subsection{Strict $\aone$-invariance and $\aone$-local Eilenberg-MacLane spaces}
There are versions of the results proved above for higher cohomology of sheaves of abelian groups; we give here the corresponding statements together with brief indications of the modifications required in the proofs.  For any Nisnevich (resp. \'etale) sheaf of abelian groups $A$, one can define Eilenberg-MacLane spaces $K(A,i)$ such that, if $U \in {\mathcal Sm}_k$, the $H^i_{Nis}(U,A)$ (resp. $H^i_{\et}(U,A)$) can be computed in terms of homotopy classes of maps from $U$ to $K(A,i)$ in $\hsnis$ (resp. $\hset$).  See \cite[pp. 55-60]{MV} for more details.

\begin{defn}
Suppose $A$ is a Nisnevich sheaf of abelian groups.  We will say that $A$ is {\em strictly $\aone$-invariant} if for every $U \in {\mathcal Sm}_k$ the pull-back map
\[
H^i_{Nis}(U,A) \longrightarrow H^i_{Nis}(U \times \aone,A)
\]
is a bijection for every $i \geq 0$.  Similarly, given an \'etale sheaf of abelian groups, we will say that $A$ is {\em strictly $\aone$-invariant in the \'etale topology} if for every $U \in {\mathcal Sm}_k$ the pull-back map
\[
H^i_{\et}(U,A) \longrightarrow H^i_{\et}(U \times \aone,A)
\]
is a bijection for every $i \geq 0$.
\end{defn}

Combining  \cite[\S 2 Proposition 1.26]{MV}, {\em ibid.} \S 2 Theorem 1.34, and {\em ibid.} \S 2 Proposition 3.19, we observe that if $A$ is a Nisnevich (resp. \'etale) sheaf of abelian groups, then $A$ is strictly $\aone$-invariant (resp. for the \'etale topology) if and only if $K(A,i)$ is $\aone$-local for every $i \geq 0$.

\begin{notation}
Suppose $A$ is an \'etale sheaf of abelian groups.  We set
\[
K_{\et}(A,i) := \alpha_*K(A,i)^f
\]
where $K(A,i)^f$ is an \'etale simplicially fibrant replacement of $K(A,i)$.
\end{notation}

The proof of the following result is essentially identical to the proof of Lemma \ref{lem:comparison1}.

\begin{lem}
\label{lem:comparisonhigher}
Suppose $A$ is an \'etale sheaf of abelian groups, then $A$ is strictly $\aone$-invariant in the \'etale topology if and only if $K_{\et}(A,i)$ is $\aone$-local.  Thus, if $A$ is strictly $\aone$-invariant in the \'etale topology, for every $U \in {\mathcal Sm}_k$ the canonical maps
\[
[\Sigma^j \wedge U_+,K_{\et}(A,i)]_{\aone} \longrightarrow H^{i-j}_{\et}(U,A)
\]
induced by adjunction are bijections for $0 \leq j \leq i$.
\end{lem}

\begin{thm}[{\cite[Theorem 4.47]{MField}}]
\label{thm:strongabelianimpliesstrict}
If $A$ is a strongly $\aone$-invariant sheaf of abelian groups, then $A$ is strictly $\aone$-invariant.
\end{thm}

Using this theorem, the proof of the next result is very similar to the proof of Corollary \ref{cor:etaleimpliesNisnevich}.

\begin{cor}
\label{cor:etaleimpliesNisnevichstrictly}
If $A$ is an \'etale sheaf of groups that is strictly $\aone$-invariant for the \'etale topology, then the Nisnevich sheaf underlying $A$ is strictly $\aone$-invariant.
\end{cor}

\subsubsection*{Deducing $\aone$-invariance properties}
\begin{defn}
\label{defn:aoneinvariant}
Recall that a presheaf (resp. sheaf, or \'etale sheaf) of sets ${\mathcal S}$ is said to be {\em $\aone$-invariant}, if for every $U \in {\mathcal Sm}_k$, the canonical map
\[
{\mathcal S}(U) \longrightarrow {\mathcal S}(U \times \aone)
\]
induced by pull-back along the projection $U \times \aone \to U$ is a bijection.
\end{defn}

The following result gives a way to construct \'etale sheaves of abelian groups that are strictly $\aone$-invariant in the \'etale topology.

\begin{lem}[{\em cf.} {\cite[\S 3.4]{VCohTh}}]
\label{lem:etalehomotopyinvarianceandtransfers}
Let $k$ be a field having characteristic exponent $p$, and suppose $A$ is an $\aone$-invariant \'etale sheaf of $\Z[1/p]$-modules with transfers (in the sense of \textup{\cite[\S 6 p. 39]{MVW}}), then $A$ is strictly $\aone$-invariant for the \'etale topology.
\end{lem}

\begin{proof}
Given a short exact sequence of \'etale sheaves of $\Z[1/p]$-modules
\[
0 \longrightarrow A' \longrightarrow A \longrightarrow A'' \longrightarrow 0,
\]
the associated long exact sequence in cohomology shows that if any two of the three sheaves are strictly $\aone$-invariant in the \'etale topology, then the third must be as well.  Using the exact sequence of \'etale sheaves
\[
0 \longrightarrow A_{tors} \longrightarrow A \longrightarrow A \tensor \Q \longrightarrow A \tensor \Q/\Z \longrightarrow 0.
\]
one reduces to treating the cases where $A$ is an \'etale sheaf of $\Q$-vector spaces or, using the assumptions, $A$ is an \'etale sheaf of torsion prime to $p$.  In the first case, one reduces to \cite[Theorem 13.8]{MVW} by using the fact the Nisnevich and \'etale cohomology coincide (see \cite[Proposition 14.23]{MVW}).  In the second case, the Suslin rigidity theorem (see \cite[Theorem 7.20]{MVW}) shows that $A$ is in fact a locally constant \'etale sheaf of groups and one concludes by applying \cite[Expos\'e XV Corollaire 2.2]{SGA43}.
\end{proof}

\subsubsection*{Proof of Proposition \ref{prop:brauergroup}}
Suppose $k$ is a field having characteristic exponent $p$.  Let $\gm'$ denote the \'etale sheaf whose sections over $U \in {\mathcal Sm}_k$ are given by
\[
U \longmapsto \O^*(U) \tensor_{\Z} \Z[1/p].
\]
We refer to $\gm'$ as the multiplicative group with characteristic exponent inverted.  This \'etale sheaf of groups is just $\gm$ if $k$ has characteristic $0$.  By \cite[p. 48 Proposition 1.4]{GrothendieckDix} we know that there is a canonical injection $Br(X) \hookrightarrow H^2_{\et}(X,\gm)$.  We can conclude that the map $Br(X) \to H^2_{\et}(X,\gm')$ induces an injection on $\ell$-torsion subgroups for $\ell$ prime to $p$.

\begin{prop}
\label{prop:strictfactor}
Let $k$ be a field having characteristic exponent $p$, suppose $X \in {\mathcal Sm}_k$ is $\aone$-connected, and $x \in X(k)$.  The structure morphism $X \to \Spec k$ induces an isomorphism
\[
H^2_{\et}(\Spec k,\gm') \longrightarrow H^2_{\et}(X,\gm').
\]
In particular, if $k$ is separably closed, then $Br(X)$ is $p$-torsion.
\end{prop}

\begin{proof}
In outline, this proof is essentially identical to the proof of Proposition \ref{prop:etalefundamentalgrouptrivial} via Proposition \ref{prop:strongfactor}.  In this case, we use Lemma \ref{lem:comparisonhigher} to reduce to showing that $\gm'$ is strictly $\aone$-invariant in the \'etale topology; this latter fact follows from Lemma \ref{lem:etalehomotopyinvarianceandtransfers}.  Indeed, $\gm'$ is an \'etale sheaf of $\Z[1/p]$-modules, has transfers given by the ``norm" map, and is $\aone$-invariant since $\gm$ itself is $\aone$-invariant.  Then $\pi_i^{\aone}(K_{\et}(\gm',2))$ is $0$ for $i \geq 3$, and is the Nisnevich sheaf associated with the presheaf $U \mapsto H^{2-i}_{\et}(U,\gm')$ for $0 \leq i \leq 2$.  We can use Grothendieck's version of Hilbert's Theorem 90 to show that $\pi_1^{\aone}(K_{\et}(\gm',2))$ is trivial.
\end{proof}

\begin{ex}
Suppose $G$ is a simply connected, semi-simple algebraic group over a field $k$ having characteristic exponent $p$.  Let $G^+(k)$ denote the subgroup of $G(k)$ generated by the images of homomorphisms from the additive group $\ga(k)$.  The quotient $G(k)/G^+(k)$ is called the Whitehead group of $G$, often denoted $W(k,G)$.  The Kneser-Tits problem asks for which groups $W(k,G) = 1$.  Slightly more generally, \cite[Question 1.1]{PGille} asks whether one can characterize groups such that $W(L,G)$ is trivial for every extension $L/k$; such groups are called $W$-trivial.  If $G$ is $W$-trivial, it is $\aone$-chain connected and thus $\aone$-connected by Proposition \ref{prop:aoneconnected}.  If $G$ is in addition {\em split} then \cite[Proposition 3.1]{PGille} shows that $G$ is $W$-trivial and hence $\aone$-connected.  Thus, Proposition \ref{prop:strictfactor} shows that the Brauer group of a $W$-trivial group is $p$-torsion.  S. Gille uses related ideas to study the Brauer group of general simply connected, semi-simple algebraic groups $G$ (see \cite{SGille}).
\end{ex}

\subsection{Algebraic groups and strong $\aone$-invariance}
We now study the subcategory of ${\mathcal Gr}^{\aone}_k$ consisting of representable objects, i.e., smooth group schemes having finite type over $k$.  Throughout this section, if $G$ is a smooth $k$-group scheme, we denote by $G^0$ the connected component of $G$ containing the identity element in the sense of algebraic groups.

\begin{lem}
\label{lem:linearalgebraiccase}
Let $k$ be a perfect field, and suppose $G$ is a smooth affine algebraic $k$-group.  The sheaf of groups $G$ is $\aone$-invariant if and only if $G^0$ is a $k$-torus.
\end{lem}

\begin{proof}
We use the following {\em d\'evissage}.  There is an exact sequence of algebraic groups
\[
1 \longrightarrow G^0 \longrightarrow G \longrightarrow \Gamma \longrightarrow 1
\]
where $\Gamma$ is the (finite) group of connected components.  Since the group $\Gamma$ is strongly $\aone$-invariant by \cite[\S 4 Proposition 3.5]{MV}, proving the statement for $G$ is equivalent to proving it for $G^0$.  Thus, we assume $G$ is connected.

As $k$ is perfect, the unipotent radical $R_u$ of $G$ is a smooth unipotent $k$-group scheme.  Therefore, $G$ fits into an exact sequence of the form
\[
1 \longrightarrow R_u \longrightarrow G \longrightarrow G^{red} \longrightarrow 1,
\]
where $G^{red}$ is a reductive $k$-group scheme.  Since $R_u$ is connected and smooth, and $k$ is perfect, by a theorem of Lazard \cite[Chapter IV \S 2.3.9]{DemazureGabriel} $R_u$ is split, i.e., admits an increasing sequence of normal subgroups with subquotients isomorphic to $\ga$.  Thus, if $R_u$ is non-trivial, it possesses a non-trivial group homomorphism from $\ga$.  Since $\ga$ is not $\aone$-invariant, it follows in this case that $R_u$ is not $\aone$-invariant either.  Thus, for $G$ to be $\aone$-invariant $R_u$ must be trivial, and we may assume $G$ is reductive.

If $G$ is reductive, we have an exact sequence of the form
\[
1 \longrightarrow R(G) \longrightarrow G \longrightarrow G^{ss} \longrightarrow 1,
\]
where $R(G)$ is a $k$-torus.  Now, since $R(G)$ is a $k$-torus, it splits over a finite separable extension $L/k$.  By \'etale descent, it follows that $R(G)$ is $\aone$-rigid.  Thus, $G$ is $\aone$-invariant if and only if $G^{ss}$ is $\aone$-invariant, so we can assume $G = G^{ss}$.

If $G$ is a (non-trivial) semi-simple group, then it splits over a finite separable extension $L/k$.  Passing to such an extension, we obtain non-trivial morphisms from $\aone_L$ (any root subgroup provides such a morphism), and thus $G$ is not $\aone$-invariant.
\end{proof}

\begin{prop}
\label{prop:commutativeinvariant}
Suppose $k$ is a perfect field, and assume $G$ is a smooth $k$-group scheme.  The sheaf $G$ is $\aone$-invariant if and only if $G^0$ is an extension of an abelian variety by a $k$-torus.
\end{prop}

\begin{proof}
By Chevalley's theorem \cite{Conrad}, there is a canonical extension of the form
\[
1 \longrightarrow G^{aff} \longrightarrow G \longrightarrow A \longrightarrow 1,
\]
where $G^{aff}$ is a normal, smooth closed affine algebraic group, and $A$ is an abelian variety.  Since $A$ is $\aone$-rigid ({\em cf.} Example \ref{ex:aonerigid}), and the underlying Nisnevich sheaf is flasque, we conclude that $A$ is strongly $\aone$-invariant.  Thus, proving the result for $G$ is equivalent to proving the result for $G^{aff}$.  Since $G^{aff}$ is a smooth affine algebraic $k$-group, we apply Lemma \ref{lem:linearalgebraiccase} to finish the proof.
\end{proof}

\begin{prop}
\label{prop:representablestronglyinvariant}
Assume $k$ is a field having characteristic $0$, and suppose $G$ is a smooth $k$-group scheme.  The \'etale sheaf $G$ is strongly $\aone$-invariant in the \'etale topology if and only if $G^0$ is an extension of an abelian variety by a $k$-torus.  If one of these equivalent conditions holds, then $G$ is strongly $\aone$-invariant in the Nisnevich topology as well.
\end{prop}

\begin{proof}
By Proposition \ref{prop:commutativeinvariant}, we know that $G$ is $\aone$-invariant if and only if $G^0$ is an extension of an abelian variety by a $k$-torus.  In this case, applying Lemma \ref{lem:transfers}, we conclude that $G^0$ is an \'etale sheaf with transfers in the sense of \cite[\S 6 p. 39]{MVW}.  Then, since $k$ has characteristic $0$, we may apply Lemma \ref{lem:etalehomotopyinvarianceandtransfers} to conclude that $G^0$ is in fact strongly $\aone$-invariant in the \'etale topology.  Also, since $k$ has characteristic $0$, we know that finite groups are strongly $\aone$-invariant in the \'etale topology.  The last statement follows from the equivalences by applying Corollary \ref{cor:etaleimpliesNisnevich}.
\end{proof}

\begin{lem}[{\cite[Lemme 3.1.2]{Orgogozo}}]
\label{lem:transfers}
If $S$ is a smooth commutative $k$-group scheme, then the \'etale sheaf underlying $S$ can be equipped canonically with transfers (in the sense of \textup{\cite[\S 6 p. 39]{MVW}}).
\end{lem}

\begin{ex}
For fields having positive characteristic, the sheaf $\gm$ is not strictly $\aone$-invariant in the \'etale topology.  Thus, $\gm$ is an \'etale sheaf whose underlying Nisnevich sheaf is strictly $\aone$-invariant, but which is not strictly $\aone$-invariant in the \'etale topology in general.  In other words, the converse to Corollary \ref{cor:etaleimpliesNisnevichstrictly} is false.  However, let us note that $\gm$ {\em is} strongly $\aone$-invariant in the \'etale topology because of Hilbert's theorem 90, i.e., $H^1_{Nis}(X,\gm) = H^1_{\et}(X,\gm)$ for any $X \in \Sm_k$.

With more work, one can construct counter-examples to the converse of Corollary \ref{cor:etaleimpliesNisnevich} even for fields having characteristic $0$.  If $C$ is a smooth curve of genus $g > 0$, then one can consider the free strongly invariant sheaf of abelian groups generated by $C$, often denoted $\Z_{\aone}(C)$ ({\em cf.} \cite[p.193]{MField}), in a manner similar to Definition \ref{defn:freestronglyaoneinvariant}.  This sheaf of groups is actually an \'etale sheaf of groups that is not strongly $\aone$-invariant in the \'etale topology.
\end{ex}

\begin{rem}
Strong $\aone$-invariance (or its failure) for $GL_n$ has been studied in great detail.  On the one hand, in \cite[\S 7]{MField} it is proven that if $X$ is a smooth {\em affine} scheme, then $[X,BGL_n]_{\aone}$ is in canonical bijection with the set of isomorphism classes of rank $n$ vector bundles on $X$ whenever $n \neq 2$.  On the other hand if $X$ is {\em not} affine, the examples of \cite{ADBundle} show there is essentially no ``lower bound" on how badly this identification can fail.
\end{rem}

\subsubsection*{Automorphism groups of smooth proper varieties}
The automorphism groups of smooth proper $k$-varieties form a quite restricted class.  Combining the next result with Proposition \ref{prop:representablestronglyinvariant}, we obtain an essentially complete understanding of $\aone$-$h$-cobordisms constructed by means of Proposition \ref{prop:bundles}.

\begin{prop}
\label{prop:automorphismgroup}
Suppose $k$ is a field having characteristic $0$.  If $X \in {\mathcal Sm}_k$ is also proper, then $Aut(X)^0$ is a smooth $k$-group scheme.
\end{prop}

\begin{proof}
For this proof, we drop any finite-type assumptions in our conventions for schemes.  An automorphism of a scheme $X$ is a morphism $f: X \to X$.  Such a morphism defines a graph $\Gamma_f \subset X \times X$.  By this construction, we can identify the functor defining $Aut(X)$ as a sub-functor of an appropriate Hilbert functor.  In the case $X$ is projective, representability of this functor follows from \cite[Theorem I.1.10 ({\em cf.} Exercise I.1.10.2)]{Kollar}.  If $X$ is only a proper scheme, then the sheaf $Aut(X)$ is represented by an {\em algebraic space} by \cite[Lemma 5.1]{Olsson}.  By \cite[II.6.7]{Knutson}, any algebraic space has a dense open affine subscheme and one can use the group action to construct a Zariski open cover of $Aut(X)$ by such schemes.  Thus, $Aut(X)$ is always a $k$-group scheme under the hypotheses.  Finally, over fields having characteristic $0$, one can show that the group scheme $Aut(X)$ is actually smooth by explicit computation of its tangent space and application of Cartier's theorem, \cite[Chapter II \S 6.1.1]{DemazureGabriel}.
\end{proof}

\section{Computing the $\aone$-fundamental group}
\label{s:fundamentalgroup}
This section is the spiritual center of the paper.  Given an $\aone$-connected space ${\mathcal X}$, it is natural to study its higher $\aone$-homotopy invariants.  If $k$ is a field, Theorem \ref{thm:stablyrational} showed that $k$-rational smooth proper surfaces are $\aone$-connected, and if furthermore $k$ has characteristic $0$, then retract $k$-rational smooth proper $k$-varieties are $\aone$-connected.  See Appendix \ref{s:notationalpostscript} for a summary of results relating $\aone$-connectivity and rationality properties.  Thus, let $X$ be an $\aone$-connected smooth variety, and fix a basepoint $x \in X(k)$.  We focus now on computing the next $\aone$-homotopy invariant of such varieties: the $\aone$-fundamental sheaf of groups, denoted $\pi_1^{\aone}(X,x)$.

To facilitate topological intuition, throughout this section we refer to $\pi_1^{\aone}(X,x)$ as simply the {\em $\aone$-fundamental group}.  To partially justify this abuse of terminology, we begin by proving or recalling a collection of results that are analogous to corresponding topological statements.  Theorem \ref{thm:aonevanKampen} establishes a version of the van Kampen theorem, and Proposition \ref{prop:torustorsor} indicates a relationship between the covering space theory associated with the $\aone$-fundamental group and geometry.  We also discuss in great detail the structure of the $\aone$-fundamental group of $\pone$, which is of fundamental importance in unstable $\aone$-homotopy theory.  Unlike its topological counterpart, Proposition \ref{prop:nontrivialfundamentalgroup} shows that the $\aone$-fundamental group of a smooth proper $\aone$-connected variety is {\em always} non-trivial; Proposition \ref{prop:etaleaonefundamentalgroupnontrivial} explains a corresponding result for \'etale-$\aone$-connected varieties.

The ultimate goal of this section, accomplished in Corollary \ref{thm:surfacesfundamentalgroup}, is to show that if $X$ is a $k$-rational smooth proper surface over an algebraically closed field $k$, the $\aone$-homotopy type of $X$ is determined by its $\aone$-fundamental group.  To establish this, we will simply compute the $\aone$-fundamental groups of all $k$-rational smooth proper surfaces.  Theorem \ref{thm:surfaces} shows that we need only perform the computation for Hirzebruch surfaces, and for certain blow-ups of points.   The first case is addressed by Proposition \ref{prop:hirzebruchfundamental}, and Proposition \ref{prop:fundamentalgroupblowup} addresses the second case by establishing a general ``reduction theorem" for blow-ups of points on smooth schemes that are covered by affine spaces in the sense of Definition \ref{defn:combinatorial}.

\subsection{Generalities on the $\aone$-fundamental group}
Suppose $X \in {\mathcal Sm}_k$, choose a basepoint $x \in X(k)$, and consider the $\aone$-fundamental group $\pi_1^{\aone}(X,x)$.  If furthermore $X$ is $\aone$-connected, and we pick another base-point $x' \in X(k)$, the $\aone$-fundamental group $\pi_1^{\aone}(X,x')$ is conjugate to $\pi_1^{\aone}(X,x)$.  For this reason, we fix and (occasionally) suppress basepoints in all our subsequent discussion.  In Definition \ref{defn:aoneinvariance} we recalled the notion of a strongly $\aone$-invariant sheaf of groups.  We also noted \cite[Theorem 5.1]{MField} shows that $\pi_1^{\aone}(X,x)$ is a strongly $\aone$-invariant sheaf of groups.  Here is a version of the classical van Kampen theorem (more general versions are known).

\begin{thm}[$\aone$-van Kampen theorem {\cite[Theorem 6.12]{MField}}]
\label{thm:aonevanKampen}
Suppose $X$ is a smooth $\aone$-connected $k$-variety covered by $\aone$-connected open subsets $U,V$ such that $U \cap V$ is $\aone$-connected.  Then we have a canonical isomorphism
\[
\pi_1^{\aone}(U) \star^{\aone}_{\pi_1^{\aone}(U \cap V)} \pi_1^{\aone}(V) \isomto \pi_1^{\aone}(X),
\]
where the operation $\star^{\aone}$ is given by \textup{Definition \ref{defn:amalgamatedsum}}.
\end{thm}

A version of covering space theory for the $\aone$-fundamental group dubbed $\aone$-covering space theory has been developed by the second author (see \cite[\S 6.1]{MField}).  For our purposes, the following result will suffice.

\begin{prop}[{\em cf.} {\cite[Corollary 5.3]{ADExcision}}]
\label{prop:torustorsor}
Suppose $\tilde{X}$ and $X$ are two smooth $\aone$-connected  $k$-varieties.  If $f: \tilde{X} \to X$ is a $\gm^{\times r}$-torsor over $X$, then the morphism $f$ is an $\aone$-fibration, one has a short exact sequence of the form
\[
1 \longrightarrow \pi_1^{\aone}(\tilde{X}) \longrightarrow \pi_1^{\aone}(X) \longrightarrow \gm^{\times r} \longrightarrow 1,
\]
and isomorphisms $\pi_i^{\aone}(\tilde{X}) \isomt \pi_i^{\aone}(X)$ for every $i > 1$.
\end{prop}

\begin{rem}
Slightly more generally, one can show that Zariski locally trivial torsors with $\aone$-rigid fibers are always $\aone$-fibrations.  Torsors under split tori over smooth schemes are examples of $\aone$-covering spaces in the sense of \cite[\S 6.1]{MField} by {\em ibid.} Lemma 4.5.  This fact has been used in \cite{ADExcision} and \cite{Wendt} to describe the $\aone$-fundamental group of a smooth proper toric variety as an extension of a torus by a strongly $\aone$-invariant sheaf of groups of arithmetic nature.
\end{rem}

The main problem with Proposition \ref{prop:torustorsor} is that it does not provide an explicit identification of the extension or the group structure on the $\aone$-fundamental group.  The problem of identifying this additional data, which Proposition \ref{prop:hirzebruchfundamental} shows to be very subtle,  will occupy us in what follows.  Nevertheless, we can use the geometry behind Proposition \ref{prop:torustorsor} to establish that smooth proper $\aone$-connected schemes necessarily have non-trivial $\aone$-fundamental groups.

\begin{prop}
\label{prop:nontrivialfundamentalgroup}
Suppose $X \in {\mathcal Sm}_k$ is $\aone$-connected and $x \in X(k)$.  We have a canonical isomorphism
\[
Hom_{{\mathcal Gr}^{\aone}_k}(\pi_1^{\aone}(X,x),\gm) \isomto Pic(X).
\]
In particular, if $X$ is a strictly positive dimensional, $\aone$-connected, smooth proper $k$-variety then $Pic(X)$ is non-trivial and thus $\pi_1^{\aone}(X)$ is non-trivial.
\end{prop}

\begin{proof}
This proof requires use of the Postnikov tower; the identification claimed in the first statement is mentioned in \cite[Lemma B.7]{MField} or \cite[Theorem 3.30]{ADExcision}.  Here is a quick sketch of the idea.  One can show that the functor $\pi_1^{\aone}(\cdot)$ induces a canonical map $[(X,x),(B\gm,\ast)]_{\aone} \to Hom_{{\mathcal Gr}^{\aone}_k}(\pi_1^{\aone}(X,x),\gm)$.  By functoriality of the Postnikov tower any map $X \to B\gm$ factors through a map $B\pi_1^{\aone}(X,x) \to B\gm$.

Now, since $\gm$ is abelian, we know that the canonical map from base-pointed to base-point free maps is an isomorphism.  Thus, we see that $[(X,x),(B\gm,\ast)] \isomt [X,B\gm]_{\aone} \isomt Pic(X)$.  If $Pic(X)$ is non-trivial, it follows that $Hom_{{\mathcal Gr}^{\aone}_k}(\pi_1^{\aone}(X,x),\gm)$ is non-trivial and thus, by the Yoneda lemma, that $\pi_1^{\aone}(X,x)$ is itself non-trivial.

Now, if $X$ is a strictly positive dimensional smooth proper variety, we claim $Pic(X)$ is non-trivial.  Indeed, since $X$ is smooth scheme over a field, it is, by our assumptions and conventions, separated, regular and Noetherian, and so admits an ample family of line bundles.  Since $X$ is strictly positive dimensional and proper, it is not affine, and thus one of these line bundles must be non-trivial.  Since $X$ is $\aone$-connected, we know $X(k)$ is non-empty.  Upon choice of a base-point $x \in X(k)$, we can appeal to the first part of the statement to finish the proof.
\end{proof}

\begin{prop}
\label{prop:etaleaonefundamentalgroupnontrivial}
If $(X,x)$ is a pointed smooth \'etale $\aone$-connected $k$-scheme.  We have a canonical identification
\[
Hom_{{\mathcal Gr}^{\et}_k}(\pi_1^{\aone,\et}(X,x),\gm) \isomto H^1_{\et}(X,\gm).
\]
If moreover $X$ is proper, then $\pi_1^{\aone,\et}(X,x)$ is non-trivial.
\end{prop}

\begin{proof}
Since $\gm$ is strongly $\aone$-invariant in the \'etale topology, $B\gm$ is $\aone$-local in the \'etale topology we have canonical bijections
\[
[L_{\aone}(X),B\gm]_{s,\et} \stackrel{\sim}{\longleftarrow} [L_{\aone}(X),B\gm]_{\aone,\et} \isomto [X,B\gm]_{\aone,\et}\isomto [X,B\gm]_{s,\et}.
\]
To say that $X$ is \'etale $\aone$-connected is equivalent to saying that $L_{\aone}(X)$ is \'etale simplicially connected.  Furthermore, we have $\pi_1^{\aone,\et}(X,x) = \pi_1^{s,\et}(L_{\aone}(x),x)$.

Again by existence and functoriality of the Postnikov tower (using the same tools mentioned in the proof of Proposition \ref{prop:nontrivialfundamentalgroup}) we have a canonical identification
\[
[L_{\aone}(X),B\gm]_{s,\et} \isomto Hom_{{\mathcal Gr}^{\et}_k}(\pi_1^{s,\et}(L_{\aone}(X),x),\gm);
\]
here we have implicitly used the fact that since $\gm$ is abelian, the canonical map from base-pointed maps to unpointed maps from any space to $B\gm$ is a bijection.  Tracking through all the bijections and identifications mentioned we deduce the existence of a bijection
\[
Hom_{{\mathcal Gr}^{\et}_k}(\pi_1^{\aone,\et}(X,x),\gm) \isomto H^1_{\et}(X,\gm),
\]
where ${\mathcal Gr}^{\et}_k$ denotes the category of \'etale sheaves of groups on $\Sm_k$.  By Hilbert's theorem 90 we know that $H^1_{\et}(X,\gm) = H^1_{Zar}(X,\gm) = Pic(X)$.
\end{proof}

\begin{rem}
It seems reasonable to expect that, generalizing \cite[Theorem 5.1]{MField}, the \'etale $\aone$-fundamental group is always strongly $\aone$-invariant in the \'etale topology, and a proof formally analogous to that of Proposition \ref{prop:nontrivialfundamentalgroup} may be used to establish Proposition \ref{prop:etaleaonefundamentalgroupnontrivial} as well.
\end{rem}

Define a notion of \'etale $\aone$-$h$-cobordism by replacing each occurrence of $\aone$-weak equivalence in Definition \ref{defn:aonehcobordism} by \'etale $\aone$-weak equivalence.  Using Proposition \ref{prop:nontrivialfundamentalgroup} or \ref{prop:etaleaonefundamentalgroupnontrivial}, the next result follows immediately from the discussion of Example \ref{ex:aonerigid}.

\begin{prop}[$\aone$-$h$-cobordism theorem]
\label{prop:etaleaonehcobordismtheorem}
Any (\'etale) $\aone$-$h$-cobordism between (\'etale) $\aone$-connected and (\'etale) $\aone$-simply connected smooth proper varieties over a field is trivial.
\end{prop}

\subsection{The $\aone$-fundamental group of $\pone$ (and related computations)}
We now discuss the computation of the $\aone$-fundamental group of $\pone$; this example, which is the simplest non-trivial case, is studied in great detail in \cite[\S 6.3]{MField}.  The discussion below uses basic properties of homotopy colimits; see \cite[\S 18.1]{Hirschhorn} for definitions and basic properties and \cite{BousfieldKan} for motivation and some useful facts.

To begin, let us first describe the ${\mathbb A}^1$-homotopy type of ${\mathbb P}^1$ ({\em cf.} \cite[\S 3 Corollary 2.18]{MV}).  The usual open cover of ${\mathbb P}^1$ by two copies of the affine line with intersection $\gm$ presents ${\mathbb P}^1$ as a push-out of the following diagram
\[
\aone \longleftarrow \gm \longrightarrow \aone.
\]
The push-out of this diagram can also be computed in the ${\mathbb A}^1$-homotopy category, where up to $\aone$-weak equivalence, it can be replaced by the diagram
\[
\ast \longleftarrow \gm \longrightarrow C(\gm).
\]
Here, $C(\gm) = \gm \wedge \Delta^1_s$ is the cone over $\gm$ (where the simplicial interval $\Delta^1_s$ is pointed by $1$).  The canonical map from the homotopy colimit to the colimit gives a morphism $\Sigma^1_s \gm \to {\mathbb P}^1$ that is an $\aone$-weak equivalence (since either morphism $\gm \hookrightarrow \aone$ is a cofibration, this follows from \cite[{Chapter XII Example 3.1iv}]{BousfieldKan}).  Now, consider Definition \ref{defn:freestronglyaoneinvariant}.

\begin{notation}
Set $F_{\aone}(1) := F_{\aone}(\gm) = \pi_1^{\aone}({\mathbb P}^1)$ where $\gm$ is pointed by $1$.
\end{notation}

The defining property of free strongly $\aone$-invariant sheaves of groups gives rise to a canonical morphism $\theta: \gm \to F_{\aone}(\gm)$.  On the other hand, Proposition \ref{prop:torustorsor} shows that the standard $\gm$-torsor ${\mathbb A}^2 \setminus 0 \to \pone$ induces a short exact sequence of $\aone$-homotopy groups
\begin{equation}
\label{eqn:centralextension}
1 \longrightarrow \pi_1^{\aone}({\mathbb A}^2 \setminus 0) \longrightarrow F_{\aone}(1) \longrightarrow \gm \longrightarrow 1
\end{equation}
that is {\em split} by $\theta$.

The ${\mathbb A}^1$-fundamental group of ${\mathbb A}^2 \setminus 0$ can be studied using an explicit description of its ${\mathbb A}^1$-homotopy type.  We pause to establish a more general result regarding the $\aone$-homotopy type of the complement of a finite set of $k$-points in ${\mathbb A}^n_k$.

\begin{prop}
\label{prop:affinespacesminuspoints}
Suppose $m$ and $n$ are integers with $m > 0$ and $n \geq 2$.
Suppose $p_1,\ldots,p_m \in {\mathbb A}^n_k$ are distinct points.  If $m > 1$, assume further that $k$ is an infinite field.  For any choice of points $q_1,\ldots,q_m \in {\mathbb A}^1_k$, there is an $\aone$-weak equivalence
\[
\Sigma^{n-1}_s \gm^{\wedge n-1} \wedge ({\mathbb A}^1 \setminus \setof{q_1,\ldots,q_m}) \isomto {\mathbb A}^n \setminus \setof{p_1,\ldots,p_m}.
\]
Thus, if $n > 2$, then ${\mathbb A}^n \setminus \setof{p_1,\ldots,p_m}$ is $\aone$-$1$-connected, and
\[
\pi_1^{\aone}({\mathbb A}^2 \setminus \setof{p_1,\ldots,p_m}) \isomto F_{\aone}(\gm^{\wedge n-1} \wedge ({\mathbb A}^1 \setminus \setof{q_1,\ldots,q_m})).
\]
\end{prop}

\begin{proof}
The second statement follows from the first by the unstable $\aone$-$n$-connectivity theorem \cite[Theorem 5.36]{MField}.  One can show that aforementioned result follows (essentially) formally from the Theorem \ref{thm:strongaoneinvaranceofpi1}.  The last statement follows from the definition of the free strongly $\aone$-invariant sheaf of groups on a sheaf of sets.

Choose coordinates $x_1,\ldots,x_n$ on ${\mathbb A}^n$.  If $k$ is infinite, for $n \geq 2$, the automorphism group of ${\mathbb A}^n$ acts $d$-transitively on ${\mathbb A}^n$ for any $d > 1$, i.e., any set of $d$-points can be moved to any other set of $d$-points by an automorphism; this follows immediately from \cite[Theorem 2]{Srinivas}.  Choose an automorphism that moves $p_1,\ldots,p_m$ to the points where $x_1 = q_i$ and $x_2 = \cdots = x_n = 0$.  Note that if $m = 1$, then the discussion of this paragraph is unnecessary.

We can then cover ${\mathbb A}^{n} \setminus \setof{p_1,\ldots,p_m}$ by the two open sets ${\mathbb A}^1 \setminus \setof{q_1,\ldots,q_m} \times {\mathbb A}^{n-1}$ and ${\mathbb A}^1 \times {\mathbb A}^{n-1} \setminus 0$.  This open cover realizes ${\mathbb A}^n \setminus \setof{p_1,\ldots,p_m}$ as the homotopy pushout of the diagram
\[
{\mathbb A}^1 \setminus \setof{q_1,\ldots,q_m} \longleftarrow {\mathbb A}^1 \setminus \setof{q_1,\ldots,q_m} \times {\mathbb A}^{n-1} \setminus 0 \longrightarrow {\mathbb A}^{n-1} \setminus 0
\]
computed in the $\aone$-homotopy category.  This homotopy pushout is better known as the {\em join} ${\mathbb A}^1 \setminus \setof{q_1,\ldots,q_m}$ and  ${\mathbb A}^{n-1} \setminus 0$, and an argument using cones as above can be used to identify this space as $\Sigma^1_s {\mathbb A}^1 \setminus \setof{q_1,\ldots,q_m} \wedge ({\mathbb A}^{n-1} \setminus 0)$.  The result then follows from a straightforward induction.
\end{proof}

\begin{rem}
The case of ${\mathbb A}^n \setminus 0$ of the above result is contained in \cite[Theorem 5.38]{MField}.  According to our above definitions, $\pi_1^{\aone}({\mathbb A}^2 \setminus 0)$ is the free strongly $\aone$-invariant sheaf of groups generated by $\gm \wedge \gm$; this group is sometimes denoted $F_{\aone}(2)$, but we will now describe it more explicitly.
\end{rem}

In classical topology, one knows that the fundamental group of a topological group is abelian; the next result proves the corresponding result in $\aone$-homotopy theory.

\begin{lem}
\label{lem:fundamentalgroupofagroupisabelian}
If $G$ is a Nisnevich sheaf of groups, the canonical map $G \to {\bf R}\Omega^1_s BG$ is a simplicial weak equivalence.  Thus, $\pi_1^{\aone}(G,1)$ is always an abelian sheaf of groups.
\end{lem}

\begin{proof}
Checking on stalks, one reduces the following result to the corresponding statement for simplicial sets.
\end{proof}

There is a projection morphism $SL_2 \to {\mathbb A}^2 \setminus 0$ that is an $\aone$-weak equivalence (being Zariski locally trivial with affine space fibers).  By Lemma \ref{lem:fundamentalgroupofagroupisabelian}, $F_{\aone}(2)$ is abelian.  It is closely related to both Milnor K-theory and Witt groups as explained in \cite[\S 2]{MField}, where a completely explicit presentation via ``symbols" (generators and relations) is given, and can be identified with the sheaf of second Milnor-Witt K-theory groups.

\begin{notation}
\label{notation:secondmilnorwittktheorysheaf}
Set $F_{\aone}(2) := F_{\aone}(\gm \wedge \gm) := \K^{MW}_2$, where $\gm$ is pointed by $1$.
\end{notation}

In any case, $F_{\aone}(1)$ fits into a split short exact sequence of the form
\[
1 \longrightarrow \K^{MW}_2 \longrightarrow F_{\aone}(1) \longrightarrow \gm \longrightarrow 1.
\]
Theorem 6.29 of \cite{MField} demonstrates that this short exact sequence is in fact a {\em central extension}.  As a sheaf of sets $F_{\aone}(1)$ is a product $\K^{MW}_2 \times \gm$ and we can be extremely explicit about the group structure on this sheaf of sets.

We will need a few pieces of notation about the sheaf $\K^{MW}_2$.  There is a canonical {\em symbol} morphism
\[
\Phi: \gm \times \gm \longrightarrow \K^{MW}_2
\]
obtained via composition of the projection $\gm \times \gm \to \gm \wedge \gm$ and the canonical morphism $\gm \wedge \gm \to \pi_1^{\aone}(\Sigma^1_s \gm \wedge \gm)$ described above.  Given a Henselian local scheme $S$, and sections $a,b \in \gm(S)$, we write $[a][b]$ for the image in $\K^{MW}_2(S)$.  The symbol morphism is, up to an explicit automorphism of $\K^{MW}_2$, related to the morphism $\theta$ by the following formula (\cite[Theorem 6.29 and Remark 6.30]{MField}):
\[
[a][b] = \Phi(a,b) = \theta(a)\theta(b)\theta(ab)^{-1}.
\]
This formula provides an explicit description of the multiplication on $\K^{MW}_2 \times \gm$ giving $F_{\aone}(1)$ its group structure.

\subsection[Preliminary computations for surfaces]{$\aone$-fundamental groups of surfaces I: Hirzebruch surfaces}
The Hirzebruch surface ${\mathbb F}_a$ is isomorphic to ${\mathbb P}(\O \oplus \O(-a))$ and comes equipped with a structure morphism ${\mathbb F}_a \to \pone$ admitting a section.  This morphism induces (split) group homomorphisms
\[
\pi_1^{\aone}({\mathbb F}_a) \longrightarrow F_{\aone}(1)
\]
for any integer $a$.  Pulling back the structure morphism along the $\gm$-torsor ${\mathbb A}^2 \setminus 0 \to \pone$ produces a trivial bundle of the form ${\mathbb A}^2 \setminus 0 \times \pone$.  Let $V$ be the $2$-dimensional representation of $\gm$ defined by the action $v \cdot (x_1,x_2) = (v^0x_1,v^{a}x_2)$.  This action induces an action of $\gm$ on ${\mathbb A}^2$ commuting with the scaling action.  Thus, the above action induces a $\gm$-action on ${\mathbb P}^1 = {\mathbb P}(V)$.  Furthermore, this $\gm$-action preserves the point with homogeneous coordinates $[1,0]$, which we refer to as $\infty$.  The induced map $\gm \to Aut({\mathbb P}^1)$ gives a morphism of sheaves
\[
\gm \longrightarrow Aut(F_{\aone}(1)).
\]
Note that $\gm$ also acts on ${\mathbb A}^2 \setminus 0$, in a manner inducing the central extension of Equation \ref{eqn:centralextension}.  The inclusion of a fiber $\pone$ (say over the image of a chosen base-point) then gives a morphism $F_{\aone}(1) \to \pi_1^{\aone}({\mathbb F}_a)$, and the aforementioned discussion shows that one has a split short exact sequence of groups
\[
1 \longrightarrow F_{\aone}(1) \longrightarrow \pi_1^{\aone}({\mathbb F}_a) \longrightarrow F_{\aone}(1) \longrightarrow 1.
\]
Our discussion of the actions shows that the action of $F_{\aone}(1)$ on itself factors through the quotient map $F_{\aone}(1) \to \gm \to Aut(F_{\aone}(1))$, and this last map is completely determined by the integer $a$.  We write
\[
\pi_1^{\aone}({\mathbb F}_a) := F_{\aone}(1) \rtimes^a F_{\aone}(1).
\]
With this notation in place, we can state the first computation.

\begin{prop}
\label{prop:hirzebruchfundamental}
We have isomorphisms of sheaves of groups
\[
\pi_1^{\aone}({\mathbb F}_{a}) \isomto \begin{cases}
F_{\aone}(1) \times F_{\aone}(1) & \text{ if } a \text{ is even, and}\\
F_{\aone}(1) \rtimes^1 F_{\aone}(1) & \text{ if } a \text{ is odd}.
\end{cases}
\]
Furthermore, the sheaves of groups $F_{\aone}(1) \times F_{\aone}(1)$ and $F_{\aone}(1) \rtimes^1 F_{\aone}(1)$ are not isomorphic.
\end{prop}

\begin{proof}
The first statement follows immediately from the proof of Lemma \ref{lem:mod2}; the isomorphisms of $\aone$-fundamental groups are induced by inclusions in the appropriate $\aone$-$h$-cobordisms.

To establish the second statement, we study the morphism of sheaves $\gm \to Aut(F_{\aone}(1))$ in more detail.  For any finitely generated separable field extension $L/k$ and any element $u \in L^*$, consider the map $\pone \to \pone$ defined on homogeneous coordinates by $[1,u^a]$.  This preserves the point with homogeneous coordinates $[1,0]$, which we called $\infty$, and induces the map $F_{\aone}(1)(L) \to F_{\aone}(1)(L)$ that we'd like to study.  If $a = 0$, this map is the trivial map.

For $a \neq 0$, we use the identification of \cite[Corollary 6.34]{MField}.  Indeed, $Aut(F_{\aone}(1))$ can be identified with the sheaf of units in $\Z \oplus \K^{MW}_1$.  Now, for any finitely generated separable extension $L/k$, and any $u \in L^*$, the map $L^* = \gm(L) \to \Z \oplus \K^{MW}_1(L)$, which is not a morphism of sheaves of groups, is given by the formula $u \mapsto (1,[u])$.   We thus want to compute the action of the element $(1,[u^a])$ by conjugation on an element of $F_{\aone}(1)$.  In case $a = 1$, this is exactly the action mentioned in Remark 6.31 of \cite{MField} and, in particular, not the trivial action.
\end{proof}

\subsection[Presentations of $\aone$-fundamental groups]{$\aone$-fundamental groups of surfaces II: blow-ups and presentations}
We now study the $\aone$-fundamental groups of blow-ups of smooth schemes that are covered by affine spaces.  Before we proceed, we recall that for any integer $n > 1$ and arbitrary choices of base-point, Proposition \ref{prop:affinespacesminuspoints} shows $\pi_1^{\aone}({\mathbb A}^{n+1} \setminus 0) = 1$; using Proposition \ref{prop:torustorsor} we deduce that $\pi_1^{\aone}({\mathbb P}^n) = \gm$.

\begin{prop}
\label{prop:fundamentalgroupblowup}
Suppose $X$ is a smooth $k$-variety of dimension $n \geq 2$ that is covered by affine spaces and $x \in X(k)$ is a $k$-point.  In case $n = 2$, we have an isomorphism $\pi_1^{\aone}(X \setminus x) \star^{\aone}_{\K^{MW}_2} e \isomt \pi_1^{\aone}(X)$.  If $n > 2$, the open immersion $X \setminus x \hookrightarrow X$ induces an isomorphism $\pi_1^{\aone}(X \setminus x) \isomt \pi_1^{\aone}(X)$.  Furthermore, we have isomorphisms:
\[
\pi_1^{\aone}({\sf Bl}_x(X)) \cong \begin{cases}
\pi_1^{\aone}(X \setminus x) \star^{\aone}_{\K^{MW}_2} F_{\aone}(1)  & \text{ if } n = 2, \text{ and } \\
\pi_1^{\aone}(X) \star^{\aone} \gm & \text{ if } n > 2.
\end{cases}
\]
\end{prop}

\begin{proof}
Suppose $X$ is a smooth $k$-variety of dimension $n$ covered by affine spaces, where $n \geq 2$.  Either $X \cong {\mathbb A}^n$, or we can cover $X$ by two open sets, the first isomorphic to ${\mathbb A}^n$, the second isomorphic to $X \setminus x$, and having intersection ${\mathbb A}^n \setminus 0$ (after applying an automorphism of affine space if necessary).  Note that ${\mathbb A}^n \setminus 0$ and ${\mathbb A}^n$ are $\aone$-connected (e.g., they are both $\aone$-chain connected).  Using this, the space $X \setminus x$ is $\aone$-connected since it admits an open cover by $\aone$-chain connected open sets (namely open sets isomorphic to ${\mathbb A}^n \setminus 0$ and ${\mathbb A}^n$).

We know that $\pi_1^{\aone}({\mathbb A}^n \setminus 0)$ is trivial if $n > 2$ and isomorphic to $\K^{MW}_2$ (for any choice of base-point) for $n = 2$; this establishes the first part of the proposition in case $X = {\mathbb A}^n$.  In general, we can write $X$ as the push-out of the diagram
\[
\xymatrix{
{\mathbb A}^n \setminus 0 \ar[r]\ar[d] & {\mathbb A}^n \ar[d] \\
X \setminus x \ar[r] & X
}
\]
Now, since ${\mathbb A}^n$ is $\aone$-contractible and hence has trivial $\aone$-fundamental group, the $\aone$-van Kampen theorem (\ref{thm:aonevanKampen}) gives us an isomorphism of the form
\[
\pi_1^{\aone}(X \setminus x) \star^{\aone}_{\pi_1^{\aone}({\mathbb A}^n \setminus 0)} e \isomto \pi_1^{\aone}(X).
\]
A straightforward computation involving the definition of amalgamated sums and Proposition \ref{prop:affinespacesminuspoints} shows that if $n > 2$, then the open immersion $X \setminus x \to X$ gives an isomorphism $\pi_1^{\aone}(X \setminus x) \isomt \pi_1^{\aone}(X)$.  On the other hand, if $n = 2$, using Notation \ref{notation:secondmilnorwittktheorysheaf}, we get an isomorphism
\[
\pi_1^{\aone}(X \setminus x) \star^{\aone}_{\K^{MW}_2} e \isomto \pi_1^{\aone}(X),
\]
which establishes the first part of our statement.

Now, consider ${\sf Bl}_x(X)$.  If $X \cong {\mathbb A}^n$, one knows that ${\sf Bl}_0({\mathbb A}^n)$ is isomorphic to the total space of the line bundle associated with the locally free sheaf $\O_{{\mathbb P}^{n-1}}(1)$ over ${\mathbb P}^{n-1}$.  Thus ${\sf Bl}_0({\mathbb A}^n)$ is $\aone$-weakly equivalent to ${\mathbb P}^{n-1}$, and we deduce that $\pi_1^{\aone}({\sf Bl}_0({\mathbb A}^n))$ is $\gm$ if $n > 2$ and $F_{\aone}(1)$ if $n = 2$.

If $X$ is covered by more than one copy of affine space, using the open cover above, together with the fact that blowing up is Zariski local, we get a Mayer-Vietoris diagram of the form
\[
\xymatrix{
{\mathbb A}^n \setminus 0 \ar[r]\ar[d] & {\sf Bl}_0{\mathbb A}^n \ar[d] \\
X \setminus x \ar[r] & {\sf Bl}_xX.
}
\]

{\em Case $n > 2$.} If $n > 2$, then we know that $\pi_1^{\aone}({\mathbb A}^n \setminus 0)$ is trivial.  The $\aone$-van Kampen theorem then provides an isomorphism
\[
\pi_1^{\aone}({\sf Bl}_{0}({\mathbb A}^n)) \star^{\aone} \pi_1^{\aone}(X \setminus x) \isomto \pi_1^{\aone}({\sf Bl}_x(X)).
\]
Thus, by the discussion in the case of ${\mathbb A}^n$, it follows that $\pi_1^{\aone}({\sf Bl}_x(X))$ is isomorphic to the amalgamated sum $\pi_1^{\aone}(X \setminus x) \star^{\aone} \gm$.  Finally, the first part of the proposition allows us to conclude that $\pi_1^{\aone}(X \setminus x) \isomt \pi_1^{\aone}(X)$.

{\em Case $n = 2$.} For $n = 2$, we know $\pi_1^{\aone}({\mathbb A}^2 \setminus 0) = \K^{MW}_2$, which we recall is abelian by Lemma \ref{lem:fundamentalgroupofagroupisabelian}.  Since, ${\sf Bl}_0({\mathbb A}^2)$ is $\aone$-weakly equivalent to $\pone$, we know $\pi_1^{\aone}({\sf Bl}_0({\mathbb A}^2)) \cong F_{\aone}(1)$.
\end{proof}

\begin{cor}
\label{cor:pnblowups}
Suppose $m$ is an integer $\geq 3$.  For distinct points $x_1,\ldots,x_n \in {\mathbb P}^m(k)$, we have an isomorphism of strongly $\aone$-invariant sheaves of groups:
\[
\gm \star^{\aone} \cdots \star^{\aone} \gm \stackrel{\sim}{\longrightarrow}
 \pi_1^{\aone}({\sf Bl}_{x_1,\ldots,x_n}({\mathbb P}^m)),
\]
where the amalgamated sum on the left hand side has $n$-factors of $\gm$.
\end{cor}

\begin{proof}
This follows immediately by induction from Proposition \ref{prop:fundamentalgroupblowup}.
\end{proof}

\begin{ex}
We can be somewhat more explicit about the structure of some of the above amalgamated products.  For example, for any point $x \in {\mathbb P}^m(k)$, there is a $\gm^{\times 2}$-torsor ${\mathbb A}^{m} \setminus 0 \times {\mathbb A}^2 \setminus 0 \to {\sf Bl}_x({\mathbb P}^m)$.  In particular, if $m > 2$, Proposition \ref{prop:torustorsor} shows that this $\gm^{\times 2}$-torsor gives rise to an exact sequence of the form
\[
1 \longrightarrow \K^{MW}_2 \longrightarrow \gm \star^{\aone} \gm \longrightarrow \gm \times \gm \longrightarrow 1.
\]
More generally, for any integer $n > 2$, $\aone$-covering space theory of \cite[\S 6.1]{MField} can be used to construct an epimorphism:
\[
\gm \star^{\aone} \cdots \star^{\aone} \gm \longrightarrow \gm \times \cdots \times \gm,
\]
where both sides contain $n$ copies of $\gm$.  Precisely, one can check that the blow-up of $n$-points of ${\mathbb P}^m$ has an $\aone$-covering space corresponding to a torsor under the torus dual to the Picard group.  The factor $\gm^{\times n}$ can be thought of as having motivic weight $1$, and the kernel of the epimorphism can be thought of as having motivic weight $2$.  A similar filtration should exist on the $\aone$-fundamental group of any smooth $\aone$-connected $k$-variety.
\end{ex}

We now study $\aone$-fundamental groups of blow-ups of points on ${\mathbb P}^2$ in greater detail.  Unfortunately, the computation is not as nice as Corollary \ref{cor:pnblowups}.  Instead, we discuss a method of presenting $\aone$-fundamental groups in terms of generators and relations.  To begin, let us compute the $\aone$-fundamental group of a wedge sum of finitely many copies of $\pone$.  For comparison, recall that the (usual) fundamental group of a wedge of $n$ circles is the free group on $n$ generators, or, equivalently, the amalgamated sum of $n$ copies of $\Z$ (the free group on one generator).

\begin{lem}
\label{lem:fundamentalgroupofawedgesum}
For each integer $n \geq 1$, we have an isomorphism
\[
F_{\aone}(1) \star^{\aone} \cdots \star^{\aone} F_{\aone}(1) \isomto \pi_1^{\aone}({\pone}^{\vee n}),
\]
where there are $n$-factors of $F_{\aone}(1)$ on the left hand side.
\end{lem}

\begin{proof}
First, let us construct a geometric model of the wedge of $n$ copies of $\pone$.  Abstractly, the wedge sum of two pointed spaces $({\mathcal X},x)$ and $({\mathcal Y},y)$ is defined as the pushout of the diagram
\[
{\mathcal X} \longleftarrow \ast \longrightarrow {\mathcal Y}
\]
where the two morphisms are inclusion of the basepoint.  Since both of these morphisms are cofibrations, the canonical map from the homotopy colimit of this diagram to the colimit is an $\aone$-weak equivalence.  Thus, it suffices to compute the $\aone$-fundamental group of a weakly equivalent diagram of spaces.

Take $\pone$ and fix a point $\infty$.  Consider the variety ${\pone}^{\times n}$, which we can think of as $n$ ordered points in $\pone$.  There are $\frac{n(n-1)}{2}$ closed subvarieties isomorphic to ${\pone}^{\times {n-2}}$ corresponding to fixing two points to be $\infty$.  Let $X_n$ be the complement in ${\pone}^{\times n}$ of the union of these closed subvarieties.  Observe that $X_n$ can be covered by $n$ open sets $X_{n,i}$ ($i = 1,\ldots,n$) of the form $\pone \times {\mathbb A}^{n-1}$ whose pairwise intersections are ${\mathbb A}^n$.  Choosing the base-point $0$ in ${\mathbb A}^n$ induces a base-point in each of these open sets.  An easy induction argument together with the $\aone$-van Kampen theorem (\ref{thm:aonevanKampen}) shows that the $\aone$-fundamental group of this space is the amalgamated sum of $n$ copies of $F_{\aone}(1)$.

Projection away from $\pone$ defines a (pointed) $\aone$-weak equivalence between the diagram consisting of ${\mathbb A}^n$ together with each of the open immersions to $X_{n,i}$ and the diagram presenting the wedge sum of $n$-copies of $\pone$.  Thus, $X_n$ is $\aone$-weakly equivalent to the wedge sum of $n$-copies of $\pone$.
\end{proof}

\begin{ex}
There is an alternate presentation of the homotopy type of a wedge sum of $n$ copies of $\pone$ as a smooth affine surface.  Let $a_1,\ldots,a_{n+1} \in k$ be a collection of $(n+1)$-distinct elements of a field $k$.  Consider the hypersurface in ${\mathbb A}^3$ defined by the equation
\[
xy = \prod_{i=1}^{n+1} (z - a_i);
\]
such hypersurfaces have been studied by Danielewski and Fieseler \cite[p. 9]{Fieseler}.  Note the similarity between this example and Example \ref{ex:conicbundles}.

Let us, for this example, lift our convention that schemes be separated.  Let ${\mathbb A}^1_{n,0}$ denote the non-separated (smooth) scheme obtained by gluing $n$ copies of $\aone$ along $\aone \setminus 0$ via the identity morphism.  One can show that the above hypersurface is the total space of a flat morphism with target ${\mathbb A}^1_{n+1,0}$ and fibers isomorphic to $\aone$.  From this one can construct a cover of the above hypersurface by $(n+1)$-copies of ${\mathbb A}^2$ glued along $\gm \times \aone$.  A homotopy colimit argument as above can be used to identify this homotopy type with ${\pone}^{\wedge n}$.
\end{ex}

Recall that if $M$ is a closed $2$-manifold, then $M \setminus pt$ has the homotopy type of a wedge of circles.  The following result proves the analogous fact in $\aone$-homotopy theory.

\begin{prop}
\label{prop:removeapointgetawedge}
Suppose $X$ is a rational smooth proper surface over an algebraically closed field $k$, and assume $n = \operatorname{rk} Pic(X)$.  If $x \in X(k)$, then there exists an $\aone$-weak equivalence
\[
X -x \cong {\pone}^{\vee n}.
\]
\end{prop}

\begin{proof}
There are several possible proofs of this fact.  They key point is that a smooth proper rational surface admits an open dense subscheme isomorphic to ${\mathbb A}^2$ whose complement is a union of a finite number of copies of $\aone$ and a point.\newline

\noindent{\em Step 1.} The case where $X = {\mathbb P}^2$ is clear since removing a point from ${\mathbb P}^2$ produces the total space of a line bundle over $\pone$.  The case where $X = {\mathbb P}^1 \times {\mathbb P}^1$ is exactly the geometric construction in the case $n = 2$ of Lemma \ref{lem:fundamentalgroupofawedgesum}.  \newline

\noindent{\em Step 2.} We will show that ${\sf Bl}_0({\mathbb A}^2)$ is homotopy equivalent to the wedge of ${\mathbb A}^2 \setminus 0 \vee \pone$.  To see this, first cover ${\sf Bl}_0({\mathbb A}^2)$ by two open sets isomorphic to ${\mathbb A}^2$ with intersection $\gm \times \aone$.  Now, the point can be assumed to lie in a single copy of ${\mathbb A}^2$ so we get a Mayer-Vietoris square of the form
\[
\xymatrix{
\gm \times {\mathbb A}^1 \ar[r]\ar[d]& {\mathbb A}^2 \setminus 0 \ar[d] \\
{\mathbb A}^2 \ar[r] & {\sf Bl}_0({\mathbb A}^2).
}
\]
Up to $\aone$-weak equivalence, we can replace this by a diagram of the form
\[
\ast \longleftarrow \gm \longrightarrow {\mathbb A}^2 \setminus 0.
\]
Replacing $\ast$ by $C(\gm)$, the homotopy colimit of the resulting diagram coincides with $\pone \vee {\mathbb A}^2 \setminus 0$. \newline

\noindent{\em Step 3.} Finally, we treat the general case by induction.  By Theorem \ref{thm:surfaces}, it suffices to prove the following fact.  Let $X$ be a rational smooth proper surface, and let $x \in X(k)$.  Choose an open subscheme of $X$ isomorphic to ${\mathbb A}^2$ and containing $x$.  Let $Y = {\sf Bl}_x X$, $D$ the exceptional divisor of the blow-up, and $y$ a point of $D$.  Then, there is an $\aone$-weak equivalence $Y \setminus y \cong X \setminus x \vee \pone$.

To see this, observe that there is a Mayer-Vietoris square of the form
\[
\xymatrix{
{\mathbb A}^2 \setminus 0 \ar[r] \ar[d]& Y \setminus D \ar[d]\\
{\sf Bl}_0{\mathbb A}^2 \setminus y \ar[r] & Y \setminus y.
}
\]
Now $Y \setminus D$ is isomorphic to $X \setminus x$.  In {\em Step 2}, we showed that ${\sf Bl}_0({\mathbb A}^2) \setminus y \cong {\mathbb A}^2 \setminus 0 \vee \pone$.  Further contemplation of {\em Step 2} reveals that the left vertical map identifies ${\mathbb A}^2 \setminus 0$ with the corresponding wedge summand in ${\mathbb A}^2 \setminus 0 \vee \pone$.  The homotopy pushout of this diagram is then precisely $X \setminus x \vee \pone$, which gives the required $\aone$-weak equivalence.

{\em Step 4.} It remains to observe, as we did before, that the number of blowups that occurs is equal to the rank of the Picard group.
\end{proof}

\begin{rem}
One can also prove this result by appeal to the theory of toric surfaces \cite[\S 2.5]{Fulton}.  Observe that by the proof of Theorem \ref{thm:surfaces} any smooth proper rational surface is $\aone$-$h$-cobordant to a smooth proper toric surface.  Namely, either $\pone \times \pone$ or an iterated toric blow-up of ${\mathbb P}^2$.  In either case, one can explicitly realize $X - x$ as a union of total spaces of line bundles over ${\mathbb P}^1$ glued along the common copy of ${\mathbb A}^2$.  A homotopy colimit argument analogous to the one above finishes the proof.  This method has the benefit of providing explicit morphisms $\pone \to X \setminus x$ that can be viewed as explicit ``geometric generators" of the $\aone$-fundamental group, namely the zero sections to the various line bundles.
\end{rem}

\begin{cor}
\label{cor:fundamentalgroupofblowupofpointsonsurface}
If $X$ is a rational smooth proper surface over an algebraically closed field $k$, and $n = \operatorname{rk} Pic(X)$, then for any $x \in X(k)$ there is an isomorphism of the form
\[
(F_{\aone}(1) \star^{\aone} \cdots \star^{\aone} F_{\aone}(1)) \star^{\aone}_{F_{\aone}(2)} e \isomto \pi_1^{\aone}(X)
\]
giving a ``presentation" of the $\aone$-fundamental group of $X$.
\end{cor}

\begin{proof}
Combine Proposition \ref{prop:fundamentalgroupblowup} and Proposition \ref{prop:removeapointgetawedge}.
\end{proof}

\begin{rem}
Much more can be said if we use aspects of $\aone$-homology developed in \cite[\S 5.2]{MField}.  Assume again $X$ is a rational smooth proper surface, fix $x \in X(k)$, and choose ${\mathbb A}^2 \subset X$ containing $x$.  Consider the Mayer-Vietoris square
\[
\xymatrix{
{\mathbb A}^2 \setminus 0 \ar[r]\ar[d] & X \setminus x \ar[d]\\
{\mathbb A}^2 \ar[r] & X
}
\]
as we did in the proof of Proposition \ref{prop:fundamentalgroupblowup}.

Observe that the canonical morphism $H_1^{\aone}({\mathbb A}^2 \setminus 0) \to H_1^{\aone}(X \setminus x)$ is an isomorphism. Indeed, by assumption, $X$ can be covered by open sets isomorphic to affine space, and we can use the Mayer-Vietoris sequence (see \cite[Proposition 3.32]{ADExcision}) together with a straightforward induction argument to deduce this fact.  Since $\pi_1^{\aone}({\mathbb A}^2 \setminus 0)$ is abelian, the $\aone$-Hurewicz theorem (\cite[Theorem 5.55]{MField}) shows that the morphism
\[
\pi_1^{\aone}({\mathbb A}^2 \setminus 0) \longrightarrow \pi_1^{\aone}(X \setminus x)
\]
factors through the abelianization of $\pi_1^{\aone}(X \setminus x)$.  In other words $\pi_1^{\aone}({\mathbb A}^2 \setminus 0) \cong \K^{MW}_2$ can be thought of as being in the ``commutator subgroup" of $\pi_1^{\aone}(X,x)$.  Nevertheless, it is unclear how to give a simple closed form expression for the $\aone$-fundamental groups of blow-ups of finitely many points on ${\mathbb P}^2$ in a manner similar to Corollary \ref{cor:pnblowups}.
\end{rem}

\begin{rem}
That there are various higher dimensional generalizations of Proposition \ref{prop:removeapointgetawedge}.  Also, it would be interesting to know which $\aone$-fundamental groups can be ``finitely presented" in the sense suggested by Corollary \ref{cor:fundamentalgroupofblowupofpointsonsurface}.
\end{rem}

\subsection{The $\aone$-homotopy classification of rational smooth proper surfaces revisited}
Finally, we can deduce the remaining theorem statement (Theorem \ref{thm:refinedsurfaces}) from \S \ref{s:introduction}.  Indeed, combining Theorem \ref{thm:surfaces}, Proposition \ref{prop:intersectionpairing}, Proposition \ref{prop:hirzebruchfundamental}, and Corollary \ref{cor:fundamentalgroupofblowupofpointsonsurface} we obtain the following result.

\begin{cor}
\label{thm:surfacesfundamentalgroup}
If $X$ and $Y$ are two rational smooth proper surfaces over an algebraically closed field, the following conditions are equivalent.
\begin{itemize}
\item[i)] The varieties $X$ and $Y$ are $\aone$-$h$-cobordant.
\item[ii)] The varieties $X$ and $Y$ are $\aone$-weakly equivalent.
\item[iii)] The varieties $X$ and $Y$ have isomorphic $\aone$-fundamental groups.
\item[iv)] The $\Z$-modules $H^{2,1}(X,\Z)$ and $H^{2,1}(Y,\Z)$ equipped with the structure of quadratic spaces by means of the intersection pairing are isomorphic as quadratic spaces.
\end{itemize}
\noindent Furthermore the set ${\mathscr S}_{\aone}(X)$ consists of exactly $1$ element.
\end{cor}

\begin{rem}
We expect that $\aone$-disconnected varieties can be studied by considering the $\aone$-fundamental (sheaf of) groupoid(s) instead of the $\aone$-fundamental (sheaf of) group(s).
\end{rem}

\begin{rem}
Combining Corollaries \ref{cor:aoneconnectednesssurfaces} and \ref{thm:surfacesfundamentalgroup} we obtain the solution to the $\aone$-surgery problem (Problem \ref{problem:aonesurgery} from the introduction) for smooth proper $\aone$-connected surfaces over algebraically closed fields $k$.  If $k$ is not algebraically closed, Corollary \ref{thm:surfacesfundamentalgroup} will not provide the classification of smooth proper $\aone$-connected surfaces due to the examples of stably $k$-rational, non rational surfaces (see Example \ref{ex:BCTSSD}).  It seems reasonable to expect that smooth proper $\aone$-connected surfaces over an {\em arbitrary} field are classified up to $\aone$-$h$-cobordism by their $\aone$-fundamental group.
\end{rem}

\subsubsection*{Extension: Homotopy minimality}
We expect the sheaf of $\aone$-connected components to be a birational invariant of smooth proper schemes; we provide evidence for this expectation in \S \ref{s:unramified}.  Generalizing from the examples we know, it also seems reasonable to expect that two smooth proper schemes that are $\aone$-weakly equivalent are also birationally equivalent.  As we have seen, blowing-up increases complexity of the $\aone$-fundamental group.  We therefore suggest a way to tie together $\aone$-homotopy types of smooth schemes having a given function field.  Indeed, Theorem \ref{thm:surfaces}, and the discussion preceding it suggest that if a function field $K$ admits a smooth proper representative, it admits a (non-unique) representative that is homotopically minimal in the following sense.

\begin{defn}
\label{defn:aoneminimal}
An $\aone$-homotopy type admitting a smooth proper $k$-variety $X$ as a representative is called {\em minimal} if given any triple $(X',\psi,\varphi)$ consisting of a smooth proper $k$-variety $X'$, an $\aone$-weak equivalence $\psi: X' \to X$, and a proper birational morphism $\varphi: X' \to Y$ to a smooth $k$-variety $Y$, $\varphi$ is an isomorphism.  Any smooth proper representative of a minimal $\aone$-homotopy type will be called {\em $\aone$-minimal}.

Given two smooth proper $k$-varieties $X$ and $Y$, we will say that $X$ is an {\em $\aone$-minimal model} for $Y$ if $X$ is $\aone$-minimal, and there exist a smooth proper $k$-variety $X'$, an $\aone$-weak equivalence $X' \to Y$ and a proper, birational morphism $X' \to X$.
\end{defn}

\begin{problem}[$\aone$-minimality]
\label{problem:aoneminimality}
Let $K$ be the function field of a smooth proper $k$-variety of dimension $n$.  Describe the set ${\mathscr M}_{\aone}(K)$ of minimal $\aone$-homotopy types for smooth proper varieties with function field $K$ by providing explicit $\aone$-minimal smooth proper $k$-varieties in each $\aone$-homotopy type in ${\mathscr M}_{\aone}(K)$.  Moreover, for a given smooth proper $k$-variety $Y$ provide an explicit procedure to find an $\aone$-minimal model of $Y$ from the previous list.
\end{problem}

The $\aone$-minimal, rational surfaces are, of course, $\pone \times \pone$ and ${\mathbb P}^2$.  The first examples of $\aone$-minimal varieties are provided by the following result.

\begin{prop}
\label{prop:rigidminimal}
The $\aone$-homotopy type of a smooth proper $\aone$-rigid $k$-variety $X$ is $\aone$-minimal in the sense of \textup{Definition \ref{defn:aoneminimal}}.
\end{prop}

\begin{proof}
Suppose $X$ is a smooth proper $\aone$-rigid variety, and $f: X \to Y$ is a proper birational morphism with $Y$ smooth and proper.  By \cite[Proposition 1.3]{KollarMori}, for any point $y$ of $Y$, either $f^{-1}(y)$ is a point or $f^{-1}(y)$ contains and is covered by rational curves.  Now, the exceptional set of $f$ is of pure codimension $1$ (\cite[Corollary 2.63]{KollarMori}). If the exceptional set of $f$ is empty, $f$ is an isomorphism.  Otherwise, we can find a non-trivial morphism $\aone \to X$ contradicting $\aone$-rigidity of $X$.
\end{proof}

\begin{rem}
The existence of stably rational, non-rational surfaces over non-algebraically closed fields, i.e., Example \ref{ex:BCTSSD}, suggests minimality is not preserved by field extension.  Nevertheless,
\end{rem}

\subsubsection*{Extension: torsion of an $\aone$-weak equivalence}
We use the notation of \S \ref{s:introduction}.  Barden, Mazur and Stallings reconsidered the $h$-cobordism theorem in the non-simply connected case using J.H.C. Whitehead's notion of torsion of a homotopy equivalence.  The $s$-cobordism theorem states that an $h$-cobordism $(W,M,M')$ of manifolds of dimension $\geq 5$ such that the inclusions $M \hookrightarrow W$ and $M' \hookrightarrow W$ are {\em simple} homotopy equivalences (i.e., the torsion vanishes) is diffeomorphic to a product.  Furthermore, $h$-cobordisms $(W,M,M')$ with $\dim M \geq 5$ are parameterized by the Whitehead group of $M$, which is a certain quotient of the algebraic $K_1$ of the group algebra $\Z\pi_1(M)$. (Note: in low dimensions, $h$-cobordisms between simply connected manifolds can fail to be products.)

We know that $\aone$-$h$-cobordisms of strictly positive dimensional smooth proper $\aone$-connected varieties are always non-trivial by Proposition \ref{prop:nontrivialfundamentalgroup}.  The $s$-cobordism theorem suggests an explanation for non-triviality of such $\aone$-$h$-cobordisms in terms of the $\aone$-fundamental group.  One can formulate a notion of $\aone$-Whitehead torsion of an $\aone$-weak equivalence.  Suppose given an $\aone$-$h$-cobordism $(W,f)$ between smooth proper $\aone$-connected $k$-varieties $X$ to $X'$.  The inclusion $X \hookrightarrow W$ induces a morphism of $\aone$-singular chain complexes (see \cite{MField} \S 5.2 for a definition), and the cone of this morphism is an $\aone$-contractible chain complex of (sheaves of) modules over the (sheaf of) group algebra(s) $\Z[\pi_1^{\aone}(X)]$.  When the $\aone$-singular chain complexes are sufficiently well understood, one can associate with this complex a {\em computable} $\aone$-Whitehead torsion.  Optimistically, one can hope for an $\aone$-$s$-cobordism theorem stating that $\aone$-$h$-cobordisms can be parameterized by an appropriately defined Whitehead group of the $\aone$-fundamental group.

\begin{rem}
The computations above suggest that the $\aone$-Whitehead torsion will likely be quite complicated in general.  According to Corollary \ref{thm:surfacesfundamentalgroup}, non-minimal $\aone$-homotopy types can, in general, contain moduli of non-isomorphic varieties.  On the other hand, the minimal $\aone$-homotopy types for rational surfaces contain a discrete set of isomorphism classes of smooth proper varieties.
\end{rem}

\subsubsection*{Extension: determining ${\mathscr S}_{\aone}(X)$}
Given a finite CW complex $X$, we now recall some aspects of the surgery problem and determination of the structure set ${\mathscr S}(X)$.  If $X$ is homotopy equivalent to a manifold then (a) the cohomology of $X$ satisfies Poincar\'e duality, and (b) $X$ has a tangent bundle or, by Atiyah duality, $X$ has a stable normal bundle.  Amazingly, these two pieces of data turn out to be essentially sufficient to identify manifolds among CW complexes, provided certain compatibility conditions are satisfied.

A finite CW complex satisfying Poincar\'e duality is called a {\em geometric Poincar\'e complex}.  Any geometric Poincar\'e complex  $X$ admits a {\em Spivak normal fibration}, which is a homotopy theoretic substitute for the stable normal bundle.  The Spivak normal fibration is a homotopy sphere bundle and is ``classified" by a map $\nu_X: X \to BG$ where $BG$ is the colimit of the classifying spaces of the monoids of homotopy self-equivalences of the sphere of dimension $n$ for a natural sequence of inclusions.  If $X$ is homotopy equivalent to a manifold, the Spivak normal fibration will admit a vector bundle reduction (classifying the stable normal bundle).  If $BO$ denotes the classifying space for the stable orthogonal group, there is a map $\iota: BO \to BG$ that induces the Hopf-Whitehead $J$-homomorphism. A vector bundle reduction of $\nu_X$ is a lift along $\iota$.  The homotopy cofiber of $\iota$ can be identified with $B(G/O)$ and reductions exist if and only if the induced map $X \to B(G/O)$ is homotopically trivial.  A primary obstruction to $X$ admitting a manifold structure is the homotopic triviality of this map, and if this obstruction vanishes, lifts are classified by the set of homotopy classes of maps $[X,G/O]$.

The secondary {\em surgery obstruction} provides a map from $[X,G/O]$ to a group $L_n(\Z(\pi_1(X)))$ defined in terms of complexes of $\Z(\pi_1(X))$-modules with duality (where $n$ is the ``formal" dimension of $X$).  Finally, the structure set ${\mathscr S}(X)$ fits into an exact sequence of {\em sets} of the form:
\[
L_{n+1}(\Z(\pi_1(X))) \Longrightarrow {\mathscr S}(X) \longrightarrow [X,G/O] \longrightarrow L_n(\Z(\pi_1(X))),
\]
called the surgery exact sequence.  Our notation signifies that group $L_{n+1}(\Z(\pi_1(X)))$ acts on the set ${\mathscr S}(X)$, and the last map on the right hand side is {\em not} in general a group homomorphism (even though both its source and target are groups)!

For an appropriate analog in $\aone$-homotopy theory, note that Algebraic K-theory is representable in the $\aone$-homotopy category (see \cite[\S 4 Theorem 3.13]{MV}).  Smooth schemes have tangent bundles, which in some situations are classified by maps to an infinite Grassmannian (\cite[\S 7]{MField}), and the statement of $\aone$-Atiyah duality for smooth projective schemes (see \cite[Theorem A.1]{Hu} or \cite[Th\'eor\`eme 2.2]{Riou}) tells us how to define the notion of an $\aone$-Poincar\'e complex.

One may define the natural analog of ``$G$" as the $\pm 1$-components of the $\pone$-infinite loop space $Q_{\pone}S^0_k$ corresponding to the $\pone$-sphere spectrum ${\mathbb S}^0_k$.  Except at the ``zeroth" level, the $\aone$-homotopy groups of this space coincide with the stable motivic homotopy groups of spheres.  One needs to prove existence of analogs of Spivak normal fibrations for $\aone$-Poincar\'e complexes.  To develop the primary K-theory obstruction for Problem \ref{problem:aonesurgery}, one needs to study the ``(sheaf theoretic) motivic $J$-homomorphism," and the ``$\pone$-loop space recognition problem" as mentioned by Voevodsky (to show that spaces like ``$G/GL$" are $\pone$-infinite loop spaces).  Adopting this point of view, computations of stable $\aone$-homotopy (sheaves of) groups (e.g., \cite{Mpi0}) have bearing on the geometry and arithmetic of algebraic varieties.   Analogs of the secondary surgery obstruction theory (even conjectural) involving the $\aone$-fundamental group are still mysterious.

\section{Birational sheaves and $\aone$-chain connectedness}
\label{s:unramified}
The goal of this section is to establish Theorem \ref{thm:propercharacterization}.  We have deferred the proof of this result here because, while it is in a sense elementary, the techniques used in the proof differ substantially from those (explicitly) used in previous sections.  We present the proof in outline here.

\begin{proof}[Proof of \textup{Theorem \ref{thm:propercharacterization}}]
We will construct a new sheaf $\pi_0^{b\aone}(X)$ (see Definition \ref{defn:birationalconnectedcomponents}) that is explicitly $\aone$-invariant in the sense of Definition \ref{defn:aoneinvariant} and has an additional birationality property.  The existence of the sheaf $\pi_0^{b\aone}(X)$, together with the proof of its $\aone$-invariance, follows from Theorem \ref{thm:definingthesheaf}; this point is both technically and notationally the most complicated part of the proof.  We will show in Proposition \ref{prop:connectednessfactorization} that there is a factorization
\[
\pi_0^{ch}(X) \longrightarrow \pi_0^{\aone}(X) \longrightarrow \pi_0^{b\aone}(X)
\]
inducing a bijection between sections of the first and last sheaves over finitely generated separable extensions $L/k$ that establishes the required result.
\end{proof}

\subsection{Birational and $\aone$-invariant sheaves}
To establish the propositions referenced in the proof of Theorem \ref{thm:propercharacterization} above, we need to introduce some terminology.  For $X \in \Sm_k$, we write $X^{(p)}$ for the set of codimension $p$ points of $X$.  We introduce a notion of birational and $\aone$-invariant sheaves inspired by (though independent from) the axiomatic framework developed in \cite[\S 1.1]{MField}; the superscript ``$b\aone$" in $\pi_0^{b\aone}$ indicates these properties.  We refer the reader to \cite[\S 2]{CTPurity}, \cite{RostChow} and \cite[\S 1.1]{MField} for some discussion of ideas leading up to the next definition.

\begin{defn}
\label{defn:birationalpresheaf}
Suppose ${\mathcal S}$ is a presheaf of sets on $\Sm_k$.  We will say that ${\mathcal S}$ is {\em birational} if it satisfies the following two properties.
\begin{itemize}
\item[(i)] For any $X \in \Sm_k$ having irreducible components $X_{\eta}$ $(\eta \in X^{(0)})$, the map
\[
{\mathcal S}(X) \longrightarrow \prod_{\eta \in X^{(0)}} {\mathcal S}(X_\eta)
\]
is a bijection.
\item[(ii)] For any $X \in \Sm_k$, and any open dense subscheme $U \subset X$, the restriction map ${\mathcal S}(X) \to {\mathcal S}(U)$ is a bijection.
\end{itemize}
A sheaf ${\mathcal S}$ of sets on $\Sm_k$ is called birational if the underlying presheaf of sets on $\Sm_k$ is birational.
\end{defn}

\begin{lem}
\label{lem:birationalpresheafisasheaf}
If ${\mathcal S}$ is a birational presheaf on $\Sm_k$, then ${\mathcal S}$ is a automatically a (Nisnevich) sheaf.
\end{lem}

\begin{proof}
Using \cite[\S 3 Proposition 1.4]{MV}, it suffices to show that given any Nisnevich distinguished square (see {\em loc. cit.} Definition 1.3), the induced square of sets obtained by applying ${\mathcal S}$ is Cartesian.  This property follows immediately from the definition of birationality.
\end{proof}

\begin{notation}
Let $\Shv^{b\aone}_k$ denote the full subcategory of $\Shv_k$ consisting of sheaves that are both birational in the sense of Definition \ref{defn:birationalpresheaf} and $\aone$-invariant in the sense of Definition \ref{defn:aoneinvariant}.
\end{notation}

\subsubsection*{An equivalence of categories}
Let $\F_k$ denote the category whose objects are finitely generated separable extension fields $L$ of $k$ and whose morphisms are field extensions.

\begin{notation}
We write $\fkset$ for the category of {\em covariant} functors from $\F_k$ to the category of sets.
\end{notation}

Suppose that $(L,\nu)$ is a pair consisting of a finitely generated separable field extension $L/k$ and a discrete valuation $\nu$ on $L$; write $\O_{\nu}$ for the corresponding discrete valuation ring and $\kappa_\nu$ for the associated residue field.  Given an object ${\mathcal S} \in \Shv^{b\aone}_k$, birationality of ${\mathcal S}$ implies that the map ${\mathcal S}(\Spec \O_\nu) \to {\mathcal S}(\Spec L)$ is a bijection.  The morphism $\O_{\nu} \to \kappa_{\nu}$ induces a map ${\mathcal S}(\Spec \O_{\nu}) \to {\mathcal S}(\Spec \kappa_\nu)$ that when composed with the inverse to the aforementioned bijection induces a specialization map ${\mathcal S}(L) \to {\mathcal S}(\kappa_\nu)$.

\begin{defn}
\label{defn:fkrset}
The category $\fkrset$ has as objects elements ${\mathcal S} \in \fkset$ together with the following additional structure:
\begin{itemize}
\item[({\bf R})] for each pair $(L,\nu)$ consisting of $L \in \F_k$ and a discrete valuation $\nu$ on $L$ with residue field $\kappa_{\nu}$ separable over $k$, a specialization (or residue) morphism $s_{\nu}: {\mathcal S}(L) \to {\mathcal S}(\kappa_{\nu})$.
\end{itemize}
Morphisms in the category $\fkrset$ are those natural transformations of functors preserving the data $({\bf R})$.
\end{defn}

With these definitions, evaluation on sections determines a functor
\begin{equation}
\label{eqn:birationalsheavesrestriction}
\Shv^{b\aone}_k \longrightarrow \fkrset
\end{equation}
that we refer to simply as {\em restriction}.  We will see momentarily that restriction is fully-faithful; let us first study its essential image.

\begin{defn}
\label{defn:birationalfkrset}
An object ${\mathcal S} \in \fkrset$ is called {\em sheaflike} if it has the following three properties.
\begin{itemize}
\item[({\bf A1})] Given pairs $(L,\nu)$ and $(L',\nu')$ consisting of finitely generated separable extensions $k \subset L \subset L'$ such that $\nu'$ restricts to a discrete valuation $\nu$ on $L$ with ramification index $1$, and both $\kappa_\nu$ and $\kappa_{\nu'}$ are separable over $k$, the following diagram commutes:
\[
\xymatrix{
{\mathcal S}(L) \ar[r]\ar[d] & {\mathcal S}(L') \ar[d] \\
{\mathcal S}(\kappa_\nu) \ar[r] & {\mathcal S}(\kappa_{\nu'}).
}
\]
\item[({\bf A2})] Given pairs $(L,\nu)$ and $(L',\nu')$ consisting of finitely generated separable extensions $k \subset L \subset L'$ such that $\nu'$ restricts to $0$ on $L$, and $\kappa_{\nu'}$ is separable over $k$, the composite map ${\mathcal S}(L) \to {\mathcal S}(L') \to {\mathcal S}(\kappa_{\nu})$ and the map ${\mathcal S}(L) \to {\mathcal S}(\kappa_{\nu})$ are equal.
\item[({\bf A3})] For any $X\in Sm_k$ irreducible with function field $F$,
for any point $z\in X^{(2)}$ with residue field $\kappa(z)$ separable over $k$, and for any point $y_0\in X^{(1)}$ with residue field $\kappa(y_0)$, such that $z\in \overline{y}_0$ and such that $\overline{y}_0\in Sm_k$, the composition
\[
{\mathcal S}(F) \longrightarrow {\mathcal S}(\kappa(y_0)) \longrightarrow {\mathcal S}(\kappa(z))
\]
is independent of the choice of $y_0$.
\end{itemize}
If furthermore, ${\mathcal S}$ has the following property, we will say that it is a {\em sheaflike and $\aone$-invariant} $\fkrset$.
\begin{itemize}
\item[({\bf A4})] For any $L \in \F_k$, the map ${\mathcal S}(L) \to {\mathcal S}(L(t))$ is a bijection.
\end{itemize}
\end{defn}

\begin{thm}
\label{thm:equivalenceofcategories}
The restriction functor
\[
\Shv^{b\aone}_k \longrightarrow \fkrset
\]
of \textup{Equation \ref{eqn:birationalsheavesrestriction}} is fully-faithful and has essential image the full subcategory of $\fkrset$ spanned by objects that are sheaflike and $\aone$-invariant.
\end{thm}

\begin{proof}
The full-faithfulness follows easily from birationality.  It is also straightforward to check that the essential image of restriction is contained in the full subcategory of $\fkrset$ spanned by objects that are sheaflike and $\aone$-invariant.  Indeed, $({\bf A1})$ follows immediately from the property of being a Nisnevich sheaf (via the distinguished square characterization), both $({\bf A2})$ and $({\bf A3})$ follow by choosing explicit smooth models for appropriate closed immersions.  We will thus show how to construct an explicit quasi-inverse functor.  Suppose given an object ${\mathcal S} \in \fkrset$ that is sheaflike and $\aone$-invariant.  We define a presheaf $\Stilde$ on $\Sm_k$ as follows.

{\em Step 1.} For an irreducible $U \in \Sm_k$ define $\Stilde(U) := {\mathcal S}(k(U))$.  Extend this assignment to all smooth schemes in the unique way to make Property (i) of Definition \ref{defn:birationalpresheaf} hold.

{\em Step 2.} Given a morphism $f: Y \to X$, we define a morphism $\Stilde(f): \Stilde(X) \to \Stilde(Y)$ as follows.  If $f$ is dominant, we define $\Stilde(f)$ to be the induced morphism ${\mathcal S}(f): {\mathcal S}(k(X)) \to {\mathcal S}(k(Y))$.  In the case $f: Y \to X$ is a closed immersion of smooth schemes, we proceed as follows.  Let ${\mathscr N}_f$ denote the normal bundle to the immersion.  Consider the diagram
\[
\xymatrix{
{\mathbb P}({\mathscr N}_f) \ar[r]^{\iota}\ar[d]^{\pi'} & {\sf Bl}_Z(X) \ar[d]^{\pi} \\
Y \ar[r]^{f} & X
}
\]
The top horizontal morphism is a codimension $1$ closed immersion, the left vertical arrow is a Zariski locally trivial morphism with projective space fibers, and the right vertical morphism is a proper birational morphism (in particular dominant).  Using Zariski local triviality of $\pi'$, together with the fact that ${\mathcal S}$ is $\aone$-invariant and $\Stilde$ only depends on the function field of its input, we observe that $\Stilde({\mathbb P}({\mathscr N}_f)) = \Stilde(Y)$.  Again using the fact that $\Stilde$ only depends on the function field of its input, we observe that $\Stilde(X_Z) = \Stilde(X)$.  The top vertical morphism is a codimension $1$ closed immersion, so we can define $\Stilde(X_Z) \to \Stilde({\mathbb P}({\mathscr N}_f))$ using the specialization morphism for the corresponding valuation.

{\em Step 3.} A general morphism $f: Y \to X$ can be factored as a closed immersion $Y \to Y \times X$ (the graph) followed by a projection (which is dominant).  We can then define $\Stilde(f)$ as the composite of these two morphisms.  We now need to check that the above constructions are actually compatible and define a functor.  These facts are checked in Lemma \ref{lem:birationalcompatibilities}.

Assuming these compatibilities, note that $\Stilde$ is by construction a birational and $\aone$-invariant presheaf.  Thus, Lemma \ref{lem:birationalpresheafisasheaf} shows $\Stilde$ actually defines a birational and $\aone$-invariant sheaf.  Furthermore, it is straightforward to check that this construction provides an inverse to restriction.
\end{proof}

\begin{lem}
\label{lem:birationalcompatibilities}
Continuing with notation as in \textup{Theorem \ref{thm:equivalenceofcategories}} (and its proof), we have the following two facts.
\begin{itemize}
\item Given a Cartesian square of the form
\[
\xymatrix{
Y' \ar[r]\ar[d] & Y \ar[d] \\
X' \ar[r] & X
}
\]
where the vertical morphisms are smooth and the horizontal morphisms are closed immersions, the diagram of sets obtained by applying $\Stilde$ commutes.
\item Given a sequence of closed immersions $Z \hookrightarrow Y \hookrightarrow X$, the triangle (of sets) obtained by applying $\Stilde$ commutes.
\end{itemize}
\end{lem}

\begin{proof}
The first point follows immediately from the functorial properties of blow-ups together with $\aone$-invariance and the fact that ${\mathcal S}$ is an $\fkrset$.

For the second point, we proceed as follows.  Let ${\mathscr N}$ denote the normal bundle to the closed immersion $Y \hookrightarrow X$.  Consider the morphism ${\sf Bl}_Y(X) \to X$.  Pulling back this morphism along the closed immersion $Z \hookrightarrow Y$ induces the projective bundle ${\mathbb P}({\mathscr N}|_{Z}) \to Z$.  Again, using the fact that $\Stilde$ only depends on the function field of its argument together with $\aone$-invariance, we conclude that $\Stilde(Z) = \Stilde({\mathbb P}({\mathscr N}|_{Z}))$ and $\Stilde(Y) = \Stilde({\mathbb P}({\mathscr N}))$.  This observation reduces us to proving the property when $Y$ is a codimension $1$ closed subscheme of $X$.  Repeating this process for $Z$, we can reduce to the case where $Z \subset Y \subset X$ is a sequence of codimension $1$ closed immersions.  In this case, using the fact that ${\mathcal S}$ is sheaflike, we can apply Property ({\bf A3}) to finish.
\end{proof}

\subsection{Birational connected components associated with proper schemes}
Via Theorem \ref{thm:equivalenceofcategories}, to define a birational and $\aone$-invariant sheaf, it suffices to specify an object of $\fkrset$ and check the properties listed in Definition \ref{defn:birationalfkrset}.  The main goal of this section is to prove the following result.


\begin{thm}
\label{thm:definingthesheaf}
Suppose $X$ is proper scheme having finite type over a field $k$.   By abuse of notation, use $X$ to denote the Nisnevich sheaf defined by the functor of points of $X$.  There exists a birational and $\aone$-invariant sheaf $\pi_0^{b\aone}(X)$ together with a morphism of sheaves $X \to \pi_0^{b\aone}(X)$ such that, for any $L \in \F_k$, the induced map $X(L) \to \pi_0^{b\aone}(X)(L)$ factors through a bijection
\[
X(L)/\smallsim \cong \pi_0^{b\aone}(X)(L)
\]
(recall \textup{Notation \ref{notation:aoneequivalenceclassesofpoints}}).
\end{thm}

\begin{rem}
One can show that the morphism $X \to \pi_0^{b\aone}(X)$ of the previous result is {\em initial} among morphisms of $X$ to birational $\aone$-invariant sheaves, but we will not need this fact.
\end{rem}

Suppose given a finitely generated separable extension $L$ of $k$, and a discrete valuation $\nu$ on $L$ with associated valuation ring $\O_{\nu}$ and residue field $\kappa_{\nu}$.  Assume $\kappa_{\nu}$ is separable over $k$.  Consider the $\fkset$ defined by
\[
L \longmapsto X(L)/\smallsim.
\]
The valuative criterion of properness implies that the map $X(\O_{\nu}) \to X(L)$ is a bijection.  Furthermore, the morphism $\O_{\nu} \to \kappa_{\nu}$ induces a morphism $X(\O_{\nu}) \to X(\kappa_{\nu})$.

\begin{lem}
\label{lem:definingfkrset}
If $X$ is a proper scheme having finite type over a field $k$, the composite morphism $X(L) = X(\O_{\nu}) \to X(\kappa_{\nu}) \to X(\kappa_{\nu})/\smallsim$ factors through a morphism
\[
s_{\nu}: X(L)/\smallsim \longrightarrow X(\kappa_{\nu})/\smallsim
\]
providing the data \textup{({\bf R})} of \textup{Definition \ref{defn:fkrset}}.  Denote by $\varpi_0^{b\aone}(X)$ the $\fkrset$ that associates with $L \in \F_k$ the set $X(L)/\smallsim$.
\end{lem}

\begin{prop}
\label{prop:definingbirationalfkrset}
The object $\varpi_0^{b\aone}(X)$ of $\fkrset$ from \textup{Lemma \ref{lem:definingfkrset}} is both sheaflike and $\aone$-invariant in the sense of \textup{Definition \ref{defn:birationalfkrset}}.
\end{prop}

The proofs of these results require some technical results about surfaces that we recall below.  Assuming for the moment the truth of this lemma and proposition, let us finish the proofs of the other results stated above.  Before proceeding, let us observe that, given a scheme $X$, one can define $\aone$-equivalence for morphisms $U \to X$ along the lines as Notation \ref{notation:aoneequivalenceclassesofpoints}.

\begin{defn}
\label{defn:birationalconnectedcomponents}
Suppose $X$ is proper scheme having finite type over a field $k$.  Write $\pi_0^{b\aone}(X)$ for the birational and $\aone$-invariant sheaf associated with the object $\varpi_0^{b\aone}(X)$ of $\fkrset$ via the equivalence of categories of Theorem \ref{thm:equivalenceofcategories}.  We refer to $\pi_0^{b\aone}(X)$ as the {\em sheaf of birational $\aone$-connected components of $X$}.
\end{defn}

\begin{proof}[Proof of Theorem \ref{thm:definingthesheaf}]
After Lemma \ref{lem:definingfkrset} and Proposition \ref{prop:definingbirationalfkrset}, it remains to construct the morphism of sheaves $X \to \pi_0^{b\aone}(X)$.   Let us construct a morphism of the underlying presheaves.  For any smooth $k$-scheme $U$, there is a well-defined map
\[
X(U) \longrightarrow X(U)/\smallsim \longrightarrow \pi_0^{b\aone}(X)(U)
\]
induced by functoriality of the construction $U \mapsto X(U)/\smallsim$ (``shrink $U$").  The associated morphism of sheaves is the required morphism.  Observe that the induced morphism $X \to \pi_0^{b\aone}(X)$ factors through the canonical epimorphism $X \to \pi_0^{ch}(X)$ of Lemma \ref{lem:chainepimorphism}.
\end{proof}

\begin{prop}
\label{prop:connectednessfactorization}
If $X$ is a proper scheme having finite type over a field $k$.   there is a canonical map $\pi_0^{\aone}(X) \to \pi_0^{b\aone}(X)$ such that the composite map
\[
\pi_0^{ch}(X) \longrightarrow \pi_0^{\aone}(X) \longrightarrow \pi_0^{b\aone}(X)
\]
(the first map is the canonical epimorphism of \textup{Lemma \ref{lem:chainepimorphism}}) induces a bijection on sections over finitely generated separable extension fields $L/k$.
\end{prop}

\begin{proof}
By Theorem \ref{thm:definingthesheaf}, we know that $\pi_0^{b\aone}(X)$ is a birational and $\aone$-invariant sheaf.  This fact implies the space $\pi_0^{b\aone}(X)$ is $\aone$-local by the equivalent conditions of \cite[\S 2 Proposition 3.19]{MV}.  By the universal property of $\aone$-localization, the canonical map $X \to \pi_0^{b\aone}(X)$ factors uniquely through the $\aone$-localization $L_{\aone}(X)$ of $X$ (see \cite[\S 2 Theorem 3.2]{MV} for this notation).  Thus, we obtain a factorization
\[
X \longrightarrow L_{\aone}(X) \longrightarrow \pi_0^{b\aone}(X).
\]
The second morphism induces for any $U \in \Sm_k$ a morphism $[U,L_{\aone}(X)]_s \to [U,\pi_0^{b\aone}(X)]_s$ functorial in $U$.  The set $[U,\pi_0^{b\aone}(X)]_s$ coincides with $\pi_0^{b\aone}(X)(U)$ by \cite[\S 2 Remark 1.14]{MV}.  Sheafifying for the Nisnevich topology produces a morphism of sheaves
\[
\pi_0^{\aone}(X) \longrightarrow \pi_0^{b\aone}(X).
\]
Finally, note that the morphism $X \to L_{\aone}(X)$ induces the morphism of sheaves $X \to \pi_0^{\aone}(X)$, and this morphism factors through $\pi_0^{ch}(X)$.
\end{proof}

\begin{conj}
For any proper scheme $X$ having finite type over a field $k$, the canonical morphism $\pi_0^{\aone}(X) \to \pi_0^{b\aone}(X)$ is an isomorphism, i.e., $\pi_0^{\aone}(X)$ is birational and $\aone$-invariant.
\end{conj}

\subsubsection*{Some technical results about surfaces}
We will recall a pair of results that will be used in the Proofs of Lemma \ref{lem:definingfkrset} and Proposition \ref{prop:definingbirationalfkrset}.  Throughout, $X$ denotes a proper scheme having finite type over a field $k$.

\begin{lem}
\label{lem:lipmanindeterminacy}
Let $S$ be an irreducible essentially smooth $k$-scheme of dimension $2$ with function field $F$, and suppose $\alpha \in X(F)$.  There exist finitely many closed points $z_1,\ldots,z_r$ in $S$, a proper birational morphism $f: S_\alpha \to S$, and $\beta \in X(S_\alpha)$ such that $S_{\alpha}$ is regular, $f$ is an isomorphism over the complement of $z_1,\ldots,z_r$ in $S$, and $\beta$ restricts to $\alpha$ under the induced map $X(S_{\alpha}) \to X(F)$.
\end{lem}

\begin{proof}
Recall that $S$ is essentially smooth if it can be written as limit of smooth schemes having finite type over $k$ with smooth affine bonding morphisms.  In any case, this result follows from a form of resolution of indeterminacy.  As $X$ is proper, $\alpha$ extends to a morphism $\Omega \to X$ where $\Omega \subset S$ is an open subscheme whose complement consists of finitely many closed points $z_1,\ldots,z_r$.  Denote by $S'$ the closure of the graph of the induced morphism $\Omega \to X$.  The projection $S' \to S$ is then proper and an isomorphism over $\Omega$.  By \cite[p. 101 Theorem and p. 155 (B)]{Lipman}, there exists a morphism $S_{\alpha} \to S'$ with $S_{\alpha}$ regular that is proper, birational and isomorphism on $\Omega$.  The composite map $S_{\alpha} \to X$ is the required morphism $\beta$.
\end{proof}

\begin{lem}
\label{lem:chainconnectedfibers}
Given a proper birational morphism $S' \to S$ between regular $k$-schemes of dimension $2$, for any point $s$ of $S$, the scheme-theoretic fiber $S'_{\kappa(s)} := S' \times_{S} \Spec \kappa(s)$ is a $\kappa(s)$-scheme that is $\aone$-chain connected.
\end{lem}

\begin{proof}
This result follows immediately from a strong factorization style result.  More precisely, \cite[Section II Theorem 1.15]{Lichtenbaum} states that a proper birational morphism between regular surfaces is a composition of blow-ups at closed points.  The fiber over any point $s$ in $S$ is a proper variety and one just needs to apply Proposition \ref{prop:blowupconnected} (note: that result is stated under they hypothesis $k$ is perfect, but that assumption is used only to establish that regular varieties over $k$ are in fact smooth).
\end{proof}

\subsubsection*{Proofs of Lemma \ref{lem:definingfkrset} and Proposition \ref{prop:definingbirationalfkrset}}
\begin{proof}[Proof of Lemma \ref{lem:definingfkrset}]
Suppose $L \in \F_k$, and suppose $\nu$ is a discrete valuation on $L$ with associated valuation ring $\O_{\nu}$, and residue field $\kappa_\nu$ assumed to be separable over $k$.  We want to prove that $X(L) \to X(\kappa_{\nu})/\smallsim$ factors through $\aone$-equivalence of points (see the discussion just prior to Notation \ref{notation:aoneequivalenceclassesofpoints}).  To do this, it suffices to prove that given a pair points $x_0,x_1 \in X(L)$, and a morphism $H: \aone_L \to X$ restricting to $x_0$ and $x_1$, the image of $x_0$ and $x_1$ under the map $X(L) \cong X(\O_{\nu}) \to X(\kappa_{\nu})$ are $\aone$-chain equivalent.

Consider the open immersion $\aone_L \to \aone_{\O_{\nu}}$.  Applying Lemma \ref{lem:lipmanindeterminacy}, we can assume that $H$ is defined on an open subscheme $U \subset \aone_{\O_{\nu}}$ whose complement is a collection of finitely many closed points $z_1,\ldots,z_r \in \aone_{\kappa_{\nu}}$, and extends to a morphism $\beta: S_{\alpha} \to X$, where $S_{\alpha}$ is smooth, and $S_{\alpha} \to \aone_{\O_{\nu}}$ is a proper birational morphism that is an isomorphism over $U$.

Again using the valuative criterion of properness, the $L$-points $x_i$ of $X$ uniquely extend to $\O_{\nu}$-points of $X$ ($i \in \{ 0,1\}$).  We view the sections $\O_{\nu} \to \aone_{\O_{\nu}}$ associated with these points as regular closed subschemes $C_i$ in $\aone_{\O_{\nu}}$.  The proper transforms $\tilde{C}_i$ of $C_i$ in $S_{\alpha}$ are closed subschemes that map properly and birationally onto $C_i$.  This observation implies that the maps $\tilde{C}_i \to C_i$ are in fact isomorphisms.  Thus, the closed points $z_i$ in $C_i$ lift uniquely to closed points $\tilde{z}_i$ in $\tilde{C}_i$.

Again for $i \in \{ 0,1 \}$, the image of $x_i$ under the map $X(L) \cong X(\O_{\nu}) \to X(\kappa_{\nu})$ is the composition
\[
\Spec \kappa_{\nu} \stackrel{x_i}{\longrightarrow} C_i \longrightarrow X
\]
where $C_i \to X$ is determined by $x_i \in X(L)$.  Using the observations of the last paragraph, this composite map factors as
\[
\Spec \kappa_{\nu} \stackrel{\tilde{z}_i}{\longrightarrow} \tilde{C}_i \longrightarrow S_{\alpha} \stackrel{\beta}{\longrightarrow} X;
\]
the morphisms $\tilde{C}_i \to S_{\alpha} \stackrel{\beta}{\to} X$ and $C_i \to S$ are equal because they agree on $\Spec L$.

The two points $z_i$ lie in $\aone_{\O_{\nu}}$ over the smooth curve $\aone_{\kappa_{\nu}}$ whose proper transform $D$ in $S_{\alpha}$ is isomorphic to $\aone_{\kappa_{\nu}}$.  By Lemma \ref{lem:chainconnectedfibers} the lifts of $z_i$ in $\tilde{C}_i$ and in $D$ are $\aone$-chain equivalent being in the same fiber of $S_{\alpha} \to \aone_{\O_{\nu}}$.  It follows that the two lifts $\tilde{z}_i$ are $\aone$-chain equivalent in $S_{\alpha}$, which implies that the images of $x_0$ and $x_1$ through $X(L) \cong X(\O_{\nu}) \to X(\kappa_{\nu})$ are $\aone$-chain equivalent.
\end{proof}

\begin{proof}[Proof of Proposition \ref{prop:definingbirationalfkrset}]
We just have to check conditions ({\bf A1})-({\bf A4}) of Definition \ref{defn:birationalfkrset}.  For ({\bf A1}), we observe that given an extension $L \subset L'$ in $\F_k$ together with discrete valuations $\nu$ and $\nu'$ satisfying the stated hypotheses, the diagram
\[
\xymatrix{
X(L) \ar[r]\ar[d] & X(L') \ar[d]\\
X(\kappa_{\nu}) \ar[r] & X(\kappa_{\nu'})
}
\]
commutes.  Lemma \ref{lem:definingfkrset} shows that the induced maps on $\aone$-equivalence classes of points also commute.  A similar argument can be used to establish ({\bf A2}).

For ({\bf A3}), we proceed along the same lines as the proof of Lemma \ref{lem:definingfkrset}.  Thus, let $z$ be a point of codimension $2$ on a smooth scheme $S$ with residue field $\kappa(z)$ separable over $k$.  Write $S_z$ for the corresponding local scheme.  Assume $y_0$ and $y_1$ are two points of codimension $1$ in $S_z$ whose closure $Y_i := \bar{y_i} \subset S_z$ is essentially smooth over $k$.  Let $F$ be the fraction field of $S_{z}$, and consider a morphism $\alpha: \Spec F \to X$.  By Lemma \ref{lem:lipmanindeterminacy}, we can find $S_{\alpha} \to S_z$ a proper birational morphism that is an isomorphism over $S_z \setminus z$ with $S_{\alpha}$ a smooth $k$-scheme, and $\beta: S_{\alpha} \to X$ inducing $\alpha$.

Since the $Y_i$ are assumed smooth, the morphisms $\tilde{Y_i} \to Y_i$ from the proper transform $\tilde{Y}_i$ is an isomorphism.  Let $\tilde{z}_i$ be the unique lifts of the closed point $z$ in $\tilde{Y}_i$.  The composite maps $X(F) \to X(\kappa(y_i)) \to X(\kappa(z))$ are equal to the corresponding composite maps $X(F) \to X(\kappa(\tilde{y}_i)) \to X(\kappa(\tilde{z}_i)) = X(\kappa(z))$.  Now, using Lemma \ref{lem:chainconnectedfibers}, the points $\tilde{z}_i$ both lie in the fiber over $z$ and are hence $\aone$-chain equivalent.  This observation implies that the images of the composite maps in $X(\kappa(z))/\smallsim$ are equal, which is what we wanted to show.

For ({\bf A4}), we need to show that the map
\[
X(L)/\smallsim \longrightarrow X(L(t))/\smallsim
\]
induced by the inclusion $L \subset L(t)$ is bijective.  Since $X$ is proper, for any $L \in \F_k$, an element of $X(L(t))$ determines a unique morphism $\pone_L \to X$.  From such a morphism, one obtains a morphism $\aone_L \to X$ by restriction.  Thus, the morphism $X(\aone_L) \to X(L(t))$ is surjective, and consequently $X(\aone_L) \to X(L(t))/\smallsim$ is surjective.  Given $H: \aone_L \to X$, $H$ is $\aone$-chain homotopic to $H_0: \aone_L \to \Spec L \to X$, where $\Spec L \to X$ is the restriction of $H$ to $0$.  However, the product map $\aone_L \times \aone_L \to \aone_L$ induces a chain homotopy in $X(\aone_L)$ between $H$ and $H_0$.  Picking a lift of an element in $X(L(t))$, we observe that the map $X(L)/\smallsim \to X(L(t))/\smallsim$ is surjective.

To prove injectivity, we proceed as follows.  It suffices to prove that given any field $L \in \F_k$, and two points $x_0,x_1 \in X(L)$, if the associated $L(t)$-points of $X$, which we denote $x_0'$ and $x_1'$,  are related by an elementary $\aone$-equivalence $H: \aone_{L(t)} \to X$, then $x_0$ and $x_1$ are themselves $\aone$-equivalent $L$-points.

There is an open dense subscheme $U \subset \aone_F$ such that $H$ is induced by a $k$-morphism $h: \aone_U \to X$, and the composite maps $U \stackrel{i}{\to} \aone_U \to X$ and $U \to \Spec(L) \stackrel{x_i}{\to} X$ are equal.  If $U$ admits an $L$-rational point (e.g., if $L$ is infinite), call it $y$, then composition with $y$ induces a morphism $h_y: \aone_L \to \aone_U \to X$ providing an elementary $\aone$-equivalence between $x_0$ and $x_1$.

If $U$ does not contain a rational point, then $h$ is defined on an open subscheme $\Omega \subset \aone_{\aone_L}$ whose complement is a finite collection of closed points, call them $z_j$.  Applying Lemma \ref{lem:lipmanindeterminacy}, we obtain a birational morphism $S_h \to \aone_{\aone_L}$ that is an isomorphism on $\Omega$, and a $k$-morphism $\tilde{h}: S_h \to X$ extending $h$.  Consider the copy of the affine line $D:= \aone_L \subset \aone_{\aone_F} = \aone_L \times \aone_L$ defined by $id \times 0$.  The intersection of $D$ with $\Omega$ is a dense open subscheme of $D$ (as $\aone_L$ has infinitely many closed points).  The proper transform of $D$ in $S_h$ is a closed curve that we denote by $\tilde{D}$.  Observe that the induced projection $\tilde{D} \to D$ is an isomorphism.

Now, the inclusion $\tilde{D} \hookrightarrow S_h$ is an elementary $\aone$-equivalence between the two $L$-points obtained by composing with $0$ and $1$; call these points $d_0$ and $d_1$.  Consider also the closed curves $Y_i$ of $\aone_{\aone_L}$ (also isomorphic to $\aone_L$) defined by the inclusions $0 \times id$ and $1 \times id$.  The proper transforms $\tilde{Y}_i \subset S_h$ of $Y_i$ are again isomorphic to $\aone_L$.  The $L$-points of $\tilde{Y}_i$ defined by inclusion at $0$, call them $\tilde{x}_i$, both lie in the same fiber of the morphism $S_h \to \aone_{\aone_L}$ as the $L$-point $d_i$.  By Lemma \ref{lem:chainconnectedfibers}, these fibers are $\aone$-chain connected, and thus the points $d_i$ and $\tilde{x}_i$ are $\aone$-equivalent $L$-points.  Finally, the composite map $\tilde{x}_i: \Spec L \to S_h \to X$ is equal to $x_i$ because the composite map $U \hookrightarrow \aone = \tilde{Y}_i \to S_h \to X$ is by construction equal to $x_i$; this provides the required $\aone$-equivalence.
\end{proof}

\appendix
\section{Aspects of homotopy theory for schemes}
\label{s:homotopyoverview}
For the convenience of the reader less familiar with notions of homotopy theory for schemes, we take this opportunity to provide an abridged presentation of some ideas from \cite{MV} supplemented by some more recent additions.  In truth there are several different ``space-level" categories giving rise to equivalent models for $\aone$-homotopy theory, yet we continue to focus on the one provided in \cite{MV}.  We emphasize the construction of $\aone$-homotopy categories as a left Bousfield localization, study some properties of $\aone$-homotopy theory in the \'etale topology, which are implicitly mentioned in \cite{MV}, and provide a simplified presentation of functoriality for $\aone$-homotopy categories with respect to change of topologies using the general theory of Quillen adjunctions.

Let $(\Sm_k)_{\et}$ and $(\Sm_k)_{Nis}$ denote the category of smooth $k$-schemes endowed with the structure of a site using either the \'etale or Nisnevich topology.  We denote by
\[
\alpha: (\Sm_k)_{\et} \longrightarrow (\Sm_k)_{Nis}
\]
the morphism of sites induced by the identity functor.  Recall this means that the presheaf pushforward $\alpha_*$ preserves sheaves, and admits a left adjoint, denoted $\alpha^*$ that preserves finite limits.  Throughout the text, we write $\Spc_k^{\tau}$ for the category of simplicial sheaves on $\Sm_k$ equipped with the topology $\tau$ where $\tau$ denotes either the Nisnevich or \'etale topologies.  Abusing notation, we write
\[
\begin{split}
\alpha_*&: \Spc_k^{\et} \longrightarrow \Spc_k, \text{ and }\\
\alpha^*&: \Spc_k \longrightarrow \Spc_k^{\et}
\end{split}
\]
for the functors induced by $\alpha_*$ and $\alpha^*$ at the level of sheaves.  All sites under consideration will have ``enough points." In particular, epimorphisms and isomorphisms can be detected stalkwise.  The Nisnevich and \'etale topology on $\Sm_k$ are both {\em subcanonical}, i.e., every representable presheaf is a sheaf.  Thus, the Yoneda functor composed with the functor of constant simplicial object induces a functor
\[
\Sm_k \to \Spc_{k}^{\tau}
\]
that is fully-faithful.  Using this functor, we identify $\Sm_k$ with the full subcategory of $\Spc_k^{\tau}$ consisting of representable objects.

\subsection{Simplicial homotopy categories}
Following modern terminology, a model category is a category that admits all limits and colimits indexed by small diagrams (i.e., it is complete and cocomplete) equipped with a model structure, i.e., three classes of morphisms called weak equivalences, cofibrations and fibrations satisfying the axioms of \cite[Definition 1.1.3]{Hovey}.  As mentioned above there are several different possible model structures on categories of sheaves, and we follow Joyal-Jardine.  All of these model structures arise from direct amalgamation of the homotopy theory of simplicial sets and sheaf theory.  (Note: the definition of model category that we use is the same as the one given in \cite[\S 2 Definition 0.1]{MV}).  We refer to \cite{GoerssJardine} for a systematic treatment of the homotopy theory of simplicial sets in the context of model categories, and \cite{Quillen}, \cite{Hovey} or \cite{Hirschhorn} for general results on model categories.

\begin{defn}
\label{defn:injectivemodelstructure}
Given a morphism $f: {\mathcal X} \to {\mathcal Y}$ in $\Spc_k^{\tau}$, we say that $f$ is a
\begin{itemize}
\item {\em simplicial weak equivalence}, if the morphisms of stalks (which are simplicial sets) induced by $f$ are weak equivalences of simplicial sets,
\item a {\em simplicial cofibration}, if $f$ is a monomorphism, and
\item a {\em simplicial fibration}, if $f$ has the right lifting property with respect to {\em acyclic simplicial cofibrations}, i.e., those morphisms that are simultaneously simplicial weak equivalences and simplicial cofibrations.
\end{itemize}
Write ${\mathbf W}_s$, ${\mathbf C}_s$ and ${\mathbf F}_s$ for the resulting classes of morphisms.
\end{defn}

Note that the category $\Spc_k^{\tau}$ is both complete and cocomplete, i.e., admits all limits and colimits indexed by small diagram categories.  In particular, this means that $\Spc_k^{\tau}$ admits an initial object (denoted $\emptyset$) and a final object (denoted $\ast$).  We write $\Spc_{k,\bullet}^{\tau}$ for the category of {\em pointed} $\tau$-simplicial sheaves, i.e., pairs $({\mathcal X},x)$ consisting of a simplicial sheaf and a morphism $x: \ast \to {\mathcal X}$.  The forgetful functor $\Spc_{k,\bullet}^{\tau} \to \Spc_k^{\tau}$ admits a left adjoint functor of ``adding a disjoint base-point."

We will say that a morphism of pointed simplicial sheaves is a simplicial weak equivalence (resp. simplicial cofibration or simplicial fibration) if the morphism of simplicial sheaves obtained by forgetting the base-points is a simplicial weak equivalence (resp. simplicial cofibration or simplicial fibration).  Recall that ${\mathcal X} \in \Spc_k^{\tau}$ is called {\em simplicially cofibrant} if the unique map $\emptyset \to {\mathcal X}$ is a simplicial cofibration and {\em simplicially fibrant} if the unique map ${\mathcal X} \to \ast$ is a simplicial fibration.  We use analogous terminology for pointed spaces.   The functor of adding a disjoint base-point preserves simplicial weak equivalences.

\begin{ex}
\label{ex:smoothschemesfibrant}
Every object in $\Spc_k^{\tau}$ is cofibrant.  It follows immediately from the definitions that any smooth $k$-scheme $X$ viewed as an object of $\Spc_k^{\tau}$ is simplicially fibrant.
\end{ex}

\begin{ex}[\u Cech objects]
\label{ex:cechobjects}
Suppose $f: {\mathcal X} \to {\mathcal Y}$ is an epimorphism in $\Spc^{\tau}_{k}$.  We let $\breve{C}(f)$ denote the simplicial sheaf whose $n$-th term is the $(n+1)$-fold fiber product of ${\mathcal X}$ with itself over ${\mathcal Y}$.  The morphism $f$ induces an augmentation map $\breve{C}(f) \to {\mathcal Y}$.  Since taking stalks commutes with formation of fiber products, one deduces that the map $\breve{C}(f) \to {\mathcal Y}$ is a simplicial weak equivalence by studying the corresponding situation for simplicial sets.  For applications to geometry, observe that if $u: U \to X$ a $\tau$-covering, then $u$ is an epimorphism of $\tau$-sheaves, and it follows that $\breve{C}(u) \to X$ is a simplicial weak equivalence in the $\tau$-topology.
\end{ex}

We also use the terminology {\em acyclic simplicial fibration} for any morphism in ${\mathbf W}_s \cap {\mathbf F}_s$.  The basic properties of these classes of morphisms is summarized in the following result.

\begin{thm}[Joyal-Jardine {\cite[Theorem 2.4 and Corollary 2.7]{Jardine}}]
\label{thm:injectivemodelstructure}
The three classes of morphisms $({\mathbf W}_s$,${\mathbf C}_s$,${\mathbf F}_s$) equip the category $\Spc_k^{\tau}$ (resp. $\Spc_{k,\bullet}^{\tau}$) with the structure of a model category.
\end{thm}

\begin{rem}
The terminology ``trivial" fibration or ``trivial" cofibration is also standard.  We have avoided the use of this terminology because of the numerous other uses of the term ``trivial" in the paper (e.g., torsors, $\aone$-$h$-cobordisms, etc.).
\end{rem}

One input in the proof of Theorem \ref{thm:injectivemodelstructure} is the existence of a fibrant resolution functor.  More precisely, one of the axioms of model categories states that given any morphism $f: {\mathcal X} \to {\mathcal Y}$ in $\Spc_k^{\tau}$, we can {\em functorially} factor $f$ as
\[
{\mathcal X} \stackrel{i}{\longrightarrow} {\mathcal Z} \stackrel{p}{\longrightarrow} {\mathcal Y}
\]
where either $i$ is an acyclic simplicial cofibration and $p$ is a simplicial fibration, or $i$ is a simplicial cofibration, and $p$ is an acyclic simplicial fibration.  We introduce some notation for the corresponding factorization for the morphism ${\mathcal X} \to \ast$.

\begin{defn}
A {\em fibrant resolution functor} on $\Spc_k^{\tau}$ is a pair $(Ex_{\tau},\theta_{\tau})$ consisting of an endo-functor $Ex_{\tau}$ and a natural transformation $\theta: Id \to Ex_{\tau}$ having the property that for any ${\mathcal X} \in \Spc_k^{\tau}$, there is an acyclic simplicial cofibration ${\mathcal X} \to Ex_{\tau}({\mathcal X})$ where $Ex_{\tau}({\mathcal X})$ is simplicially fibrant.
\end{defn}

In view of Theorem \ref{thm:injectivemodelstructure}, the general machinery of model categories shows that the localization $\Spc_k^{\tau} [{\mathbf W}_s^{-1}]$ exists.  We refer to such a category as a simplicial homotopy category.

\begin{defn}
The {\em $\tau$-simplicial homotopy category}, denoted $\hstau$, is defined by
\[
\hstau := \Spc_k^{\tau} [{\mathbf W}_s^{-1}],
\]
i.e., it is the homotopy category of $\Spc_k^{\tau}$ for the model structure of Definition \ref{defn:injectivemodelstructure}.  The {\em pointed $\tau$-simplicial homotopy category}, denoted $\hsptau$, is similarly defined by
\[
\hsptau := \Spc_{k,\bullet}^{\tau}[{\mathbf W}_s^{-1}].
\]
Given ${\mathcal X},{\mathcal Y} \in \Spc_k^{\tau}$ we write $[{\mathcal X},{\mathcal Y}]_{s,\tau}$ for the set of homomorphisms between ${\mathcal X}$ and ${\mathcal Y}$ in $\hstau$.   Given $({\mathcal X},x)$ and $({\mathcal Y},y) \in \Spc_{k,\bullet}^{\tau}$ we write $[({\mathcal X},x),({\mathcal Y},y)]_{s,\tau}$ for the set of morphisms between these two pointed spaces in $\hsptau$.  When $\tau$ denotes the Nisnevich topology, we drop it from notation.
\end{defn}

If we let $\Delta^i_s$ denote the constant sheaf represented by the simplicial $i$-simplex, and $\partial \Delta^i_s$ denote its boundary, the simplicial spheres are denoted $S^i_s$.  One defines the simplicial homotopy sheaves as follows.

\begin{defn}
For a pointed space $({\mathcal X},x)$ we write $\pi_i^{s,\tau}({\mathcal X},x)$ for the $\tau$-sheaf on ${\mathcal Sm}_k$ associated with the presheaf on ${\mathcal Sm}_k$ defined by $U \mapsto [\Sigma^i_s \wedge U_+,({\mathcal X},x)]_{s,\tau}$.
\end{defn}

\subsubsection*{Functoriality of simplicial homotopy categories: change of sites}
The next lemma is the basis for statements regarding functoriality of simplicial homotopy categories.

\begin{lem}
\label{lem:quillenfunctors}
The functor $\alpha^*$ preserves simplicial cofibrations, simplicial weak equivalences, and acyclic simplicial weak equivalences.  The functor $\alpha_*$ preserves simplicial fibrations, and acyclic simplicial fibrations.
\end{lem}

\begin{proof}
The fact that $\alpha^*$ preserves simplicial cofibrations, simplicial weak equivalences, and acyclic simplicial cofibrations follows immediately from the fact that all these conditions can be checked stalkwise.  The fact that $\alpha^*$ preserves simplicial fibrations and acyclic simplicial fibrations is an easy consequence of adjointness of $\alpha_*$ and $\alpha^*$ (see \cite[Proposition 7.2.18]{Hirschhorn}).
\end{proof}

We now apply the theory of Quillen adjunctions (see, e.g., \cite[\S 1.3]{Hovey}) to study functoriality properties for simplicial homotopy categories as the Grothendieck topology changes.  The resulting functoriality statements are simpler than the general functoriality properties for simplicial homotopy categories developed in \cite{MV}.  Set ${\mathbf R}\alpha_* = \alpha_* \circ Ex$.

\begin{cor}
The functors $\alpha_*$ and $\alpha^*$ induce an adjoint pair of functors
\[
\begin{split}
{\mathbf R}\alpha_*&: \hset \longrightarrow \hsnis \\
\alpha^*&: \hsnis \to \hset.
\end{split}
\]
\end{cor}

\begin{proof}
This follows from the general machinery of Quillen adjunctions.  For the purposes of notation, recall that Lemma \ref{lem:quillenfunctors} shows that $\alpha_*$ is a right Quillen functor and $\alpha^*$ is a left Quillen functor.  The functor ${\mathbf R}\alpha_*$ is the right derived functor of $\alpha_*$.  Since every object of $\Spc_k$ is already cofibrant, the left derived functor of $\alpha^*$ is precisely $\alpha^*$.
\end{proof}

\subsection{$\aone$-homotopy categories}
To obtain $\aone$-homotopy categories, one must perform a further categorical localization to ``invert the affine line."  We refer the reader to \cite[\S 3.3]{Hirschhorn} for a more detailed discussion of Bousfield localizations, which is the general context under which the definitions and results can be collected.

\begin{defn}
\label{defn:aonelocal}
An object ${\mathcal X} \in \Spc_k^{\tau}$ is called {\em $\tau$-$\aone$-local} (or just $\aone$-local if $\tau$ is clear from context) if, for any object ${\mathcal Y} \in \Spc_k^{\tau}$, the canonical map
\[
[{\mathcal Y},{\mathcal X}]_{\tau,s} \longrightarrow [{\mathcal Y} \times \aone,{\mathcal X}]_{\tau,s},
\]
induced by pullback along the projection ${\mathcal X} \times \aone \to {\mathcal X}$, is a bijection.
\end{defn}

We summarize some conditions that are equivalent characterizations of the property of being $\aone$-local.

\begin{defn}
A morphism $f: {\mathcal X} \to {\mathcal Y}$ in $\Spc_k^{\tau}$ is called
\begin{itemize}
\item a {\em $\tau$-$\aone$-weak equivalence} if for any $\aone$-local ${\mathcal Z} \in \Spc_k^{\tau}$ the map
\[
f^*: [{\mathcal Y},{\mathcal Z}]_{s,\tau} \longrightarrow [{\mathcal X},{\mathcal Z}]_{s,\tau}
\]
is a bijection.  If $\tau$ denotes the Nisnevich topology, we drop it from the notation.
\item an $\aone$-cofibration if it is a simplicial cofibration (i.e., a monomorphism), and
\item a $\tau$-$\aone$-fibration if it has the right lifting property with respect to $\aone$-acyclic cofibrations, i.e., those maps that are simultaneously $\aone$-cofibrations and $\tau$-$\aone$-weak equivalences.  Again, if $\tau$ denotes the Nisnevich topology, we drop it from notation.
\end{itemize}
A morphism $f$ in $\Spc_{k,\bullet}^{\tau}$ is said to be a $\tau$-$\aone$-weak equivalence (resp. $\aone$-cofibration, $\tau$-$\aone$-fibration) if the underlying morphism in $\Spc_k^{\tau}$ obtained by forgetting the base-point is a $\tau$-$\aone$-weak equivalence (resp. $\aone$-cofibration, $\tau$-$\aone$-fibration).  Again, if $\tau$ denotes the Nisnevich topology, we drop it from the notation.  We write ${\mathbf W}_{\aone}$ for the class of $\aone$-weak equivalences, ${\mathbf C}_{\aone}$ for the class of $\aone$-cofibrations, and ${\mathbf F}_{\aone}$ for the class of $\aone$-fibrations.
\end{defn}

\begin{lem}[{\cite[\S 2 Lemma 3.19]{MV}}]
\label{lem:aonelocalequivalentconditions}
If ${\mathcal X} \in \Spc_{k}^{\tau}$ is a simplicially fibrant space, then following conditions are equivalent.
\begin{itemize}
\item[i)] The space ${\mathcal X}$ is $\aone$-local.
\item[ii)] The space ${\mathcal X}$ is $\aone$-fibrant.
\item[iii)] For any smooth scheme $U$, the map of simplicial sets ${\mathcal X}(U) \to {\mathcal X}(U \times \aone)$ is a weak equivalence.
\end{itemize}
\end{lem}

\begin{ex}
\label{ex:aonelocalsmoothschemes}
As a smooth $k$-scheme $X$ is simplicially fibrant, one can check that $X(U)$ is a fibrant simplicial set for any smooth $k$-scheme $U$.  From Lemma \ref{lem:aonelocalequivalentconditions} one deduces that a smooth $k$-scheme $X$ is $\aone$-local if and only for any smooth $k$-scheme $U$ the map
\[
Hom_{\Sm_k}(U,X) \to Hom_{\Sm_k}(U \times \aone,X)
\]
induced by $U \times \aone \to U$ is a bijection.
\end{ex}

The fundamental existence result from \cite{MV} we use is summarized in the following result.

\begin{thm}[{\cite[\S 2 Theorem 3.2]{MV}}]
The three classes of morphisms $({\mathbf W}_{\aone}$,${\mathbf C}_{\aone}$,${\mathbf F}_{\aone}$) equip the category $\Spc_k^{\tau}$ (resp. $\Spc_{k,\bullet}^{\tau}$) with the structure of a model category.  Moreover, the resulting model category is the left Bousfield localization of $\Spc_k^{\tau}$ (resp. $\Spc_{k,\bullet}^{\tau}$) equipped with injective local model structure with respect to the class of (pointed) maps $\{{\mathcal X} \times \aone \to {\mathcal X}\}$ (in the pointed case, we view $\aone$ as a space pointed by $0$).  In particular, the identity functor $\Spc_k^{\tau} \to \Spc_k^{\tau}$ is a left and right Quillen functor.
\end{thm}

A space ${\mathcal X} \in \Spc_k^{\tau}$ is {\em $\aone$-fibrant} if the map ${\mathcal X} \to \ast$ is an $\aone$-fibration.  More generally, the morphism ${\mathcal X} \to \ast$ can be factored functorially as an $\aone$-acyclic cofibration followed by an $\aone$-fibration.   Thus, we obtain an $\aone$-fibrant resolution functor, i.e., a pair $(Ex_{\tau,\aone},\theta_{\tau,\aone})$ consisting of an endo-functor $Ex_{\tau,\aone}: \Spc_k^{\tau} \to \Spc_k^{\tau}$ and a natural transformation $\theta_{\tau,\aone}: Id \to Ex_{\tau,\aone}$ such that for any ${\mathcal X} \in \Spc_k^{\tau}$, the map ${\mathcal X} \to Ex_{\tau,\aone}({\mathcal X})$ is an $\aone$-acyclic cofibration with $Ex_{\tau,\aone}({\mathcal X})$ an $\aone$-fibrant space.

\begin{defn}
\label{defn:aonehomotopycategories}
We write $\ho{k}$ for the category $\Spc_k[{\mathbf W}_{\aone}^{-1}]$, $\hop{k}$ for the category $\Spc_{k,\bullet}[{\mathbf W}_{\aone}^{-1}]$, $\het{k}$ for the category $\Spc_k^{\et}[{\mathbf W}_{\aone}^{-1}]$, and $\hpet{k}$ for the category $\Spc_{k,\bullet}^{\et}[{\mathbf W}_{\aone}^{-1}]$.  We write
\[
L_{\aone}: \Spc_k^{\tau} \longrightarrow \Spc_k^{\tau}
\]
for the left derived functor of $Id$ and call it the {\em $\aone$-localization functor}.  Given two (\'etale) spaces ${\mathcal X},{\mathcal Y} \in \simpnis$ (resp. $\simpet$), we write $[{\mathcal X},{\mathcal Y}]_{\aone}$ (resp. $[{\mathcal X},{\mathcal Y}]_{\aone,\et}$) for the set of morphisms computed in $\ho{k}$ (resp. $\het{k}$).  Similarly, given two pointed (\'etale) spaces $({\mathcal X},x)$ and $({\mathcal Y},y)$, we write $[({\mathcal X},x),({\mathcal Y},y)]_{\aone}$ (resp. $[({\mathcal X},x),({\mathcal Y},y)]_{\aone,\et}$) for the set of morphisms computed in $\hop{k}$ (resp. $\hpet{k}$).
\end{defn}

\begin{defn}
For a pointed space $({\mathcal X},x)$ we write $\pi_i^{\aone,\tau}({\mathcal X},x)$ for the $\tau$-sheaf on ${\mathcal Sm}_k$ associated with the presheaf on ${\mathcal Sm}_k$ defined by $U \mapsto [\Sigma^i_s \wedge U_+,({\mathcal X},x)]_{\aone,\tau} = [\Sigma^i_s \wedge U_+,(L_{\aone}({\mathcal X}),x)]_{s,\tau}$.
\end{defn}

The next result summarizes the functoriality results we use; the proof is immediate.

\begin{lem}
\label{lem:aonefunctoriality}
The functor $\alpha^*$ preserves $\aone$-weak equivalences, $\aone$-cofibrations and $\aone$-acyclic cofibrations.  The functor ${\bf R}\alpha_*$ preserves $\aone$-local objects.
\end{lem}

\section{Notational postscript}
\label{s:notationalpostscript}
We felt that introducing so many (somewhat subtilely) different notions of connectedness and rationality in the preceding sections warranted inclusion of a summary of the notation and various implications.

\subsubsection*{Notions of connectedness}
\noindent$\pi_0^{\aone}(\cdot)$ - the sheaf of $\aone$-connected components; see Definition \ref{defn:sheavesofaoneconnectedcomponents}.\newline
\noindent$\pi_0^{\aone,\et}(\cdot)$ - the sheaf of \'etale $\aone$-connected components; see Definition \ref{defn:sheavesofaoneconnectedcomponents}.\newline
\noindent$\pi_0^{ch}(\cdot)$ - the sheaf of $\aone$-chain connected components; see Definition \ref{defn:svsingularconnected}.\newline
\noindent $\pi_0^{\et,ch}(\cdot)$ -the sheaf of \'etale $\aone$-chain connected components; see Definition \ref{defn:svsingularconnected}.\newline
\noindent $a_{\et}\pi_0^{\aone}(\cdot)$ - the \'etale sheafification of the functor $U \mapsto [U,{\mathcal X}]_{\aone}$; see Equation \ref{eqn:etalecomparison}.\newline
\noindent $\pi_0^{b\aone}(\cdot)$ - the sheaf of birational $\aone$-connected components; see Definition \ref{defn:birationalconnectedcomponents}.
\vskip 5pt
Assume $X \in \Sm_k$.  Lemma \ref{lem:chainepimorphism} provides the first two epimorphisms below, and Lemma \ref{lem:etalecomparison} provides the third.
\begin{equation*}
\begin{split}
\pi_0^{ch}(X) &\longrightarrow \pi_0^{\aone}(X), \\
\pi_0^{\et,ch}(X) &\longrightarrow \pi_0^{\et,\aone}(X), \text{ and }\\
a_{\et}\pi_0^{\aone}(X) &\longrightarrow \pi_0^{\et,\aone}(X).
\end{split}
\end{equation*}
We also introduced two geometric notions of connectedness: $\aone$-chain connectedness and weak $\aone$-chain connectedness (see Definition \ref{defn:chainconnected}), together with the notion of a variety covered by affine spaces (see Definition \ref{defn:combinatorial}).  The implications of various connectivity properties induced by the above epimorphisms, together with Lemma \ref{lem:combinatorialconnected} are summarized in the following diagram.
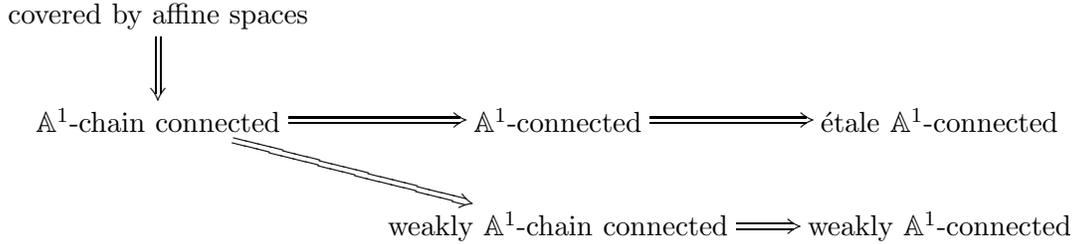
\begin{equation*}
\xymatrix{
\text{covered by affine spaces} \ar@{=>}[d] & & \\
 \aone\text{-chain connected} \ar@{=>}[r] \ar@{=>}[dr]& \aone\text{-connected} \ar@{=>}[r] & \text{\'etale }\aone\text{-connected} \\
& \text{weakly }\aone\text{-chain connected} \ar@{=>}[r]& \text{weakly }\aone\text{-connected}
}
\end{equation*}


\subsubsection*{Notions of near rationality}
The interrelationships between the various rationality properties (for smooth proper varieties over a field $k$) we considered are summarized in the following diagram.
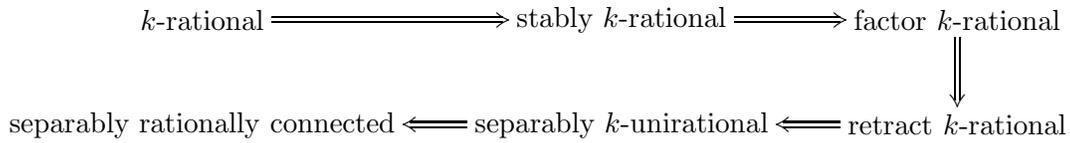
\begin{equation*}
\xymatrix{
k\text{-rational} \ar@{=>}[r] & \text{stably }k\text{-rational}  \ar@{=>}[r] & \text{factor }k\text{-rational} \ar@{=>}[d] \\
 \text{separably rationally connected}& \ar@{=>}[l]\text{separably }k\text{-unirational} & \ar@{=>}[l]\text{retract }k\text{-rational}
}
\end{equation*}
Definitions of all of these terms can be found in Definition \ref{defn:rationalitynotions}, though see also \cite[\S 1]{CTSRationalityFields} and \cite[IV.3.1-2]{Kollar} for more details.  Lemma \ref{lem:rationalitynotions} or straightforward consideration of the relevant definitions provide the implications in the above diagram.

\subsubsection*{Connections between $\aone$-connectivity and rationality properties}
Finally, we can connect the two diagrams above.  Again, let us restrict ourselves to considering only smooth {\em proper} varieties.  If $k$ is a perfect field, Corollary \ref{cor:separablyrationallyconnected} shows that
\begin{equation*}
\text{separably rationally connected} \Longrightarrow \text{weakly }\aone\text{-connected}.
\end{equation*}
Theorem \ref{thm:propercharacterization} (together with Proposition \ref{prop:aoneconnected}) shows that for arbitrary fields $k$,
\begin{equation*}
\aone\text{-chain connected} \Longleftrightarrow \aone\text{-connected},
\end{equation*}
and
\begin{equation*}
\text{weakly }\aone\text{-chain connected} \Longleftrightarrow \text{weakly }\aone\text{-connected}.
\end{equation*}
If furthermore, $k$ has characteristic $0$, Theorem \ref{thm:stablyrational} shows that
\begin{equation*}
\text{retract }k\text{-rational} \Longrightarrow \aone\text{-chain connected}.
\end{equation*}

\begin{footnotesize}
\bibliographystyle{alpha}
\bibliography{hcobordisms}
\end{footnotesize}
\end{document}